\documentclass[12pt]{article}

\usepackage{amsmath,amssymb,amsfonts,amsthm,bbm}
\usepackage{color,graphicx}
\RequirePackage[OT1]{fontenc}
\RequirePackage[numbers]{natbib}
\usepackage[colorinlistoftodos]{todonotes}

\RequirePackage{amsthm,amsmath,amsfonts,amssymb,amscd,mathrsfs}
\RequirePackage[numbers]{natbib}
\RequirePackage[colorlinks,citecolor=blue,urlcolor=blue]{hyperref}

\graphicspath{ {./figures/} }

\newtheorem{theorem}{Theorem}

\newtheorem{corollary}{Corollary}
\newtheorem{lemma}{Lemma}
\newtheorem{proposition}{Proposition}
\newtheorem{condition}{Condition}

\newtheorem{remark}{Remark}

\usepackage[english]{babel}
\usepackage[utf8]{inputenc}
\usepackage{amsmath,amsthm, amssymb, latexsym, color}

\newcommand{\R}{\mathbb{R}}

\newcommand{\G}{\mathscr{G}}
\newcommand{\sS}{\mathscr{S}}

\newcommand{\Z}{\mathbb{Z}}
\newcommand{\N}{\mathbb{N}}

\newenvironment{enumerate*}%

\makeatletter  
\@ifundefined{@currsizeindex} 
  {\RequirePackage{relsize}\let\dolarger\relsize} 
  {\def\dolarger#1{\larger[#1]}} 

\newcommand*\@@bigtimes[2]{\vphantom{\prod} 
  \vcenter{\hbox{\dolarger{4}$\m@th#1\mkern-2mu\times\mkern-2mu$}}} 
\newcommand*\bigtimes{\mathop{\mathpalette\@@bigtimes\relax}\displaylimits} 
\makeatother

\begin{document}

\title{\textsc{Bernstein--von Mises theorems for \\ time evolution equations}}

\author{\\ {Richard Nickl} \\ \\  \textsc{University of Cambridge}}

\maketitle

\begin{abstract}
We consider a class of infinite-dimensional dynamical systems driven by non-linear parabolic partial differential equations with initial condition $\theta$ modelled by a Gaussian process `prior' probability measure. Given discrete samples of the   state of the system evolving in space-time, one obtains updated `posterior' measures on a function space containing all possible trajectories. We give a general set of conditions under which these non-Gaussian posterior distributions are approximated, in Wasserstein distance for the supremum-norm metric, by the law of a Gaussian random function. We demonstrate the applicability of our results to periodic non-linear reaction diffusion equations 
\begin{align*}
\frac{\partial}{\partial t} u - \Delta u &= f(u) \\
u(0) &= \theta
\end{align*}
where $f$ is any smooth and compactly supported reaction function. In this case the limiting Gaussian measure can be characterised  as the solution of a time-dependent Schr\"odinger equation with `rough' Gaussian initial conditions whose covariance operator we describe.
\end{abstract}

\tableofcontents

\section{Introduction}

\subsection{Bayesian inference in dynamical systems}

Denote by $L^2(\Omega)$ the Hilbert space of square-integrable vector fields over a bounded domain $\Omega \subset \R^d$. For  initial condition $\theta = u(0)$, consider states $u_\theta(t)$ at times $t>0$ of a dynamical system evolving in $L^2(\Omega)$. Examples we have in mind include the \textit{Navier-Stokes}, \textit{reaction-diffusion} or \textit{McKean-Vlasov} systems \cite{T97, R01, CF88, S91}, where $u_\theta$ models the velocity field of an incompressible fluid, the concentration of a chemical substance, or the distribution of interacting particles, respectively. The infinitesimal dynamics of such systems are described by a non-linear partial differential equation (PDE)
\begin{equation}\label{genpde}
\frac{\partial u}{\partial t} = \Delta u  +F(u)
\end{equation}
where $\Delta$ is the Laplacian and $F$ is a functional modelling the underlying non-linearity.

\smallskip

In scientific applications, the initial condition $\theta$ is uncertain and moreover, physical observations are necessarily discrete and include statistical error. For instance, following the paradigm of probabilistic numerics \cite{D88, CDRS09, S10, BOGOS19}, an experimenter may sample the state $u_\theta$ of the system at positions $(t_i, \omega_i)_{i=1}^N$ drawn uniformly and independently at random from the time-space cylinder $[0,T] \times \Omega$, with $T>0$ a fixed time horizon. One thus observes the vector $Z^{(N)}=(Y_i,t_i, \omega_i)_{i=1}^N$ from the regression equation
\begin{equation}\label{data}
Y_i = u_\theta(t_i, \omega_i) + \varepsilon_i, ~i=1, \dots, N,~N \in \mathbb N,
\end{equation}
where the $\varepsilon_i \sim \mathcal N(0,I)$ are independent and identically distributed (iid) Gaussian random noise. The infinite product probability measure describing the law of the sequence $Z^{(\infty)}$ arising from initial condition $\theta$ will be denoted by $P_{\theta}^\mathbb N$.

\smallskip

The uncertainty about the state $\theta$ of the system before measurements are taken can be updated to a \textit{posterior distribution} on the $u_\theta$'s given $Z^{(N)}$, by applying the rules of conditional probability. This follows seminal ideas of Laplace (Chapter VI in \cite{L1812}) and nowadays is called \textit{Bayesian inference} \cite{GV17} or sometimes more specifically -- when considering dynamical systems in applied sciences -- `data assimilation' or `filtering' \cite{EVvL22, RC15, LSZ15, MH12, K03}, where (\ref{data}) specifically corresponds to `deterministic' data assimilation without `model error'. In infinite-dimensional settings where $\theta$ is a function, a typical approach is to represent the uncertainty about the initial condition by a Gaussian random field $(\theta(x): x \in \Omega)$ whose mean function is based on past knowledge (or in absence of it, equals zero) and whose covariance is an inverse power of the Laplace operator on $\Omega$. If we denote by $\Pi$ the induced `prior' probability distribution on $\theta$, the updating step (i.e., Bayes' theorem) generates the posterior Gibbs' probability distribution for $\theta|Z^{(N)}$ of the form
 \begin{equation} \label{post}
d\Pi(\theta|Z^{(N)}) \propto \exp \Big\{-\frac{1}{2} \sum_{i=1}^N|Y_i-u_\theta(t_i, \omega_i)|^2 \Big\} d\Pi(\theta),~~\theta \in L^2(\Omega).
\end{equation}
We also obtain an update for the whole dynamical system $u_\theta(t)$ whose marginal time distributions on $L^2(\Omega)$ are given by the image measures
\begin{equation} \label{imagemeasure}
\hat \Pi_{t,N} = Law (u_\theta(t)), ~\theta \sim \Pi(\cdot|Z^{(N)}), ~t \ge 0.
\end{equation}
For the last observation times $t_{(N)} = \max_{i \le N} t_i \in (0,T],$ the preceding laws are sometimes called \textit{filtering distributions}; for $t>t_{(N)}$ they can be used to forecast future states of the system, while for $t<t_{(N)}$ they are referred to as `smoothing distributions', see, e.g., p.xii in \cite{LSZ15} or p.173 in \cite{RC15} for this terminology from the data assimilation literature. Various iterative algorithms exist that aim to compute these posterior measures by numerical (e.g., MCMC, or particle filter) methods \cite{S10, CRSW13, HSV14, RC15, LSZ15, CHSV24, GW24}. The resulting inferences are widely used in scientific application areas, but statistical theory for the performance of such methods in non-linear settings remains elusive.

\smallskip

The distributions $\hat \Pi_{t,N}$ are random probability measures in function space that involve the non-linearity of the underlying evolution operator both in the Gibbs' measure (\ref{post}) and in the push-forward (\ref{imagemeasure}), and have no simple closed form representation.  A first sanity check for this Bayesian approach would be to establish its `posterior consistency' \cite{GV17}; namely that $\hat \Pi_{t,N}$ converges in an appropriate sense to Dirac measure $\delta_{u_{\theta_0}}$ at the state $u_{\theta_0}(t)$ arising from `ground truth' initial condition $\theta_0$, at least as the signal to noise ratio increases (i.e., as sample size $N \to \infty$), and with high $P_{\theta_0}^\mathbb N$-probability. For the $2d$-Navier-Stokes model such a result was proved recently in \cite{NT23} -- see Remark \ref{template} for more discussion.

In this article we aim to further advance our understanding of the statistical behaviour of the random measures $\hat \Pi_{t,N}$ for large sample sizes $N \to \infty$. Specifically, following another idea of Laplace, we investigate whether the non-Gaussian posterior measures $\hat \Pi_{t,N}$ are perhaps \textit{approximately} Gaussian, a phenomenon that goes by the name of Bernstein-von Mises (BvM) theorem in the mathematical statistics literature. Such results hold in `regular' finite-dimensional statistical models  \cite{vdvaart1998} but are more intricate in high- or infinite-dimensional settings \cite{F99, CN13, CN14, GV17, NP23}. We will devise a template tailored to dynamical systems of parabolic type (\ref{genpde}) that allows to show that in a sense to be made precise, uniformly in fixed time windows, and with high $P_{\theta_0}^\mathbb N$-probability,
\begin{equation} \label{pitg}
\hat \Pi_{t,N} \approx \mathcal N\big(u_{\tilde \theta_N}, \mathcal C_t/N\big), t>0,
\end{equation}
where $\tilde \theta_N$ is the posterior mean vector and where the limit law $\mathcal N(0,\mathcal C_t)$ can be characterised by the solution at time $t$ of a \textit{linear stochastic} PDE with certain `rough' Gaussian initial conditions.  We demonstrate how the theory works for a concrete nonlinear reaction-diffusion system and discuss structurally similar non-linear parabolic PDEs (\ref{genpde}) in Remark \ref{template}.

These Gaussian approximations will be shown to hold in the uniform norm over $\Omega$, and at the  convergence rate $1/\sqrt N$. This is possible by first proving a Gaussian approximation in weak `Schwartz'-type topologies for the initial conditions $\theta$ (partly following ideas from \cite{CN13, CN14, N20} in different models), and then using strong smoothing properties of the semigroup underlying (\ref{genpde}) at positive times $t>0$. Some discussion of the last point and its relationship to non-existence results \cite{F99} for infinite-dimensional Bernstein-von Mises theorems can be found in Remark \ref{clever}.

\subsection{Asymptotics for posterior measures in reaction-diffusion equations}

We let $\Omega = [0,1]^d$ denote the $d$-dimensional torus and consider periodic solutions $u=u_\theta$ to the  PDE
\begin{align} \label{evol} 
\frac{\partial}{\partial t} u &= \Delta u + f(u), ~~\text{ on } (0,\infty) \times \Omega, \\
u(0)&=\theta,~~\text{ on } \Omega, \notag
\end{align}
where $\Delta$ is the Laplacian and $f$ is a given non-linear `reaction' term that is smooth (infinitely differentiable), compactly supported and time-independent, acting on $u$ by composition $f(u)=f \circ u$ point-wise, $(f \circ u)(t,\omega) = f(u(t,\omega)), t>0, \omega \in \Omega$. The unique solutions of this PDE induce a dynamical system $(u_\theta(t,\cdot): t>0)$ evolving in $L^2(\Omega)$, see \cite{T97, R01}. In applications the initial condition $\theta$ may represent the distribution of some substance and (\ref{evol}) describes its evolution over time where the non-linear reaction term models the amount of the substance that is created or destroyed in dependence $f(u(t))$ of the current state $u(t)$. For the proofs we will assume the dimension of the state space $\Omega=[0,1]^d$ to be $d \le 3$ and we also only consider scalar $u$ taking values in $\R$ rather than general vector fields $u: \Omega \to \R^d$. Extensions to $d \ge 4$ are mostly of a technical nature (involving Schauder theory as in Sec.~5.2 of \cite{M89}, rather than energy estimates) while coupled systems of equations could be dealt with by mostly notational changes as long as  $f =\nabla F$ is a smooth and compactly supported gradient vector field.

\smallskip

We model the uncertainty about the initial condition $\theta$ of the system by a $\gamma$-regular Gaussian random field $(\theta(x): x \in \Omega)$ whose prior $Law(\theta)$ on the space $L^2_0(\Omega) = L^2(\Omega) \cap \{h: \int_\Omega h=0\}$ arises from
\begin{equation}\label{prior}
\theta \sim \Pi = \mathcal N(0, \rho^2(-\Delta)^{-\gamma}),~~ \gamma>1+d/2,~ \rho>0,
\end{equation}
see Condition \ref{priorgen} for details. The choice of $\rho$ provides necessary prior regularisation to prevent the nonlinearities inherent in the posterior Gibbs measure (\ref{post}) to behave erratically at `low temperatures' $N \to \infty$. A way that achieves this exhibited in \cite{MNP21} takes the form
\begin{equation}\label{ronacher}
\rho=\rho_N=1/(\sqrt N \delta_N), \text{ where } \delta_N \simeq N^{-\gamma/(2\gamma+d)}.
\end{equation}
The posterior distribution $\Pi(\theta|Z^{(N)})$ and its pushforward $\hat \Pi_{t,N}$ then arise from data (\ref{data}) as in (\ref{post}) and (\ref{imagemeasure}), with dynamical system described by the PDE (\ref{evol}) for any fixed $f$ that is infinitely differentiable and compactly supported. 

\smallskip

Our main Theorem \ref{ganzwien} will show that the (non-Gaussian) posterior measures $\hat \Pi_{t,N}$ are approximated by a Gaussian law, as announced in (\ref{pitg}). This approximation will hold as $N \to \infty$, and in $P_{\theta_0}^\mathbb N$-probability, with ground truth initial condition $\theta_0$ that generated the data (\ref{data}). The most delicate step in the identification of the limit distribution is the  construction of a mean zero Gaussian random field  $(\vartheta(x): x \in \Omega)$ whose covariance is related to an appropriate inverse of the underlying `Fisher information' operator of the statistical model (see, \cite{vdvaart1998} or Ch.3 in \cite{N23} for more on this notion). We will show that this process indeed exists for the reaction-diffusion system (\ref{evol}) and that the law $\mathcal N_{\theta_0}$ it induces defines a Gaussian Borel measure in the negative Sobolev space $H^{-a}$ whenever $a>1+d/2$. Then let $u_{\theta_0}$ be the solution of (\ref{evol}), write $f'=df/dx$, and consider the Gaussian process $U$ over $(0,T] \times \Omega$ obtained from the unique weak solution of the following \textit{linear} time-dependent Schr\"odinger equation with random initial condition:
\begin{align} \label{linshow}
\frac{\partial}{\partial t} U(t,\cdot) - \Delta U(t, \cdot) - f'(u_{\theta_0}(t,\cdot)) U(t, \cdot) &= 0 \text{ on } (0,\infty) \times \Omega
\notag \\
U(0, \cdot) &= \vartheta \sim \mathcal N_{\theta_0}.
\end{align}
Even though the initial condition $\vartheta$ is not point-wise defined as a function (almost surely), our proofs will imply that the weak solutions $U(t, \cdot)$ exist and almost surely define continuous functions on $\Omega$ for $t >0$. This is related to the time-smoothing properties of the parabolic solution operator underlying (\ref{linshow}) and the fact that the null space of the Schr\"odinger operator $\Delta + f'(\theta_0)$ governing the `explosion' at time $t=0$ will be seen to be finite-dimensional, see also Remark \ref{clever}. 

\smallskip

We will assume that the true initial condition $\theta_0$ (but not the prior model) is smooth to simplify the statement of the following result. But this restriction is not necessary and can be weakened to sufficient Sobolev-regularity of $\theta_0$. Fix any time window $0<t_{min} < t_{max}<\infty$ and define the (real) Banach space 
\begin{equation} \label{calc}
\mathscr C := C\big([t_{\min}, t_{\max}], C(\Omega)\big),~~\|v\|_{\mathscr C} := \sup_{t\in[t_{\min}, t_{\max}], x \in \Omega} |v(t,x)|,
\end{equation}
of continuous maps on $[t_{\min}, t_{\max}] \times \Omega$. Denote by $\mathscr W_1$ (equal to $W_{1, \mathscr C}$ in (\ref{wassdef}) below) the $1$-Wasserstein distance on the space of probability measures on $\mathscr C$. The Gaussian law $\mathcal N_{\theta_0}$ featuring in the following theorem is explicitly constructed in Proposition \ref{limmeas} below.

\begin{theorem}\label{ganzwien}
Let $\mu_{N}= \mu(\cdot|Z^{(N)})$ be the conditional law on $\mathscr C$ of the stochastic process $$\Big\{\sqrt N (u_\theta(t,x) - u_{\tilde \theta_N}(t,x))|Z^{(N)}: t \in [t_{\min}, t_{\max}], x \in \Omega\Big\}$$ where $\theta \sim \Pi(\cdot|Z^{(N)})$ arises from posterior (\ref{post}) with data (\ref{data}) in the reaction-diffusion system (\ref{evol}) for $f: \R \to \R$ that is infinitely differentiable and compactly supported, prior $\Pi=\Pi_N$ in (\ref{prior}) for $\rho$ as in (\ref{ronacher}), integer $\gamma>2+3d$, $d \le 3$, and where $\tilde \theta_N = E^\Pi[\theta|Z^{(N)}]$ is the posterior mean in $L_0^2(\Omega)$. Denote by $\mu$ the law in $\mathscr C$ of the Gaussian random function arising from the unique weak solution $U$ to the PDE (\ref{linshow}) with initial condition $\vartheta \sim \mathcal N_{\theta_0}$ and smooth $\theta_0 \in L^2_0$. Then we have as $N \to \infty$
$$\mathscr W_1(\mu_{N}, \mu) \to^{P_{\theta_0}^\mathbb N} 0 ~~\text{ as well as }~~\sqrt N (u_{\tilde \theta_N} - u_{\theta_0}) \to_{\mathscr C}^d \mu.$$
\end{theorem} 

From convergence in law $\to ^d$ in the space $\mathscr C$ one deduces in particular that $\sqrt N(u_{\tilde \theta_N}-u_{\theta_0})$ is uniformly tight in $\mathscr C$, so the theorem implies that the updated posterior mean estimates $u_{\tilde \theta_N}(t)$ are concentrated near the `true'  dynamical system $u_{\theta_0}(t)$, and that their uniform $\|\cdot\|_{\mathscr C}$ fluctuations at a $\sqrt N$ scale are asymptotically Gaussian.

\smallskip

Theorem \ref{ganzwien} has various important applications that we only outline briefly here. One is to uncertainty quantification and the frequentist validity of posterior credible bands which are random subsets $C_N \subset \mathscr C$ of posterior probability $1-\alpha$ for some fixed $0<\alpha<1$.  Theorem \ref{ganzwien} and arguments as in the proof of Proposition 7.3.21 in \cite{GN16} can be combined to show that appropriate $C_N$ are of diameter $O_{P_{\theta_0}^\mathbb N}(1/\sqrt N)$ and satisfy 
\begin{equation} \label{cbandlim}
P_{\theta_0}^\mathbb N\big((u_{\theta_0}(t,x): t \in [t_{\min}, t_{\max}], x \in \Omega) \in C_N\big) \to 1-\alpha,~\text{ as } N \to \infty,
\end{equation}
see Corollary 1.3 in \cite{KN25} for how to prove such a result in a related setting.

\smallskip

Another interesting application of our results concerns guarantees for posterior computation: When Gaussian approximation theorems such as Theorem \ref{ganzwien} hold, they furnish a strategy to prove the existence of polynomial-time computational algorithms based on MCMC \cite{BC09, HSV14, NW24, N23}, overcoming potential computational hardness barriers \cite{BMNW23}. We note that the hypothesis of a Gaussian approximation is also used in recent work on convergence guarantees for Kalman-type filters in non-linear settings \cite{CHSV24}.

\subsection{Notation and preliminaries}

For real numbers $a,b$, we write $a \lesssim b$ whenever $a \le Cb$ for some fixed constant $C>0$, and $a \simeq b$ if $a \lesssim b, b \lesssim a$. We write $Z \sim \mu$ when a random variable $Z$ has law $\mu$, write $\to^P$ for convergence in probability, $\to^d$ for convergence in distribution, and use the standard $O_P, o_P$ notation for stochastic orders of magnitude, see \cite{vdvaart1998}.

For $1 \le p \le \infty$ and $(\mathscr Z, \mathcal A, \mu)$ a measure space, the $L^p(\mu)$-spaces are defined in the usual way as all measurable functions $H$ on $\mathscr Z$ such that $|H|^p$ is $\mu$-integrable, or essentially bounded if $p=\infty$. When $\mathscr Z$ is a subset of $\R^d$ we take $\mu$ to be Lebesgue measure on the Borel sets, unless specified otherwise. For $X$ a normed linear space, we define the function space $L^p([0,T],X)$ of measurable maps $H$ from $[0,T] \to X$ such that $\|H(t)\|_{X}$ lies in $L^p([0,T])$. We also define the space $C([0,T], X)$ of continuous maps from $[0,T] \to X$ normed by the supremum norm $\sup_{0<t<T}\|H(t)\|_X$ as well as the spaces of weakly differentiable maps
\begin{equation*}
C^1([0,T], X) = \Big\{H: \|H\|_{C^1([0,T], X)}:=\sup_{0<t<T}\|H(t)\|_{X} +  \textrm{ess}\sup_{0<t<T}\|H'(t)\|_{X}<\infty \Big\},
\end{equation*}
\begin{equation*}
H^1([0,T], X) = \Big\{H: \|H\|^2_{H^1([0,T], X)}:=\int_0^T\|H(t)\|^2_{X}dt + \int_0^T\|H'(t)\|^2_{X}dt<\infty \Big\}.
\end{equation*}
For $\Omega = [0,1]^d$ the $d$-dimensional torus (with opposite points identified), $C(\Omega)$ denotes the Banach space of bounded continuous functions defined on $\Omega$, equipped with supremum norm $\|\cdot\|_\infty$, and $C^\gamma(\Omega)$ denotes the usual spaces of functions defined on $\Omega$ whose partial derivatives up to order $\gamma \in \mathbb N$ are bounded functions; for $\gamma \notin \mathbb N$ these are the H\"older spaces, and $C^\infty(\Omega) = \cap_{\gamma >0} C^\gamma(\Omega)$ denotes the set of smooth (infinitely-differentiable) functions defined on $\Omega$. We also introduce the spaces
\begin{equation} \label{notsosmooth}
C^{1,\infty}(\Omega) \equiv \cap_{b>0} C^1([0,T], C^b(\Omega)).
\end{equation}
Further define the Sobolev spaces $H^\gamma(\Omega), \gamma \ge 0,$ of functions whose weak partial derivatives up to order $\gamma$ lie in $L^2(\Omega)$. The norms $\|\cdot\|_{C^\gamma}, \|\cdot\|_{H^\gamma}$ on $C^\gamma, H^\gamma, \gamma \in \mathbb N,$ are then given by $\|H\| + \sum_{|\alpha|=\gamma} \|D^\alpha H\|,$ where $D^\alpha$ is the (weak) partial differential operator for multi-index $\alpha$ and the norms $\|\cdot\|$ equal $\|\cdot\|_\infty$ or $\|\cdot\|_{L^2}$, respectively. 

The Sobolev spaces $H^\gamma$ carry equivalent sequence space norms 
\begin{equation}\label{seqnorm}
\|h\|_{h^\gamma}^2 = \sum_{j \ge 0} (1+\lambda_j)^{\gamma} |\langle h, e_j \rangle_{L^2}|^2,
\end{equation}
where $\langle g, h \rangle_{L^2}= \int_\Omega g \bar h$ is the inner product of $L^2(\Omega)$, and the $e_j$ are the $L^2(\Omega)$-orthonormal eigenfunctions of the periodic Laplacian $\Delta$ for (negative) eigenvalues $\lambda_j$. Concretely we have $e_0=1, \lambda_0=0$, and
\begin{equation} \label{weyl}
\Delta e_j = - \lambda_j e_j, ~j \ge 1,~~ 0<\lambda_{j} \le \lambda_{j+1} \simeq j^{2/d} ~\text{ as }~j \to \infty,
\end{equation}
where
\begin{equation}\label{eftrig}
e_j(x) \propto e^{2\pi i k_j \cdot x},~x \in \Omega,~k_j \in \Z^d,
\end{equation}
and where $j=0, 1, 2, \dots$ is an enumeration of the integer lattice $\Z^d$ such that $e_{2j} = \overline{e_{2j-1}}, j \ge 1$. An equivalent norm is obtained by replacing $(1+\lambda_j)^\gamma$ by $(1+\lambda_j^\gamma)$ and this norm is further equivalent to the graph norm of the image of $L^2$ under the differential operator $(id-\Delta^{\gamma/2})$ when $\gamma/2 \in \N$.

We define topological dual spaces $H^{-\gamma}(\Omega) = (H^\gamma(\Omega))^*$ with norm
\begin{equation}\label{dual}
\|f\|_{H^{-\gamma}} = \sup_{\psi \in H^{\gamma}: \|\psi\|_{H^{\gamma}} \le 1} | \langle f , \psi \rangle_{L^2}|,
\end{equation}
where the supremum may be restricted to $\psi \in C^\infty(\Omega)$. If we understand $\langle h, e_j \rangle_{L^2}$ as the action $T_h(e_j)$ of periodic Schwartz distributions $T_h$ on smooth test functions $e_j \in C^\infty(\Omega)$, then the $h^{-\gamma}$ norms from (\ref{seqnorm}) are equivalent to these dual norms also for negative $-\gamma<0$, and (\ref{dual}) is valid for all $\gamma \in \R$. We define closed subspaces $$H^\gamma_0 =H^\gamma \cap \{h: \langle h, 1\rangle_{L^2}=0\}$$ and note that $H^0_0(\Omega) = L^2_0(\Omega)$. On $H^\gamma_0(\Omega)$ we have equivalent norms as in (\ref{seqnorm}) but using only $\lambda_j^{\gamma}$ in place of $(1+\lambda_j)^\gamma$, taking note of the Poincar\'e inequality  $\lambda_1 \ge 1/2\pi$. 

If we set $B^a=H^a$ for $a>d/2$ and $B^a=C^a$ for integer $a \in [0, d/2]$, then the multiplier inequality for Sobolev norms is
\begin{equation}\label{multbasic}
\|fg\|_{H^a} \lesssim \|f\|_{H^a} \|g\|_{B^a}
\end{equation}
For $a<0$ we can use (\ref{dual}) and (\ref{multbasic}) to obtain
\begin{equation}\label{multbasicn}
\|fg\|_{H^a} \lesssim \|f\|_{H^a} \|g\|_{B^{|a|}.}
\end{equation}
We also need the following interpolation inequality for Sobolev norms (see (A.16), p.473, in \cite{TI} and Lemma 3.27 in \cite{R01})
 \begin{equation}\label{interpol}
 \|h\|_{h^\xi} \le \|h\|^{1-m}_{h^{\bar \gamma}} \|h\|^{m}_{h^{-b}}, ~~m= \frac{\bar \gamma - \xi}{\bar \gamma+b}, ~-\infty<-b \le \xi \le \bar \gamma<\infty.
 \end{equation}

To metrise the distance between two probability measures $\mu, \nu$ on a metric space $(X,d)$, we will use the transportation Wasserstein-$1$ distance,
\begin{equation} \label{wassdef}
W_1(\mu, \nu) = W_{1,(X,d)}(\mu, \nu) = \sup_{F: X \to \R, \|F\|_{Lip}\le 1} |EF(Z_1)-EF(Z_2)|
\end{equation}
 where $Z_1 \sim \mu, Z_2 \sim \nu$ and $$\|F\|_{Lip} = \sup_{x \neq y; x,y \in X} \frac{|F(x)-F(y)|}{d(x,y)}.$$ See Chapter 6 in \cite{V09} for equivalent definitions and further properties.

\section{Functional Bernstein--von Mises theorems}

\subsection{A general result for non-linear parabolic PDEs}

Let $\Omega = [0, 1]^d, d \in \mathbb N$, be the $d$-dimensional torus, and let $W$ be a finite-dimensional vector space. If each entry $u_k$ of a $W$-valued vector field $u : \Omega \to W$ lies in a function space $X(\Omega)$ we write $u \in X(\Omega,W)$, and the corresponding norm is denoted by $\|u\|^2_{X(\Omega, W)} =\sum_{k} \|u_k\|_{X(\Omega)}^2$. Often we will just write $X(\Omega)$ for $X(\Omega, W)$. We consider a general forward map 
\begin{equation} \label{gmap}
\theta \mapsto \G(\theta),~~\G: H^1(\Omega,W) \to L^2([0,T], L^2(\Omega, W)),~~T>0.
\end{equation}
When $W=\R$ this accommodates the solution map of the reaction-diffusion system (\ref{evol}), but further allows dynamical systems of vector fields, such as coupled systems of PDEs. The measurement model (\ref{data}) is then a special case of the random design regression equation 
\begin{equation} \label{modelG}
Y_i = \G(\theta)(X_i) + \varepsilon_i, ~~i=1, \dots, N,~~\varepsilon_i \sim^{iid} \mathcal N(0, Id_W),
\end{equation}
where the $X_i$ are drawn iid from the uniform distribution $\lambda = dt dx/T $ on $\mathcal X = [0,T] \times \Omega$, independently of the Gaussian noise vectors $\varepsilon_i$. Just as after (\ref{data}) we write $Z^{(N)} := (Y_i,X_i)_{i=1}^N$, denote the product measure describing the law of the infinite observation vector $(Y_i, X_i)_{i=1}^\infty$ by $P_{\theta}^\mathbb N$, and occasionally write $L^2_\lambda(\mathcal X)$ or just $L^2(\mathcal X)$ for $L^2([0,T], L^2(\Omega))$ -- these are the same spaces but $L^2_\lambda(\mathcal X)$ carries an inner product $$\langle g,h \rangle_{L^2_\lambda(\mathcal X)} = \frac{1}{T}\int_0^T \int_\Omega g(t,x)h(t,x) dtdx$$ scaled by $1/T$ which is a minor difference relevant in some adjoint calculations in the proofs. This random design regression setting permits to borrow from the theory developed in \cite{N23} as well as the use of tools from concentration of product measures in infinite-dimensions (e.g., Ch.3 in \cite{GN16}). Modelling explicit dependence structures in the design complicates the development significantly but is in principle possible, for instance as in \cite{N24}.

\smallskip

We now turn to the prior probability measure for the parameter $\theta \in H^1$ -- we refer to Ch.2 in \cite{GN16} for standard background from the theory of Gaussian processes, such as the definition of their reproducing kernel Hilbert spaces (RKHS).

\begin{condition}\label{priorgen} Consider the centred Gaussian Borel probability measure $\Pi'=\Pi'_\gamma$ defined on $L^2_0(\Omega,W)$ with  RKHS $\mathcal H$ equal to the (real) Hilbert space $H^\gamma_0(\Omega,W)$ for some $\gamma>1+d/2$. Then for $\theta' \sim \Pi'$ take as prior $\Pi=\Pi_{\gamma, N}$ the law of 
$$\theta = \frac{1}{\sqrt N \delta_N} \theta' \text{ where }\delta_N = N^{-\frac{\gamma}{2\gamma+d}},$$ with resulting RKHS-norm $\|\cdot\|_{\mathcal H_N} = \sqrt N \delta_N \|\cdot\|_\mathcal H$.
\end{condition}

By the Karhunen-Lo\`eve theorem, for $h_j$ any orthonormal basis of the RKHS $H_0^\gamma(\Omega, W)$, the prior $\Pi'_\gamma$ can be represented by the law of a Gaussian random series 
 \begin{equation} \label{gaussser}
\theta'(x) =  \sum_{j=1}^\infty g_{j} h_j,~~\text{ where } g_{j} \sim^{iid} N(0,1),
\end{equation}
augmented to independent such copies in each of its coordinates whenever $dim(W)>1$. Shrinking the prior towards zero ensures a degree of a priori regularisation and has been used throughout proofs in the non-linear inverse problems literature since \cite{MNP21}. By the theory of radonifying maps (as in Thm B.1.3 in \cite{N23}), $\Pi'_\gamma$ and $\Pi$ define centred Gaussian Borel probability measures on the separable Hilbert space $$ H^{\bar \gamma}_0(\Omega, W),~\text{for any}~1 \le \bar \gamma<\gamma-d/2,$$ which serves as the natural `parameter' space charged almost surely by prior draws. The restriction $\int_\Omega \theta =0$ is convenient in the proofs involving the Fisher information operator as it removes the kernel of the Laplacian $\Delta$ relevant in Condition \ref{gemol}F) below -- for the reaction-diffusion system from Theorem \ref{ganzwien}, this restriction could be removed by PDE-specific arguments but we abstain from this to streamline the presentation. The posterior law of $\theta|Z^{(N)}$ then arises from a dominated model and is given by
\begin{equation}\label{postlog}
d\Pi(\theta|Z^{(N)}) \propto e^{\ell_N(\theta)} d\Pi(\theta),~~\ell_N(\theta) = -\frac{1}{2} \sum_{i=1}^N |Y_i - \G(\theta)(X_i)|_W^2,~~\theta \in H^1_0(\Omega, W),
\end{equation}
see \cite{GV17} or Sec.1.2.3 in \cite{N23}.

\smallskip

We now formulate some analytical conditions on $\G$ -- these are designed to be verifiable for solution maps $\theta \mapsto \G(\theta)$ of non-linear parabolic PDEs as in (\ref{genpde}).   We denote the balls in $H^r_0$ by 
\begin{equation} \label{rball}
U(r, B)=\big\{u \in H^r_0(\Omega, W): \|u\|_{H^r} \le B \big\},~B>0,
\end{equation}
 and the regularity properties will be required to hold uniformly on such balls with constants that may depend on $B$ and the fixed ground truth $\theta_0$. The constants may depend on further parameters $T,W, d, \bar \gamma, \gamma, \zeta$ which we however suppress in the notation.

\begin{condition}\label{gemol}
Let $\theta_0 \in H^{\gamma}_0(\Omega, W)$ for some $\gamma>1+d$. Suppose a forward map $\G$ as in (\ref{gmap}) satisfies the following hypotheses for some $\bar \gamma, \zeta$ such that $1 \le \bar \gamma < \gamma -d/2$, $d/2 < \zeta < \gamma -d/2$, respectively:

\smallskip

A) For every $B>0$ and some $u=u(B)>0$,
\begin{equation} \label{ubd}
\sup_{\theta \in U(\bar \gamma, B), 0 < t \le T, x \in \Omega} |\G(\theta)(t,x)| \le u <\infty.
\end{equation}

\smallskip

B) For every $B>0$, there exists $c=c(B)>0$ such that  $$\|\G(\theta)-\G(\vartheta)\|_{L^2([0,T], L^2(\Omega))} \le c \|\theta - \vartheta\|_{L^2(\Omega)}~~\forall \theta, \vartheta \in U(\bar \gamma, B).$$ 

\smallskip

C) [Linear approximation:] There exists a continuous linear operator 
\begin{equation}\label{linmap}
\mathbb I_{\theta_0} \equiv D\G_{\theta_0} : L^2(\Omega) \to L^2_\lambda(\mathcal X)
\end{equation}
and constant $ c=c(\theta_0, B)>0,$ such that for all $h \in U(\bar \gamma, B), B>0,$ one has $$\|\G(\theta_0+h) - \G(\theta_0) - \mathbb I_{\theta_0}[h]\|_{L^2([0,T], L^2(\Omega))} \le c \|h\|^{2}_{L^2(\Omega)}.$$

\smallskip

D) [$L^\infty$-mapping properties:] Suppose that for all $B>0$ there exists $c=c(\theta_0, B)>0$ such that
\begin{equation}\label{zetabus}
\|\G(\theta) - \G(\vartheta)\|_{L^\infty([0,T]), L^\infty(\Omega))} \le c \|\theta - \vartheta\|_{H^{\zeta}},~~\forall \theta, \vartheta \in U(\zeta, B) \cap H^1.
\end{equation}
Suppose also that $\mathbb I_{\theta_0}$ from (\ref{linmap}) is continuous from $H^{\zeta} \to L^\infty([0,T],L^\infty(\Omega))$.

\smallskip

E) [Stability:] Suppose that for all $B>0$ there exists a  constant $c=c(\theta_0, B)>0$ such that $$\|\G(\theta) - \G(\theta_0)\|_{L^2([0,T], L^2(\Omega))} \ge c\|\theta-\theta_0\|_{H^{-1}},~~ \forall \theta \in U(\bar \gamma, B).$$

\smallskip

F) [Inverse information:] Take $\mathbb I_{\theta_0}$ from (\ref{linmap}) with Hilbert space adjoint operator $$\mathbb I_{\theta_0}^* : L_\lambda^2(\mathcal X) \to L^2(\Omega),$$ and consider the information operator $\mathbb I_{\theta_0}^*\mathbb I_{\theta_0}$ acting on $H^\eta_0 (\Omega) \subset L^2(\Omega)$. For $\Delta$ the Laplacian, some $\eta_0 \ge 0$ and all $\eta \ge \eta_0$, assume that 
\begin{equation}\label{ical}
\mathcal I \equiv \Delta \mathbb I_{\theta_0}^*\mathbb I_{\theta_0}
\end{equation}
defines a continuous linear homeomorphism from $H^\eta_0 \to H^\eta_0$.
\end{condition}

\begin{remark}[Priors on subspaces] \label{subprior} \normalfont In some applications it is of interest to consider $\theta'$ in (\ref{gaussser}) projected onto a closed subspace $\mathbb H$ of $H^\gamma_0(\Omega, W)$, for instance to enforce the divergence free constraint in the $2d$-Navier-Stokes model with $W=\R^2$ -- see \cite{NT23, KN25}. If this subspace $\mathbb H$ is compatible with the standard periodic Sobolev scale (as is the case with periodic divergence free vector fields), we can invoke Condition \ref{gemol} to hold with $H^\gamma_0(\Omega, W)$ replaced everywhere by $\mathbb H$, and the proof of Theorem \ref{mainbvm} below still applies after mostly notational adjustments, see \cite{KN25}, in particular Section 2.5.
\end{remark}

\begin{remark} [Condition \ref{gemol} and parabolic PDE] \normalfont \label{template}
Conditions A)-D) can be expected to follow from regularity estimates for `dissipative' systems of the form (\ref{genpde}), see e.g., \cite{R01, E10}. We will demonstrate in Sec.\ref{pdean} how this works for periodic reaction-diffusion equations (\ref{evol}), but similar estimates exist for McKean-Vlasov systems  \cite{NPR24} and for $2d$-Navier-Stokes equations  \cite{NT23}. In contrast, Conditions E) and F) are more subtle and related to identifiability properties of the dynamical system, and the question of statistical consistency of likelihood-based inference methods -- see \cite{NT23} and also \cite{N24, NPR24} for related settings. Notably, \cite{NT23} prove the consistency of the posterior distributions $\hat \Pi_{t,N}$ from (\ref{imagemeasure}) in the periodic $2d$-Navier-Stokes setting, based on a logarithmic stability estimate for $\theta \mapsto \G(\theta)$. The Lipschitz stability condition E) is, however, a stronger requirement, which we verify for reaction diffusion equations below. It is possible because we have access to \textit{time average} measurements near the origin $t=0$ in (\ref{modelG}) -- note that the logarithmic lower bounds in \cite{NT23} explicitly rule out such situations. Our proof relies on the symmetry of certain Schr\"odinger operators (\ref{sop}) appearing in the linearisation. This approach may not apply directly to general PDEs (\ref{genpde}), but Condition E) and the related injectivity part of Condition F) can still be verified for `dissipative' parabolic systems such as Navier-Stokes dynamics, as was shown recently in the follow up paper \cite{KN25}. See also Remark \ref{clever} for more discussion. \end{remark}

We now state a Bernstein-von Mises theorem for the posterior on $\theta$ in the `weak' norm topology of the space $H^{-k-2}$ for appropriate $k>0$. The posterior mean vector $\tilde \theta_N= E^\Pi[\theta|Z^{(N)}]$ exists as a Bochner integral in $L^2_0$ and hence also in $H^{-k-2}$, while the Gaussian measure $\mathcal N_{\theta_0}$ is constructed in Proposition \ref{limmeas} as a tight Borel law on $H^{-k-2}$.

\begin{theorem}\label{mainbvm}
Suppose Conditions \ref{priorgen} and \ref{gemol} hold for some integer $\gamma>\max(3\zeta + (3d/2)+2, 3 \zeta + d +3, \eta_0)$ and $k> \max(\zeta + (5d/2), \eta_0+d/2)$. Let $\theta|Z^{(N)} \sim \Pi(\cdot|Z^{(N)})$ with posterior from (\ref{postlog}). For $W_1$ the Wasserstein distance (\ref{wassdef}) on the (real) Hilbert space $H^{-k-2}(\Omega)$, 
$$W_1(Law(\sqrt N (\theta - \tilde \theta_N)|Z^{(N)}), \mathcal N_{\theta_0}) \to^{P_{\theta_0}^\mathbb N} 0$$ as $N \to \infty$, as well as
$$\sqrt N (\tilde \theta_N - \theta_0) \to^d \mathcal N_{\theta_0} \text{ in } H^{-k-2}.$$
\end{theorem}

The particular choice of $k$ is not important when combined in conjunction with a parabolic smoothing argument such as the one to be used in the proof of Theorem \ref{ganzwien} below. The lower bound on $\gamma$ could be improved slightly by introducing further technicalities (e.g., by Schauder- rather than just energy estimates in Condition \ref{gemol}D). 

The general proof strategy follows ideas developed in \cite{CN14, CR15, N20, MNP21a} who derive Bernstein-von Mises theorems in a variety of `non-conjugate' settings. Infinite  dimensional (`functional') such results have so far only been obtained in `direct' iid models in \cite{CN14, CvP21} and for non-linear inverse problems in \cite{N20, NS19} not covering Gaussian priors. Here we extend the `semi-parametric' proof in \cite{MNP21a} for fixed one-dimensional functionals $\langle \theta, \psi \rangle$ to hold uniformly in $\psi$ belonging to a unit ball of $H^{k+2}$, thereby obtaining a Gaussian approximation in the infinite-dimensional $\|\cdot\|_{H^{-k-2}}$ space. We also prove our results in the stronger Wasserstein distance rather than just in a metric for weak convergence, allowing to obtain the asymptotics of the posterior mean and first moments.

\subsection{Proof of Theorem \ref{mainbvm}}

Let us remark first that the hypotheses on $\gamma, \kappa$ we employ are typically `more than sufficient' to control the size of various approximation terms in the proofs that follow -- with the exception of `critical' terms that arise at the end of the proof of Lemma \ref{oldtimer}. To proceed we need to fix a few constants in advance: given $k$ from the theorem and $\zeta>d/2$ from Condition \ref{gemol}D, we take $\kappa$ satisfying 
\begin{equation} \label{ks}
\max(2d+\zeta, \eta_0)<\kappa < k-\frac{d}{2}.
\end{equation}
Further we take $\bar \gamma$ sufficiently close to $\gamma-d/2$ satisfying
\begin{equation}\label{gs}
3\zeta +d +2 <\bar \gamma<\gamma-\frac{d}{2},
\end{equation}
and assume without loss of generality that Condition \ref{gemol} holds for this $\bar \gamma$. [We note that if Condition \ref{gemol} holds for some $\bar \gamma'$, then it also holds for all $\bar \gamma \ge \bar \gamma'$ because $H_0^{\bar \gamma} \subset H_0^{\bar \gamma'}$.] 

\medskip

Define the span in $L^2(\Omega, W)$ of the ($dim(W)$-dimensional tensors of) trigonometric polynomials (\ref{eftrig}) as
\begin{equation}\label{ej}
E_J = span\{e_j: 0 \le j \le J\} \text{ with associated } L^2 \text{-projector } P_{E_J}: L^2 \to E_J,
\end{equation} 
where $J$ is even such that $E_J$ is conjugation-stable and $P_{E_J}$ maps real functions to real functions. Since the RKHS of the base prior from Condition \ref{priorgen} equals $\mathcal H = H^\gamma_0(\Omega)$ with equivalent norm $\|\cdot\|_{h^\gamma}$, and using (\ref{weyl}), we have for every $J \in \mathbb N$ and some $c>0$ the standard approximation bounds
\begin{equation}\label{approxh}
\|P_{E_J}\phi - \phi\|_{L^2} \le c \|\phi\|_{\mathcal H} J^{-\gamma/d},~~\|P_{E_J}\phi\|_{\mathcal H} \le c \max(1,J^{(\gamma - r)/d}) \|\phi\|_{H^r},~~\forall \phi \in L^2_0(\Omega),
\end{equation}
which we shall use repeatedly in the proofs.

\subsubsection{Posterior contraction, regularity, and localisation}

The first step is a global `consistency' result about the contraction of the posterior distribution (\ref{postlog}) towards the `ground truth' initial condition $\theta_0$. Proofs of this kind for `direct' regression models follow ideas laid out in \cite{GV17}, specifically for Gaussian priors see \cite{vdvaart2008}. In the context of (non-linear) inverse problems as relevant here, such proofs were developed in \cite{MNP21} and then, in a general setting in Chapter 2 in \cite{N23}, assuming `stability' estimates for the $\G$ map that we will verify using Condition \ref{gemol}E). A novel feature we require is uniform control of the interaction of the RKHS inner product $\langle \cdot, \cdot \rangle_{\mathcal H}$ of the prior with test functions exhausting the dual norm of $H^{-\kappa}$. 
 
 \smallskip

For $\gamma, \delta_N$ from Condition \ref{priorgen} and $\bar \gamma, \kappa$ from (\ref{gs}), (\ref{ks}), define sequences
\begin{equation}\label{deltan}
\tilde \delta_N(\xi) = \delta_N^{(\bar \gamma-\xi)/(\bar \gamma+1)},~ 0 \le \xi \le \bar \gamma,~~ \tilde \delta_N(0) \equiv \tilde \delta_N,
\end{equation}
as well as
\begin{equation} \label{approxrkhs}
J_N \in \mathbb N,~ J_N \simeq N \delta_N^2,~~K_N =\sqrt {J_N} \max(1,J_N^{(\gamma-\kappa)/d}).
\end{equation}
In the following proof only Conditions \ref{gemol}A), B) and E) are used, and the hypothesis on $\gamma$ from Theorem \ref{mainbvm} could be somewhat weakened, as inspection of the proof shows. 

\begin{theorem} \label{postcont} For $\psi \in L_0^2(\Omega, W)$ define $L^2$-projections onto $E_J$ from (\ref{ej}) as
\begin{equation}\label{projpsi}
p_N(\psi)=P_{E_{J_N}}\psi
\end{equation}
Let $\bar \gamma, \kappa$ be as in (\ref{gs}), (\ref{ks}), the RKHS $\mathcal H_N$ as in Condition \ref{priorgen}, and for $L>0$ define measurable sets
\begin{align} \label{barthetan}
\bar \Theta_N &=\Bar \Theta_{N,L} := \Big\{\theta \in H^1_0: \|\theta\|_{H^{\bar \gamma}} \le L, ~ \sup_{\psi \in U(\kappa,1)} |\langle \theta, p_N(\psi) \rangle_{\mathcal H_N}| \le L \sqrt N \delta_N K_N  \Big\}  \\
&~~~~ \cap ~\Big\{\|\G(\theta)-\G(\theta_0)\|_{L^2([0,T], L^2(\Omega))} \le L \delta_N,~ \|\theta - \theta_0\|_{H^\xi} \le L \tilde \delta_N(\xi) ~\forall 0 \le \xi <\bar \gamma \Big\}. \notag
\end{align}
Let the posterior measure $\Pi(\cdot |Z^{(N)})$ be as in (\ref{postlog}) arising from data (\ref{modelG}) and prior $\Pi_{\gamma, N}$ from Condition \ref{priorgen}. Then for all $b>0$ we can choose $L$ large enough but finite such that
\begin{equation} \label{ratekey}
P_{\theta_0}^\mathbb N \big(\Pi(\bar \Theta_{N,L} |Z^{(N)}) \le 1 - e^{-bN\delta_N^2} \big) \to_{N \to \infty} 0.
\end{equation}
We further have for any $1 \le q<\infty, 0 \le \xi \le \bar\gamma,$ that
 \begin{align}\label{meanreg}
 \|E^\Pi[\theta|Z^{(N)}] - \theta_0\|_{H^\xi} &\le (E^\Pi[\|\theta-\theta_0\|^q_{H^\xi}|Z^{(N)}])^{1/q} = O_{P_{\theta_0}^\mathbb N}(\tilde \delta_N(\xi))
 \end{align}
\end{theorem}
\begin{proof}
We start with a preparatory inequality: Consider
\begin{align} \label{excesss}
\Pi(\Theta'_N) & \equiv \Pi \big(\theta:\sup_{\psi \in U(\kappa,1)} |\langle \theta, p_N(\psi) \rangle_{\mathcal H_N}| >L  \sqrt N \delta_N K_N \big) \notag \\
&= \Pi'\big(\theta':\sup_{\psi \in U(\kappa,1)} |\langle \theta', p_N(\psi) \rangle_{\mathcal H}| >L  K_N \big),
\end{align} 
with $\langle \cdot, \cdot \rangle_{\mathcal H_N} = N \delta_N^2 \langle \cdot, \cdot \rangle_\mathcal H$ and $\theta= \theta'/{\sqrt N \delta_N}$. From the series expansion (\ref{gaussser}) for $\theta' \sim \Pi'$ and Parseval's identity we know that $$\{X(\psi) = \langle \theta', p_N(\psi) \rangle_{\mathcal H}: \psi \in U(\kappa,1)\}$$ is a mean zero Gaussian process with covariance metric $$d^2_X(\psi, \psi') =  \|p_N(\psi)-p_N(\psi')\|_{\mathcal H}^2 \lesssim \max(1,J_N^{2(\gamma-\kappa)/d} )\|\psi-\psi'\|^2_{H^{\kappa}} \equiv \sigma_N^2 \|\psi-\psi'\|^2_{H^{\kappa}},$$ using (\ref{approxh}) in the last estimate.  The image of $U(\kappa,1)$ under the projection map $p_N$ onto $E_{J_N}$ thus has $\eta$-covering numbers for $d_X$-distance bounded by that of a ball of radius $\sigma_N$ in a $J_N+1$-dimensional vector space and is hence of the order $$N(U(\kappa, 1), d_X, \eta) \lesssim \Big(\frac{A\sigma_N}{\eta}\Big)^{J_N+1}, ~0<\eta<A\sigma_N,~A>0,$$ see Section 4.3.7 in \cite{GN16}. We apply Dudley's metric entropy bound, Theorem 2.3.7 in \cite{GN16}, with $t_0=\psi=0\in U(\kappa,1)$ and the substitution $v = \eta /\sigma_N$ to obtain
$$E\sup_{\psi \in U(\kappa,1)} |X(\psi)| \lesssim  \int_0^{2\sigma_N} \sqrt {J_N\log \Big(\frac{A\sigma_n}{\eta}\Big)}d\eta \lesssim \sqrt {J_N} \sigma_N.$$
Then using the concentration inequality Theorem 2.1.20 in \cite{GN16} for suprema of Gaussian processes we obtain for all $L$ large enough that the probability in (\ref{excesss}) is bounded as
 \begin{align}
 \Pi(\Theta'_N) &\le \Pi' \Big(\theta': \sup_{\psi \in U(\kappa,1)} |\langle \theta', p_N(\psi) \rangle_{\mathcal H}| - E\sup_{\psi \in U(\kappa,1)} |\langle \theta', p_N(\psi) \rangle_{\mathcal H}|  > (L/2)  K_N  \Big) \notag \\
 &\lesssim e^{-cL^2 K_N^2/\sigma_N^2}  \lesssim e^{-c'L^2 N \delta_N^2} \label{excessbd}
 \end{align}
 for some $c,c'>0$.
 
We now apply Theorem 2.2.2 in \cite{N23} with parameter space $\Theta = H^1_0$ and regularisation space $\mathcal R = H_0^{\bar \gamma}(\Omega)$ (as remarked on p.31 in \cite{N23}, the proof applies in our periodic setting just as well). We can verify Condition 2.1.1 in \cite{N23} for $\kappa=0, \mathcal X = [0,T] \times \Omega, V=W,$ in view of Condition \ref{gemol}A) and B), and Condition 2.2.1 in \cite{N23} with such $\mathcal R$ by Condition \ref{priorgen}. Thus Theorem 2.2.2 from \cite{N23} implies for all $L$ large enough,
\begin{equation} \label{fwdratebd}
\Pi\Big(\theta \in (\Theta'_N)^c: \|\theta\|_{H^{\bar \gamma}} \le L, \|\G(\theta)-\G(\theta_0)\|_{L^2(\mathcal X)} \le L \delta_N |Z^{(N)}\Big) \to^{P_{\theta_0}^\mathbb N} 1
\end{equation}
as $N \to \infty$, in fact as in \cite{N23} with the required convergence rate bound, and restriction to $(\Theta'_N)^c$ being possible by (\ref{excessbd}) and the remark on p.33 in \cite{N23}. The proof further implies by virtue of (1.28) in \cite{N23} that 
\begin{equation}\label{AN}
P_{\theta_0}^\mathbb N \big(A_N \big) \to_{N \to \infty} 1 \textit{ where } A_N = \left\{\int_{H^1_0} e^{\ell_N(\theta)-\ell_N(\theta_0)} d\Pi_N(\theta) \ge e^{-(A+2) N \delta_N^2} \right\}
\end{equation}
for some $A=A(\Pi')>0$, which we will use below.

\smallskip

The next step is based on Condition \ref{gemol}E) which implies for $\theta$ in the event inside of the probability in (\ref{fwdratebd}) that for some constant $c=c(L,\theta_0, T)$,
\begin{equation}\label{pseudolin}
\|\G(\theta) - \G(\theta_0)\|_{L^2(\mathcal X)} \ge  c \|\theta-\theta_0\|_{H^{-1}}.
\end{equation}
 From interpolation (\ref{interpol}) with $b=1, m=(\bar \gamma-\xi)/(\bar \gamma+1)$ we deduce for all $0 \le \xi \le \bar \gamma$,
\begin{equation}\label{stabilityestimate}
\|\theta - \theta_0\|_{H^\xi} \lesssim \|\G(\theta) - \G(\theta_0)\|^m_{L^2(\mathcal X)}
\end{equation}
which combined with (\ref{fwdratebd}) implies the limit (\ref{ratekey}).

It remains to prove (\ref{meanreg}). The first inequality follows from Jensen's inequality. Then from the Cauchy-Schwarz inequality we can bound 
\begin{align} \label{qreg}
(E^\Pi[\|\theta-\theta_0\|^q_{H^\xi}|Z^{(N)}])^2 &\lesssim L^{2q}\tilde \delta^{2q}_N(\xi) + E^\Pi[\|\theta-\theta_0\|_{H^\xi}^{2q}|Z^{(N)}] \Pi(\|\theta-\theta_0\|_{H^\xi}>L\tilde \delta_N(\xi)|Z^{(N)}) \notag \\
&\equiv L^{2q}\tilde \delta^{2q}_N(\xi) + t_N.
\end{align}
For the second summand we use (\ref{ratekey}), (\ref{AN}), Markov's inequality,  Fubini's theorem, (\ref{postlog}), $E_{\theta_0}^\mathbb N[e^{\ell_N(\theta) - \ell_N(\theta_0)}] \le 1$ and finiteness of all $\|\cdot\|_{H^\xi}$-norm moments of $\Pi$ from Condition \ref{priorgen} (cf.~Exercise 2.1.5 in \cite{GN16})) to obtain
\begin{align*}
P_{\theta_0}^\mathbb N \big(t_N \ge \tilde \delta^{2q}_N(\xi)\big) &\le P_{\theta_0}^\mathbb N\Big(e^{(A+2-b)N\delta_N^2}\int \|\theta-\theta_0\|_{H^\xi}^{2q} e^{\ell_N(\theta)-\ell_N(\theta_0)} d\Pi(\theta) > \tilde \delta^{2q}_N(\xi) \Big) + o(1) \\
&\le e^{(A+2-b)N\delta_N^2} \tilde \delta^{-2q}_N(\xi) \int \|\theta-\theta_0\|_{H^\xi}^{2q} d\Pi(\theta) + o(1) \to_{N \to \infty} 0
\end{align*}
 for $b$ and then $L$ large enough, so that (\ref{meanreg}) follows.
\end{proof}
We can use the previous contraction theorem to show that the posterior measure is asymptotically equivalent to a hypothetical posterior $\Pi^{\bar \Theta_N}(\cdot|Z^{(N)})$ arising from (\ref{postlog}) with `localised' prior 
\begin{equation} \label{respr}
\Pi^{\bar \Theta_N} = \frac{\Pi(\cdot \cap \bar \Theta_N)}{\Pi(\bar \Theta_N)}
\end{equation}
 restricted to the regularisation set $\bar \Theta_N=\bar \Theta_{N,L}$ from ({\ref{barthetan}). Specifically, from the arguments on p.142 in \cite{vdvaart1998}, Theorem \ref{postcont} then implies that for any $b>0$ and $L(b)$ large enough, with $P_{\theta_0}^\mathbb N$-probability approaching one as $N \to \infty$,
\begin{equation} \label{tvbd}
\|\Pi(\cdot|Z^{(N)}) - \Pi^{\bar \Theta_{N,L}}(\cdot|Z^{(N)})\|_{TV} \lesssim \Pi(\bar \Theta_{N,L}^c|Z^{(N)}) \lesssim e^{-b N\delta_N^2}
\end{equation}
where $\|\cdot\|_{TV}$ is the usual total variation norm on the space of probability measures on $L^2_0(\Omega)$. 

\smallskip

Now if we denote by $\tau_N$ the (conditional) law of $\sqrt N (\theta - T)$ for $\theta \sim \Pi(\cdot|Z^{(N)})$ and any fixed re-centring $T \in H^{-k-2}$, and if we denote by $\bar \tau_N$ the corresponding law where $\theta$ is replaced by a draw $\theta \sim \Pi^{\bar \Theta_N}(\cdot|Z^{(N)})$, then we obtain for the Wasserstein distance $W_1=W_{1, H^{-k-2}}$ featuring in Theorem \ref{mainbvm}, the following approximation:
\begin{proposition} \label{wassconv}
For any $T \in H^{-k-2}$ and some $c>0$ we have as $N \to \infty$ $$W_{1}(\tau_N, \bar \tau_N) = O_{P_{\theta_0}^\mathbb N}(e^{-cN\delta_N^2}).$$
\end{proposition}
\begin{proof} We use Theorem 6.15 in \cite{V09} to obtain for $\vartheta_o = \sqrt N (\theta_0-T)$ the inequality
\begin{align}
&W_1(\tau_N, \bar \tau_N) \le \int_{H^{-k-2}} \|\vartheta-\vartheta_{o}\|_{H^{-k-2}} d|\tau_N - \bar \tau_N|(\vartheta) \notag \\
&\leq 2\sqrt N \|\Pi(\cdot|Z^{(N)}) - \Pi^{\bar \Theta_{N,L}}(\cdot|Z^{(N)})\|_{TV} + \\
&~~~ \int_{\|\vartheta- \vartheta_{o}\|_{H^{-k-2}} > \sqrt N} \|\vartheta-\vartheta_{o}\|_{H^{-k-2}} d\tau_N(\vartheta) +  \int_{\|\vartheta-\vartheta_{o}\|_{H^{-k-2}} > \sqrt N} \|\vartheta-\vartheta_{o}\|_{H^{-k-2}} d\bar \tau_N(\vartheta). \notag
\end{align}
From (\ref{tvbd}) the first term is of order $\sqrt N e^{-b N \delta_N^2} \to 0$ as $N \to \infty$. For the second term we have from $\vartheta-\vartheta_o = \sqrt N(\theta-\theta_0)$, the Cauchy-Schwarz inequality, the continuous imbedding $L^2 \subset H^{-k-2}$ and Theorem \ref{postcont}, for all $N$ large enough,
\begin{align*}
& \sqrt N E^{\Pi}[\|\theta- \theta_{0}\|_{H^{-k-2}} 1_{\{\|\theta-\theta_0\|_{H^{-k-2}} > 1\}}|Z^{(N)}] \\
&\le \sqrt N \big(E^{\Pi}[\|\theta- \theta_{0}\|_{L^2}^2|Z^{(N)}])^{1/2} \Pi(\|\theta-\theta_0\|_{L^2}> 1 |Z^{(N)})\big)^{1/2} =O_{P_{\theta_0}^\mathbb N}\big( e^{-bN \delta_N^2/2}\tilde \delta_N N^{1/2} \big).
\end{align*} 
Finally, by definition of $\bar \Theta_N$ and since $\tilde \delta_N=o(1)$, we have $\Pi^{\bar \Theta_N}(\|\theta-\theta_0\|_{H^{-k-2}}> 1 |Z^{(N)})=0$ from some $N$ onwards so the third term equals zero eventually, and the result is proved. 
\end{proof}

The previous result holds in fact for the $W_{1, L^2_0}$-distance as long as $T \in L^2_0$. We conclude that it suffices to prove the first limit in Theorem \ref{mainbvm} for $\bar \tau_N$ rather than $\tau_N$.

\subsubsection{Inverting the Fisher information operator, and the limiting process}

Let $\psi \in H_0^{\eta+2}(\Omega)$ so that $\Delta \psi$ lies in $H^\eta_0(\Omega)$. If $\eta \ge \eta_0$ then Condition \ref{gemol}F) implies that there exists $\bar \psi \in H^\eta_0(\Omega)$ such that $\mathcal I \bar \psi = \Delta \psi$, which we write as
\begin{equation} \label{barpsi}
\bar \psi =\mathcal I^{-1}\Delta \psi,
\end{equation}
for continuous inverse $\mathcal I^{-1}$. By definition of $\mathcal I = \Delta \mathbb I_{\theta_0}^* \mathbb I_{\theta_0}$ this implies
$$\Delta \mathbb I_{\theta_0}^* \mathbb I_{\theta_0} (\bar \psi) = \Delta \psi ~\Rightarrow~ \Delta(\mathbb I_{\theta_0}^* \mathbb I_{\theta_0} (\bar \psi) - \psi) =0$$  and therefore (e.g., by the maximum principle, or examining Fourier coefficients) $$\mathbb I_{\theta_0}^* \mathbb I_{\theta_0} (\bar \psi) - \psi=c_\psi \text{ where } c_\psi \text{ is constant on } \Omega.$$ [Since $\int_\Omega \psi =0$ this constant equals the first Fourier coefficient $c_\psi = \langle \mathbb I_{\theta_0}^* \mathbb I_{\theta_0} (\bar \psi), 1 \rangle_{L^2}$, but we shall not need this here.]  We can write, for all $h \in L^2_0(\Omega)$, 
\begin{equation}\label{keyfishinv}
\langle \mathbb I_{\theta_0} h, \mathbb I_{\theta_0} \bar \psi \rangle_{L^2(\mathcal X)} = \langle h, \mathbb I_{\theta_0}^*\mathbb I_{\theta_0} \bar \psi \rangle_{L^2(\Omega)}  = \langle h, \mathbb I_{\theta_0}^*\mathbb I_{\theta_0} \bar \psi - c_\psi\rangle_{L^2(\Omega)} = \langle h, \psi \rangle_{L^2},
\end{equation}
so $\bar \psi$ acts as the inverse $(\mathbb I_{\theta_0}^*\mathbb I_{\theta_0})^{-1}\psi$ for such $\psi$. We also have, by continuity of $\mathcal I^{-1}$,
\begin{equation}\label{fishinvbd}
\|\bar \psi\|_{H^\eta} = \|\mathcal I^{-1}[\Delta \psi]\|_{H^\eta} \lesssim \|\Delta \psi\|_{H^\eta} \lesssim \|\psi\|_{H^{\eta+2}},~~ \psi \in H^{\eta+2}_0.
\end{equation}

\medskip

What precedes will be central to our proof of Theorem \ref{mainbvm}, but also is the key to construct the limiting Gaussian measure $\mathcal N_{\theta_0}$: For $g,h \in L^2_0(\Omega) \cap C^\infty(\Omega)$, let $\bar g = \mathcal I^{-1}[\Delta g]$, $\bar h = \mathcal I^{-1}[\Delta h]$ and define a mean zero Gaussian process $\mathbb W$ with covariance 
\begin{align*}
E\mathbb W(g) \mathbb W (h) &= \langle \mathbb I_{\theta_0} \bar g, \mathbb I_{\theta_0} \bar h\rangle_{L^2(\mathcal X)} = \langle \bar  g, \mathbb I_{\theta_0}^* \mathbb I_{\theta_0} \bar h \rangle_{L^2(\Omega)},~~~ g,h \in L^2_0(\Omega) \cap C^\infty.
\end{align*}
Extending from real-valued $g,h$ to complex test functions in the usual way by setting $\mathbb W(g+ih) = \mathbb W(g) - i\mathbb W(h)$, the induced law defines a cylindrical probability measure on $\mathbb C^\N$ by the action $Z=(Z_j = \mathbb W(e_j): j \ge 1)$ of $\mathbb W$ on the orthonormal basis functions $e_j \in C^\infty$ from (\ref{eftrig}). Its restriction (identifying conjugate-symmetric sequences with real periodic distributions) to the real subspace $H_0^{-a}$ of $\mathbb C^\N$ exists almost surely since 
\begin{align*}
E\|Z\|_{h^{-a}}^2 &\lesssim \sum_{j \ge 1} \lambda_j^{-a} E|\mathbb W(e_j)|^2  = \sum_{j \ge 1} \lambda_j^{-a} \langle \mathcal I^{-1} \Delta e_j, \mathbb I_{\theta_0}^*\mathbb I_{\theta_0} \mathcal I^{-1}\Delta e_j \rangle_{L^2(\Omega)} \\
&= \sum_{j \ge 1} \lambda_j^{-a+2} \langle \mathcal I^{-1} e_j, \Delta^{-1} \Delta\mathbb I_{\theta_0}^*\mathbb I_{\theta_0}\mathcal I^{-1} e_j\rangle_{L^2(\Omega)} \\ 
&= \sum_{j \ge 1} \lambda_j^{-a+1} \langle -\mathcal I^{-1} e_j, e_j \rangle_{L^2(\Omega)} \le  \sum_{j \ge 1} \lambda_j^{-a+1} \|\mathcal I^{-1} e_j\|_{h^{\eta_0}} \|e_j\|_{h^{-\eta_0}} \\
&\lesssim \sum_{j \ge 1} \lambda_j^{-a+1} \|e_j\|_{h^{\eta_0}} \|e_j\|_{h^{-\eta_0}}  \le \sum_{j \ge 1} \lambda_j^{-a+1} < \infty
\end{align*}
whenever $2(a-1)/d>1$. Here we have used (recall (\ref{weyl})) that $e_j$ are the eigenfunctions of both $\Delta$ and its inverse $\Delta^{-1}$ for eigenvalues $-\lambda_j, -\lambda_{j}^{-1}$ on $L^2_0$, the Cauchy-Schwarz inequality, and that $-\mathcal I^{-1}$ is a bounded operator on $h^{\eta_0}$. Therefore, using that $H_0^{-a}$ is separable and standard results on Gaussian measures (e.g., Section 2.1 in \cite{GN16}), we obtain:

\begin{proposition}\label{limmeas}
The cylindrical law of $\mathbb W$ defines a tight centred Gaussian Borel probability measure $\mathcal N_{\theta_0}$ on the (real) Hilbert space $H^{-a}(\Omega)$ for every $a>1+d/2$.
\end{proposition}

One may show further that $a>1+d/2$ is necessary for this result to be true (by initially deducing $E|Z_j|^2 \gtrsim \lambda_j$ from Condition \ref{gemol}F) and (\ref{keyfishinv})), in particular $\mathbb W$ does not extend to a Gaussian random function with sample paths in $L^2(\Omega)$.

\subsubsection{Perturbation expansion of the posterior Laplace transform}

We now turn to an asymptotic approximation of the Laplace transform of the random probability measure $\Pi^{\bar \Theta_N}(\cdot|Z^{(N)})$. For any fixed $\psi \in H^{\kappa+2}(\Omega)$ with $\kappa$ as in (\ref{ks}), let $$\bar \psi=\bar \psi_{\theta_0} = \mathcal I^{-1} \Delta \psi,$$ which by (\ref{fishinvbd}) and the Sobolev imbedding theorem defines an element of $$\bar \psi_{\theta_0} \in H^{\kappa}_0(\Omega) \subset C^2(\Omega) \subset L^2(\Omega).$$ Moreover, recall that $P_{E_J}$ denotes the $L^2$-projector onto $E_J = span\{e_j: 0 \le j \le J\}$. For $J_N \simeq N\delta_N^2$ from (\ref{approxrkhs}) let us write $p_N(\cdot) = P_{E_{J_N}}(\cdot)$ and for any $\theta \in H^{\bar \gamma}(\Omega)$ define 
\begin{equation}\label{perturbio}
\theta_{(t, \psi)} := \theta - \frac{t}{\sqrt N} p_N(\bar \psi_{\theta_0}), ~t \in \R,
\end{equation}
which again lies in the support $H_0^{\bar \gamma}(\Omega)$ of the prior (noting also that $\langle p_N(h), e_0 \rangle_{L^2}=0$ whenever $\langle h, e_0 \rangle_{L^2}=0$). Then for $\mathbb I_{\theta_0}$ as in (\ref{linmap}) define
\begin{equation}\label{hatpsi}
\hat \Psi_{N} = \langle \theta_0, \psi \rangle_{L^2} + \frac{1}{N} \sum_{i=1}^N \langle \varepsilon_i, \mathbb I_{\theta_0} p_N (\bar \psi_{\theta_0})(X_i)\rangle_W.
\end{equation}
It will be shown in (\ref{clt}) that $\sqrt N (\hat \Psi_{N} - \langle \theta_0, \psi \rangle_{L^2})$ converges in distribution for fixed $\psi \in C^\infty$ to the $\mathcal N(0, \|\mathbb I_{\theta_0}\bar \psi_{\theta_0}\|_{L^2(\mathcal X)}^2)$ distribution, so $\hat \Psi_N$ serves as an appropriate centring for a Bernstein-von Mises theorem.

\begin{theorem}\label{laplaceapp}Let $\psi \in H^{\kappa+2}(\Omega)$ be real-valued and consider the localised posterior from (\ref{tvbd}) for any sufficiently large but fixed $L>0$. Then we have for every $t \in \R$ that
\begin{align*}
&E^{\Pi^{\bar \Theta_N}}\big[\exp\big\{t \sqrt N \big(\langle \theta, \psi\rangle_{L^2} - \hat \Psi_{N} \big)\big\}|Z^{(N)}\big] = e^{\frac{t^2}{2} \|\mathbb I_{\theta_0}(\bar \psi)\|_{L^2(\mathcal X)}^2} \times \frac{\int_{\bar \Theta_N} e^{\ell_N(\theta_{(t, \psi)})}d\Pi(\theta)}{\int_{\bar \Theta_N} e^{\ell_N(\theta)}d\Pi(\theta)} \times e^{r_N(t, \psi)}
\end{align*}
where $\ell_N$ is as in (\ref{postlog}), and where for every $B>0, t \in \R$, $$\sup_{\|\psi\|_{H^{\kappa+2}} \le B}|r_N(t,\psi)|=o_{P_{\theta_0}^\mathbb N}(1),~~\text{ as } N \to \infty.$$
\end{theorem}
\begin{proof}

We can restrict to $\psi \in U(\kappa+2, B)$ from (\ref{rball}) since $\int_\Omega \theta =0$ and $\Delta \int_\Omega \psi =0$. To prove the theorem we expand as on p.72 in \cite{N23},
$$\ell_N(\theta)-\ell_N(\theta_{(t)}) = I + II$$
where
\begin{equation}\label{eins}
I = \frac{t}{\sqrt N} \sum_{i=1}^N \langle \varepsilon_i, \mathbb I_{\theta_0}p_{N}(\bar \psi_{\theta_0})(X_i) \rangle_{W} + R_{0, N}(\theta, \psi) - R_{t,N}(\theta, \psi)
\end{equation}
with
\begin{equation}
R_{t, N}(\theta, \psi) = \sum_{i=1}^N \langle \varepsilon_i, \G(\theta_{(t, \psi)})(X_i) - \G(\theta_0)(X_i) - D\G_{\theta_0}(X_i)[\theta_{(t, \psi)} - \theta_0] \rangle_W.
\end{equation}
and where
\begin{equation}\label{zwei}
II = -\frac{N}{2} \|\G(\theta_0)-\G(\theta)\|_{L^2(\mathcal X)}^2 + \frac{N}{2} \|\G(\theta_0)-\G(\theta_{(t,\psi)})\|_{L^2(\mathcal X)}^2 + \frac{1}{2}(W_{t,N}(\theta, \psi) - W_{0,N}(\theta, \psi))
\end{equation}
with 
$$W_{t,N}(\theta, \psi) = \sum_{i=1}^N\big[|\G(\theta_0)(X_i)-\G(\theta_{(t,\psi)})(X_i)|_W^2 - E \big|\G(\theta_0)(X_i)-\G(\theta_{(t,\psi)})(X_i)|_W^2\big ].$$
From Markov's inequality and Lemma \ref{oldtimer} below we deduce that the remainder terms $R_{t,N},W_{t,N}$ are all $o_{P_{\theta_0}^\mathbb N}(1)$ uniformly in $\theta \in \bar \Theta_N$ and $\psi \in U(\kappa+2,B)$. 

Next, using Condition \ref{gemol}C, the first two summands in $II$ can be expanded, as on p.73 in \cite{N23}, as
\begin{equation}\label{getthere}
-t \sqrt N \langle \mathbb I_{\theta_0} (\theta - \theta_0), \mathbb I_{\theta_0} p_N[\bar \psi] \rangle_{L^2(\mathcal X)} + \frac{t^2}{2} \|\mathbb I_{\theta_0}p_N(\bar \psi)\|_{L^2(\mathcal X)}^2 + V_N
\end{equation}
where $V_N=o_{P_{\theta_0}^\mathbb N}(1)$ in view of (\ref{deltan}), 
the Cauchy-Schwarz inequality, continuity of $\mathbb I_{\theta_0}: L^2(\Omega) \to L^2([0,T], L^2(\Omega))$, so that $\theta \in \bar \Theta_N$, $\psi \in U(\kappa+2, B)$,
$$N\|\theta_{(t, \psi)}-\theta_0\|_{L^2}^3 \lesssim N \big(\|\theta-\theta_0\|_{L^2} + \frac{t}{\sqrt N} \|p_N(\bar \psi)\|_{L^2} \big)^3 \lesssim N\tilde \delta^3_N(0) = N^{-a},$$ using also that by our hypothesis on $\gamma$,
$$a = \frac{\gamma (\bar \gamma -2) - d(\bar \gamma+1)}{(2\gamma+d)(\bar \gamma +1)}>0.$$ This bound is further uniform in $\psi \in U(\kappa+2, B)$ as (\ref{fishinvbd}), (\ref{ks}) imply that $\bar \psi$ and then also $p_N(\bar \psi)$ are uniformly bounded in $L^2$. Similarly $J_N^{-\kappa/d} = o(1)$ and also $$\sqrt N \tilde \delta_N(0) J_N^{-\kappa/d} = N^{\frac{\frac{\gamma}{\bar \gamma+1} + \frac{d}{2} - \kappa}{2\gamma+d}}=o(1)$$ in view of (\ref{deltan}), (\ref{gs}), (\ref{ks}) and our assumption on $\gamma, \kappa$, so that using also (\ref{approxh}) and again continuity of $\mathbb I_{\theta_0}$, we can replace (\ref{getthere}) by
\begin{equation}
-t \sqrt N \langle \mathbb I_{\theta_0} (\theta - \theta_0), \mathbb I_{\theta_0} \bar \psi \rangle_{L^2} + \frac{t^2}{2} \|\mathbb I_{\theta_0}[\bar \psi]\|_{L^2}^2 + o(1).
\end{equation}
To conclude let us write $$\hat W_N = \frac{t}{\sqrt N} \sum_{i=1}^N \langle  \mathbb I_{\theta_0}p_{N} (\bar \psi_{\theta_0})(X_i), \varepsilon_i \rangle_{W}$$ for the first term in (\ref{eins}). Then using the definitions (\ref{postlog}), (\ref{hatpsi}) as well as the identity (\ref{keyfishinv}) with $h = \theta-\theta_0$, we obtain 
\begin{align*}
&E^{\Pi^{\bar \Theta_N}}\big[\exp\{t \sqrt N \big(\langle \theta, \psi\rangle_{L^2} - \hat \Psi_{N} \big)\}|Z^{(N)}\big] \\
&= \frac{\int_{\bar \Theta_N} e^{t \sqrt N \langle \theta - \theta_0, \psi\rangle_{L^2} - \hat W_N + \ell_N(\theta) - \ell_N(\theta_{(t,\psi)})} e^{\ell_N(\theta_{(t, \psi)})}d\Pi(\theta)}{\int_{\bar \Theta_N} e^{\ell_N(\theta)} d\Pi(\theta)} \\
&= e^{\frac{t^2}{2} \|\mathbb I_{\theta_0}(\bar \psi)\|_{L^2(\mathcal X)}^2} \times \frac{\int_{\bar \Theta_N} e^{\ell_N(\theta_{(t, \psi)})}d\Pi(\theta)}{\int_{\bar \Theta_N} e^{\ell_N(\theta)}d\Pi(\theta)} \times e^{r_N(t,\psi)},
\end{align*}
for a sequence $r_N(t,\psi)$ with the desired properties, so the result follows.
\end{proof}

The following lemma contains the main technical work to control the stochastic remainder terms in (\ref{eins}), (\ref{zwei}) in the previous proof. It is based on applying inequalities from empirical process theory (Ch.3 in \cite{GN16}). We need to upgrade the proofs of Lemmata 4.1.7 and 4.1.8 in \cite{N23} to include the additional uniformity in $\psi$, and to hold under our weaker hypothesis (\ref{zetabus}). 

\begin{lemma}\label{oldtimer} Let $\kappa$ be as in (\ref{ks}) and $\bar \Theta_{N,L}$ as in (\ref{barthetan}) with $\bar \gamma$ as in (\ref{gs}).  We have for every fixed $t \in \R$ and $L>0$ that $$E\sup_{\theta \in \bar \Theta_{N,L}, \psi \in U(\kappa+2,1)} \Big|\sum_{i=1}^N \langle \varepsilon_i, \G(\theta_{(t, \psi)})(X_i) - \G(\theta_0)(X_i) - D\G_{\theta_0}(X_i)[\theta_{(t, \psi)} - \theta_0] \rangle_W\Big| = o(1)$$ as well as $$E\sup_{\theta \in \bar \Theta_{N,L}, \psi \in U(\kappa+2,1)}\Big|\sum_{i=1}^N\big[|\G(\theta_0)(X_i)-\G(\theta_{(t,\psi)})(X_i)|_W^2 - E \big|\G(\theta_0)(X_i)-\G(\theta_{(t,\psi)})(X_i)|_W^2\big]\Big| = o(1).$$
\end{lemma}
\begin{proof}
Let us deal with the first supremum: setting 
\begin{align*}
g_{\theta,\psi} &=  \G\big( \theta - \frac{t}{\sqrt N} p_N(\bar \psi_{\theta_0})\big) - \G(\theta_0)- D\G_{\theta_0}\Big[ \theta - \frac{t}{\sqrt N} p_N(\bar \psi_{\theta_0}) - \theta_0\Big].
\end{align*}
we need to bound
\begin{equation}\label{supa}
E \sup_{\theta \in \bar \Theta_{N,L}, \psi \in U(\kappa+2, 1)}\big|\sum_{i=1}^N \langle \varepsilon_i, g_{\theta,\psi}(X_i)\rangle_W \big| = E \sup_{\theta \in \bar \Theta_{N,L}, \psi \in U(\kappa+2, 1)}\big|\sum_{i=1}^N f_{\theta,\psi}(Z_i) \big| 
\end{equation}
where the $Z_i$ are iid copies of the random variable $Z=(\varepsilon_1, X_1)$ on $\R \times [0,T] \times \Omega$ from after (\ref{modelG}), and where $$f_{\theta,\psi}(z)= \langle e, g_{\theta,\psi}(x) \rangle_W,~~z=(e,x) \in \R \times [0,T] \times \Omega,$$ are centred random variables, $Ef_{\theta, \psi}(Z)=0$. Let us set $W=\mathbb R$ for simplicity, the general case is proved after only notational adjustment. By Condition \ref{gemol}C we can bound the `weak variances' uniformly in $\psi, \theta$ as
$$Ef_{\theta, \psi}^2(Z) \le C \big\| \theta - \theta_0 - \frac{t}{\sqrt N} p_N(\bar \psi_{\theta_0}) \big\|_{L^2(\Omega)}^4 \lesssim \tilde \delta^4_N(\zeta) + \frac{1}{N^2}  \equiv \sigma_N^2,$$ where we note that $\psi \in U(\kappa+2, B)$ implies that $\bar \psi$ and then also $p_N(\bar \psi)$ are uniformly bounded in $H^{\kappa} \subset L^2$ by (\ref{fishinvbd}), (\ref{ks}). [This bound actually holds for $\zeta=0$ but the chaining inequality to be used below with envelopes $F_N$ would not allow us to take advantage of it because of the `Poissonian' regime of Bernstein's inequality.] Further from Condition \ref{gemol}D and again (\ref{ks}) we have the uniform bound 
$$\|g_{\theta,\psi}\|_\infty \lesssim \| \theta - \theta_0\|_{H^{\zeta}} + \frac{t}{\sqrt N} \|p_N(\bar \psi_{\theta_0})\|_{H^{\zeta}} \lesssim \tilde \delta_N(\zeta) + \frac{1}{\sqrt N} \le v \tilde \delta_N(\zeta),~~ \text{ some } v>0.$$ We can hence take point-wise strictly positive `envelope' functions $$|f_{\theta,\psi}(e,x)|\le F_N(e,x) \equiv v (1+|e|) \tilde \delta_N(\zeta),~~ e \in \R, x \in [0,T] \times \Omega,$$ and introduce, for any finitely supported probability measure $Q$ on $\R \times [0,T] \times \Omega$, the constants $$s^2_Q= E_Q[(1+|e|)^2],~~\|F_N\|_{L^2(Q)} = s_Q v \tilde \delta_N(\zeta).$$ Then for any $\theta \in \bar \Theta_{N,L}, \psi \in U(\kappa+2, 1)$, and using Condition \ref{gemol}D,
\begin{align*}
\|f_{\theta, \psi} - f_{\theta', \psi'}\|_{L^2(Q)} &\le \|f_{\theta, \psi} - f_{\theta, \psi'}\|_{L^2(Q)} + \|f_{\theta, \psi'} - f_{\theta', \psi'}\|_{L^2(Q)} \\
&\lesssim s_Q\|g_{\theta, \psi} - g_{\theta, \psi'}\|_{\infty} + s_Q\|g_{\theta, \psi'} - g_{\theta', \psi'}\|_{\infty} \\
&\lesssim \frac{s_Q}{\sqrt N}\|p_N(\bar \psi - \bar \psi')\|_{H^\zeta} + s_Q\|\theta-\theta'\|_{H^\zeta}\\
&\le \frac{1}{\sqrt N v\tilde \delta_N(\zeta)} \|F_N\|_{L^2(Q)} \|\bar \psi - \bar \psi'\|_{H^\zeta} + \frac{1}{v\tilde \delta_N(\zeta)}\|F_N\|_{L^2(Q)}\|\theta-\theta'\|_{H^\zeta}.
\end{align*}
The logarithm of the $u$-covering numbers of $p_N(U(\kappa, B)) \times U(\bar \gamma, L)$ by balls of $H^\zeta$-radius at most $u$ can be bounded by $(A/u)^{d/(\min(\kappa, \bar \gamma)-\zeta)}$ for some $A>0$, in view of Proposition A.3.1 in \cite{N23}, and so using Theorem 3.5.4 in \cite{GN16} with entropy functional $$\mathscr J_N(s) \lesssim \int_0^s \Big(\frac{1}{v \tilde \delta_N(\zeta)u}\Big)^{\frac{d}{2\beta}}du,~0<s \le 1, ~\beta=\min(\kappa, \bar \gamma)-\zeta$$ and $U_N= v\tilde \delta_N(\zeta) \max_{i\le N} (1+|\varepsilon_i|)$ (such that $EU_N^2 \lesssim \tilde \delta^2_N(\zeta) \log N$ in view of Lemma 2.3.3 in \cite{GN16}) we can estimate (\ref{supa}) from above by a constant multiple of
\begin{equation}
\sqrt N \max \Big[\tilde \delta_N(\zeta) \mathscr J_N\big(\frac{\sigma_N}{\tilde \delta_N(\zeta)}\big), \frac{\sqrt {\log N} \tilde \delta^3_N(\zeta) \mathscr J_N^2(\sigma_N/\tilde \delta_N(\zeta))}{\sqrt N \sigma_N^2}\Big].
\end{equation}
Straightforward calculations show that the first term in the preceding maximum is of order
$$\sqrt N \int_0^{\sigma_N} u^{-d/2\beta}du \lesssim \sqrt N \tilde \delta_N^{2- d/\beta}(\zeta)$$ while the second is of order
$$\sqrt {\log N} \tilde \delta^{1-2d/\beta}_N(\zeta).$$
We see from this that we require $\beta=\min(\kappa, \bar \gamma)-\zeta >2d$ for this term to converge to zero, and then we also have $$\sqrt N \tilde \delta^{2- d/\beta}_N(\zeta) \to 0 ~\text{ since }~ \tilde \delta^3_N(\zeta) = o(N^{-1}),$$ in view of (\ref{ks}), (\ref{gs}), (\ref{deltan}) and our hypotheses on $\gamma$, after some calculations. 

The second empirical process in the lemma is of smaller order of magnitude than (\ref{supa}) and treated by similar (in fact simpler) arguments, as in the proof of Lemma 4.1.8 in \cite{N23}, and left to the reader.
\end{proof}

To proceed, the ratio of integrals featuring in Theorem \ref{laplaceapp} can be written as
\begin{equation}\label{postbus00}
\frac{\int_{\bar \Theta_N} e^{\ell_N(\theta_{(t, \psi)})}d\Pi(\theta)}{\int_{\bar \Theta_N} e^{\ell_N(\theta)}d\Pi(\theta)}  = \frac{\int_{\bar \Theta_N} e^{\ell_N(\theta_{(t,\psi)})} \frac{d\Pi(\theta)}{d\Pi_t(\theta)} d\Pi_t(\theta)}{\int_{\bar \Theta_N} e^{\ell_N(\theta)}d\Pi(\theta)}
\end{equation}
where $\Pi_t$ is the Gaussian law of $\theta + tp_N(\bar \psi_{\theta_0})/\sqrt N$ with $\theta \sim \Pi$. By construction the projections $p_N(\bar \psi_{\theta_0}) $ lie in $E_{J_N} \ominus \{constants\}$ and hence in the RKHS $H_0^\gamma = \mathcal H$ of the prior from Condition \ref{priorgen}, which has `rescaled' RKHS norm $\|\cdot\|_{\mathcal H_N} = \sqrt N \delta_N \|\cdot\|_{\mathcal H}$. Applying the Cameron-Martin theorem (e.g., Theorem 2.6.13 in \cite{GN16}) we obtain an expression for the Radon-Nikodym density and in turn the bound
\begin{align*}
\Big | \log \frac{d\Pi}{d\Pi_t}(\theta) \Big |&\le \frac{|t|}{\sqrt N} \big|\langle \theta, p_{N} (\bar \psi_{\theta_0})\rangle_{\mathcal H_N}\big| + \frac{t^2}{N} \|p_{N} (\bar \psi_{\theta_0})\|^2_{\mathcal H_N}.
\end{align*}
For $\psi \in U(\kappa+2, 1)$ we know from (\ref{fishinvbd}) that $\bar \psi$ lies in $U(\kappa, c)$ for some $c>0$. Thus on the event $\bar \Theta_N$ from (\ref{barthetan}) we bound the modulus of the first term in the exponent by
$$\frac{|t|}{\sqrt N} |\langle \theta, p_N(\bar \psi_{\theta_0})\rangle_{\mathcal H_N}| \lesssim |t| \delta_N K_N \lesssim  N^{(-\gamma + \frac{d}{2} + \max(\gamma-\kappa,0))/(2\gamma+d)} \to 0$$ as $N \to \infty$,
in view of (\ref{approxrkhs}) and our hypothesis on $\gamma$. Also by (\ref{approxh}) with $r = \kappa$ and (\ref{fishinvbd})
\begin{equation}\label{quadbd}
\frac{t^2}{N} \|p_{N} (\bar \psi_{\theta_0})\|^2_{\mathcal H_N} = t^2 \delta_N^2 \|p_{N} (\bar \psi_{\theta_0})\|^2_{\mathcal H} \lesssim  \delta_N^2 \max(J_N^{(2\gamma-2\kappa)/d}, 1) \|\psi\|^2_{H^{\kappa+2}} \to 0
\end{equation}
 as $N \to \infty$. Both preceding bounds are uniform in $\psi \in U(\kappa+2, 1)$, in particular $d\Pi/d\Pi_t = 1+o(1)$ uniformly on $\bar \Theta_N$.

For shifted integration regions $$\bar \Theta_{N,t, \psi} = \{\vartheta = \theta_{(t,\psi)} : \theta \in \bar \Theta_N\}$$ so that $Law(\vartheta)=\Pi$ under $\Pi_t$, and after renormalising by $\int_{H^1} e^{\ell_N(\theta)} d\Pi(\theta)$, we further deduce that the ratio in (\ref{postbus00}) equals  
\begin{align} \label{postbus0}
(1+o(1))\frac{\int_{\bar \Theta_{N, t, \psi}} e^{\ell_N(\vartheta)}d\Pi(\vartheta)}{\int_{\bar \Theta_N} e^{\ell_N(\theta)}d\Pi(\theta)}
&=(1 +o(1)) \frac{\Pi(\bar \Theta_{N,t, \psi}|Z^{(N)})}{\Pi(\bar \Theta_{N}|Z^{(N)})}  \\
& \le (1 +o(1)) \frac{1}{\Pi(\bar \Theta_{N}|Z^{(N)})} =(1+o_{P_{\theta_0}^\mathbb N}(1)) \label{postbus000}
\end{align}
for fixed $t$ and using Theorem \ref{postcont}, an upper bound that is uniform in $\psi$. 

Finally, for any fixed $\psi \in L^2_0 \cap C^\infty$ one has 
\begin{equation}\label{shifpost}
\Pi(\bar \Theta_{N,t, \psi}|Z^{(N)}) \to 1
\end{equation}
in probability: To see this notice that the four defining inequalities in (\ref{barthetan}) of Theorem \ref{postcont} are preserved after adding the perturbation $tp_N(\bar \psi)/\sqrt N$ and increasing the constant $L$ to $2L$; indeed since $\psi$ is smooth so is $\bar \psi$ in view of (\ref{fishinvbd}), and so are its projections onto $E_J$, hence $N^{-1/2}\|p_N(\bar \psi)\|_{H^\xi} \lesssim N^{-1/2} = o(\tilde \delta_N(\xi))$ for $0 \le \xi \le \bar \gamma.$ Combined with Condition \ref{gemol}B), (\ref{quadbd}), the limit (\ref{shifpost}) now follows from Theorem \ref{postcont}. Therefore, for smooth $\psi \in L^2_0$ the ratio in (\ref{postbus0}) is
\begin{equation} \label{postlim}
\frac{\Pi(\bar \Theta_{N,t, \psi}|Z^{(N)})}{\Pi(\bar \Theta_{N}|Z^{(N)})} \to^{P_{\theta_0}^N} 1,
\end{equation}
which will be required for the convergence of the finite-dimensional distributions below.

\subsubsection{Convergence in function space}

We initially prove Theorem \ref{mainbvm} with a centring $\hat \theta_N$ that is different from $\tilde \theta_N=E^\Pi[\theta|Z^{(N)}]$: Evaluating $\hat \Psi_N=\hat \Psi_{N,j}$ from (\ref{hatpsi}) at all the trigonometric polynomials $\psi = e_j, j \ge 1,$ from (\ref{eftrig}), we obtain a (via $Z^{(N)}$ random) Fourier series 
\begin{equation} \label{hatter}
\hat \theta_N = \sum_{j \ge 1}\hat \Psi_{N,j} e_j.
\end{equation}
Note that here $\bar e_j=\mathcal I^{-1}\Delta e_j$ so $\bar e_j$ is \textit{not} the complex conjugate of $e_j$, and that $\hat \theta_N$ is real since $\theta_0, \varepsilon_i$ are. Using independence, Conditions \ref{gemol}C, F, (\ref{weyl}) and $k>\eta_0+d/2$ we have for some $C>0$
\begin{align}\label{momhat}
NE_{\theta_0}^\mathbb N \|\hat \theta_N - \theta_0\|^2_{H^{-k-2}} & \lesssim \frac{1}{N} \sum_{j=1}^\infty \lambda_j^{-k-2}\sum_{i=1}^N E_{\theta_0}^\mathbb N  |\langle \varepsilon_i, \mathbb I_{\theta_0}(p_N(\bar \psi_j))(X_i)\rangle_W|^2  \\
&=  \sum_j  \lambda_j^{-k-2} \|\mathbb I_{\theta_0}[p_N(\bar e_j)]\|_{L^2(\mathcal X)}^2  \lesssim  \sum_j  \lambda_j^{-k-2} \|\bar e_j\|_{L^2(\Omega)}^2  \notag \\
& = \sum_j \lambda_j^{-k} \|\mathcal I^{-1}e_j\|_{L^2}^2 \lesssim \sum_j \lambda_j^{-k} \|e_j\|_{H^{\eta_0}}^2 \lesssim \sum_j \lambda_j^{-k+\eta_0} <\infty \notag
\end{align}
which in particular implies that $\hat \theta_N$ defines a random variable in $H_0^{-k-2}$. 

Recalling Proposition \ref{wassconv} with $T_N = \hat \theta_N$, we now study the `localised' posterior, specifically the resulting conditional laws
\begin{equation}\label{bartau}
\bar \tau_N = Law \big(\sqrt N (\theta - \hat \theta_N)|Z^{(N)}\big), ~\theta \sim \Pi^{\bar \Theta_N}(\cdot|Z^{(N)}),
\end{equation}
 in $H^{-k-2}$. Further recall the law $\mathcal N_{\theta_0}$ from Proposition \ref{limmeas}. 
 
 \begin{proposition}
 Let $W_1$ be the Wasserstein distance on $H^{-k-2}$ from (\ref{wassdef}). We have  $$W_1(\bar \tau_N, \mathcal N_{\theta_0}) \to^{P_{\theta_0}^\mathbb N} 0 ~~\text{ as } N \to \infty.$$
 \end{proposition}
\begin{proof} 
For $J \in \N$ to be chosen, approximation spaces $E_J$ from (\ref{ej}), and $Z_N, Z$ random variables in $H^{-k-2}$ with laws $\bar \tau_N, \mathcal N_{\theta_0}$ (conditional on the data vector $Z^{(N)}$), respectively, we define projected laws $$\bar \tau_{N,J}= Law ((P_{E_J}(Z_N)), ~~n_J = Law(P_{E_J}(Z)).$$ From the triangle inequality, (\ref{wassdef}) and the Cauchy-Schwarz inequality we have 
\begin{align} \label{wasser3e}
W_1(\bar \tau_N, \mathcal N_{\theta_0}) &\le W_1(\bar \tau_{N,J}, n_J) + W_1(\bar \tau_N, \bar \tau_{N,J}) + W_1(n_J, \mathcal N_{\theta_0}) \notag \\
&\le W_1(\bar \tau_{N,J}, n_J) + (E\|P_{E_J}(Z_N)- Z_N\|^2_{H^{-k-2}})^{\frac{1}{2}} + (E \|P_{E_J}(Z) -Z\|^2_{H^{-k-2}})^{\frac{1}{2}} \notag \\
&= a + b +c.
\end{align}
 Let $\epsilon>0$ be given. We will show that each of the terms a), b), c) can be made less than $\epsilon/3$ by appropriate choice of $J$, all $N$ large enough, and with $P_{\theta_0}^\mathbb N$-probability at least $1-\epsilon$, so that the proposition follows.

\smallskip

Term c): Since $\Pr(\|Z\|_{H^{-a}} <\infty)=1$ for $a>1+d/2$ by Proposition \ref{limmeas}, Fernique's theorem (e.g., Theorem 2.1.20 and Exercise 2.1.5 in \cite{GN16}) implies that $E\|Z\|^2_{H^{-a}}<\infty$. Thus for any $\epsilon>0$ we can choose $J=J(\epsilon, k, d)$ fixed but finite so that
\begin{align} \label{gausstight}
E \|P_{E_J}(Z) -Z\|^2_{h^{-k-2}} \lesssim E\sum_{j>J} \lambda_j^{-k-2}|\langle Z, e_j \rangle_{L^2}|^2 \le \lambda_J^{a-k-2} E\|Z\|^2_{H^{-a}} <(\epsilon/3)^2
\end{align}
follows, using also $k+2>a$ and (\ref{ks}), (\ref{weyl}) in the last inequality.

\smallskip

Term b): Theorem \ref{laplaceapp} with choice $\psi= \lambda_j^{-(\kappa+2)/2}e_j \in U(\kappa+2, 1)$ 
and corresponding centring $$\hat \Psi_{N, j, \kappa}= \lambda_j^{-(\kappa+2)/2} \hat \Psi_{N,j}$$ from (\ref{hatpsi})
imply, using also (\ref{postbus000}), that for some $C,c>0$ and all $j \ge 1$, with $P_{\theta_0}^\mathbb N$-probability converging to one,
\begin{equation}\label{fullbd}
E^{\Pi^{\bar \Theta_N}}\Big[\exp\big\{\sqrt N Re\big(\langle \theta, \lambda_j^{-(\kappa+2)/2}e_j \rangle_{L^2} - \hat \Psi_{N, j, \kappa} \big)\big\}|Z^{(N)}\Big] \le c e^{\frac{1}{2} \|\mathbb I_{\theta_0}(\lambda_j^{-(\kappa+2)/2}\bar e_j)\|_{L^2(\mathcal X)}^2} \le C,
\end{equation}
where in the last inequality we used $\sup_{j \ge 1}\|\mathbb I_{\theta_0}(\lambda_j^{-(\kappa+2)/2}\bar e_j)\|_{L^2}^2 \le c<\infty$, proved just as in (\ref{momhat}). Then since $x^2 \le 2e^x$ for $x \ge 0$, and applying the previous bound also with opposite signs, and for imaginary parts, we can choose $J=J(\epsilon, k, \kappa)$ large enough such that with $P_{\theta_0}^\mathbb N$-probability as close to one as desired,
\begin{align}\label{fullbd1}
E \|P_{E_J}(Z_N) -Z_N\|^2_{H^{-k-2}} &=  \sum_{j>J} \lambda_j^{\kappa+2-k-2}  E^{\Pi^{\bar \Theta_N}}\big[N|\langle \theta, \lambda_j^{-(\kappa+2)/2} e_j \rangle_{L^2} - \hat \Psi_{N,j, \kappa}|^2|Z^{(N)}\big] \notag \\
&\lesssim \sum_{j>J} \lambda_j^{\kappa-k}  <(\epsilon/3)^2,
\end{align}
using (\ref{weyl}) and $k> \kappa+d/2$ in view of (\ref{ks}) in the last step.

\smallskip

It remains to consider term a) for any $J$ fixed. For any conjugate-symmetric vector $a=(a_j)\in\mathbb C^J$, meaning
$a_{2m}=\overline{a_{2m-1}}$ for $m=1,\ldots,J/2$, let
\begin{equation}\label{psicw}
\begin{aligned}
\sum_{j=1}^J a_j\langle\theta,e_j\rangle_{L^2}
&=
\Big\langle\theta,\sum_{j=1}^J\overline{a_j}e_j\Big\rangle_{L^2}
\equiv \langle\theta,\psi_{J,a}\rangle_{L^2},~\sum_{j=1}^J a_j\hat\Psi_{N,j} \equiv \hat\Psi_{N,J,a},
\end{aligned}
\end{equation}
in slight abuse of $\hat\Psi$ notation. Both $\psi_{J,a}$ and
$\hat\Psi_{N,J,a}$ are real-valued. [These choices of $a$ parameterise
all real linear functionals of the conjugate-symmetric Fourier
coordinates.] Further $\psi_{J,a}$ belongs to $ C^\infty(\Omega) \cap L^2_0$ and we can apply Theorem \ref{laplaceapp} combined with (\ref{postbus0}), (\ref{postlim}) to deduce convergence of Laplace transforms
$$E^{\Pi^{\bar \Theta_N}}\big[\exp\big\{t \sqrt N \big(\langle \theta, \psi_{J,a} \rangle_{L^2} - \hat \Psi_{N, J, a} \big)\big\}|Z^{(N)}\big] \to e^{\frac{t^2}{2} \|\mathbb I_{\theta_0}(\bar \psi_{J,a})\|_{L^2(\mathcal X)}^2}$$
for all $t \in \R$, as $N \to \infty$ and in $P_{\theta_0}^\mathbb N$-probability. By Exercise 4.3.3 in \cite{N23} and the Cramer-Wold device (\cite{vdvaart1998}, p.16), this implies that as $N \to \infty$,
\begin{equation}\label{weakj0}
\beta(Law(P_{E_J}(Z_N)), Law(P_{E_J}(Z))) \to^{P_{\theta_0}^\N} 0,
\end{equation}
where $\beta$ is any metric (Ch.~11.3 in \cite{D02}) for weak convergence of laws on $\R^{J} \approx E_{J}$. To upgrade this to convergence in the Wasserstein distance $W_{1, E_J}$ we first show
\begin{equation} \label{uiui}
E\|P_{E_J}(Z_N)\|_{E_{J}}^2 = O_{P_{\theta_0}^\mathbb N}(1),~~J \in \mathbb N \text{ fixed}.
\end{equation}
In view of the inequality $\|h\|_{E_{J}} \le c(J, k) \|h\|_{h^{-k-2}}$ for $h \in E_{J}$ it suffices to bound $$NE^{\Pi^{\bar \Theta_N}}\big[\|\theta - \hat \theta_N\|^2_{H^{-k-2}}|Z^{(N)}\big] \le 2 N\|\hat \theta_N - \theta_0\|^2_{h^{-k-2}} + 2NE^{\Pi^{\bar \Theta_N}}\big[\|\theta - \hat \theta_N\|^2_{H^{-k-2}}|Z^{(N)}\big].$$
The first term is bounded in $P_{\theta_0}^\mathbb N$-probability in view of (\ref{momhat}) and Markov's inequality. The second term can be estimated just as in (\ref{fullbd1}) but extending the series over all indices $j$ (so with $J=0$), and is hence also $O_{P_{\theta_0}^\mathbb N}(1)$, so that (\ref{uiui}) holds. In particular we deduce that for all $R>0$
\begin{equation}\label{rtail}
E\|P_{E_J}(Z_N)\|_{E_{J}}1_{\|P_{E_J}(Z_N)\|_{E_{J}}>R} \le \frac{E\|P_{E_J}(Z_N)\|^2_{E_{J}}}{R} =O_{P_{\theta_0}^\mathbb N}(R^{-1}).
\end{equation}
But now (\ref{weakj0}), (\ref{rtail}) and Theorem 6.9 in \cite{V09} imply that for every $J \in \mathbb N$ and as $N \to \infty$ $$W_{1, E_{J}}(\bar \tau_{N,J}, n_{J}) \to^{P_{\theta_0}^\mathbb N}0.$$ Using also that $W_{1, H^{-k-2}}$ is Lipschitz equivalent to $W_{1,E_{J}}$ on $E_J$ for the projected laws, we conclude that for every $J \in \N$ and $\epsilon>0$ fixed we can find $N_0(\epsilon, J)$ large enough such that for $N \ge N_0$,
\begin{equation}\label{fidilimit}
P_{\theta_0}^\mathbb N\big(W_1(\bar \tau_{N,J}, n_{J}) >\epsilon/3)<\epsilon,
\end{equation}
completing the proof of convergence to zero in probability of $W_1(\bar \tau_N, \mathcal N_{\theta_0})$ from (\ref{wasser3e}). 
\end{proof}

\smallskip

Combining the preceding proposition with Proposition \ref{wassconv} for $T=\hat \theta_N$ implies for unrestricted posterior draws $\theta \sim \Pi(\cdot|Z^{(N)})$ that 
\begin{equation}\label{hatlim}
W_1(Law(\sqrt N (\theta - \hat \theta_N)|Z^{(N)}), \mathcal N_{\theta_0}) \to^{P_{\theta_0}^\mathbb N} 0
\end{equation}
as $N \to \infty$, that is, Theorem \ref{mainbvm} holds for the centering $\hat \theta_N$ instead of $\tilde \theta_N=E^\Pi[\theta|Z^{(N)}]$. But convergence in $W_1$ implies convergence of moments by Theorem 6.9 in \cite{V09}, so we obtain $$\sqrt N (E^\Pi[\theta |Z^{(N)}] -\hat \theta_N)  \to^{P_{\theta_0}^\mathbb N} E_{\mathcal N_{\theta_0}}Z =0 \text{ in } H^{-k-2},$$ 
and then also
\begin{equation*}
W_1\big(Law(\sqrt N (\theta - \hat \theta_N)), Law(\sqrt N(\theta- \tilde \theta_N))\big) \le \sqrt N \|E^\Pi[\theta|Z^{(N)}] - \hat \theta_N\|_{H^{-k-2}} = o_{P_{\theta_0}^\mathbb N}(1),
\end{equation*}
which combined with (\ref{hatlim}) completes the proof of the first limit in Theorem \ref{mainbvm}. 

\smallskip

\textbf{Asymptotics for the posterior mean:} To prove the second limit in Theorem \ref{mainbvm} note that by the last display, the limit distribution of $\sqrt N (\tilde \theta_N - \theta_0)$ in $H^{-k-2}$ coincides with the limit distribution of the (real-valued) Fourier series $$Q_N=\sqrt N (\hat \theta_N - \theta_0) = \frac{1}{\sqrt N} \sum_{j\ge 1}\sum_{i=1}^N \langle \varepsilon_i, \mathbb I_{\theta_0} p_N (\bar e_j)(X_i)\rangle_W e_j,$$ recalling also (\ref{hatpsi}) and (\ref{hatter}).  Fix $\psi = \psi_{J,a} \in L^2_0 \cap C^\infty$ as in (\ref{psicw}) so that the resulting $\bar \psi$ from (\ref{barpsi}) also lies in $L^2_0 \cap C^\infty$ in view of (\ref{fishinvbd}). By Chebyshev's inequality and Condition \ref{gemol}C), we have as $N \to \infty$,
\begin{align*}
E^\mathbb N_{\theta_0} \Big[\frac{1}{\sqrt N} \sum_{i=1}^N \langle \varepsilon_i, \mathbb I_{\theta_0} p_N (\bar \psi)(X_i) - \mathbb I_{\theta_0} (\bar \psi)(X_i) \rangle_W\Big]^2  &\le \|\mathbb I_{\theta_0} p_N (\bar \psi_{\theta_0}) -  \mathbb I_{\theta_0} \bar \psi_{\theta_0}\|^2_{L^2(\mathcal X)}\\
& \lesssim \|p_N (\bar \psi_{\theta_0})-\bar \psi_{\theta_0}\|_{L^2(\Omega)} \to 0.
\end{align*}
Therefore Markov's inequality, Theorem 2.7 in \cite{vdvaart1998} and the central limit theorem imply 
\begin{equation}\label{clt}
\langle Q_N , \psi_{J,a} \rangle_{L^2}  \to^d N\big(0,\|\mathbb I_{\theta_0} \bar \psi_{J,a} \|_{L^2}^2\big)
\end{equation}
for any fixed vector $a$. By the Cramer-Wold device (\cite{vdvaart1998}, p.16) this implies convergence of finite-dimensional distributions 
\begin{equation} \label{cltj}
P_{E_J}(Q_N) \to^d P_{E_J}(Z), Z \sim \mathcal N_{\theta_0},~\text{ in } E_{J}
\end{equation}
for every $J$ fixed and $P_{E_J}$ from (\ref{ej}). Writing $q_{N,J}=Law(P_{E_J}(Q_N)), n_J = Law (P_{E_J}(Z))$, we then decompose  as in (\ref{wasser3e}) but now with $\beta$ the bounded Lipschitz metric for weak convergence (\cite{D02}, Theorem 11.3.3) in $H^{-k-2}$,
\begin{align*}
& \beta(Law (Q_N), \mathcal N_{\theta_0}) \le \beta (q_{N,J}, n_J) + \beta (n_J, \mathcal N_{\theta_0}) + \beta(q_{N,J}, Law (Q_N)) \\
& \le \beta (q_{N,J}, n_J) + (E\|P_{E_J}(Z) - Z\|^2_{H^{-k-2}})^{\frac{1}{2}} + (E\|P_{E_J}(Q_N)-Q_N\|^2_{H^{-k-2}})^{\frac{1}{2}}.
\end{align*}
Given $\epsilon>0$ we can bound the middle term in the last line by $\epsilon/3$ as in (\ref{gausstight}), for $J=J(\epsilon)$ large enough. For the last term we can repeat the arguments from (\ref{momhat}) to obtain the bound $\sum_{j>J}\lambda_j^{-k+\eta_0}<\epsilon/3$, again for $J$ large enough. Finally the first term converges to zero as $N \to \infty$ by (\ref{cltj}) and hence can be made less than $\epsilon/3$ as well for $N$ large enough. Conclude that as $N \to \infty$, $$\sqrt N (\hat \theta_N - \theta_0) \to^d Z \sim \mathcal N_{\theta_0}~~\text{ in } H^{-k-2},$$ as desired. The proof of Theorem \ref{mainbvm} is complete.

\subsection{Application to the reaction diffusion model and proof of Theorem \ref{ganzwien}}

We now apply Theorem \ref{mainbvm} to the forward map $\G(\theta) \equiv u_\theta$ arising from the reaction-diffusion system (\ref{evol}).

\begin{theorem}\label{bvminv}
Assume $\theta_0 \in C^\infty(\Omega) \cap L^2_0(\Omega)$. Then $\G$ satisfies Condition \ref{gemol} for  $$\zeta>d/2, \bar \gamma>3+d/2, \eta_0>2+d/2, d \le 3, W=\R.$$ As a consequence, for prior $\theta \sim \Pi$ as in Condition \ref{priorgen} with integer $\gamma>2+3d$, posterior law of $\theta|Z^{(N)}$ arising as in (\ref{postlog}), posterior mean $\tilde \theta_N = E^\Pi[\theta|Z^{(N)}]$, and if $\mathcal N_{\theta_0}$ is the Gaussian measure from Proposition \ref{limmeas}, then we have
$$W_1(Law(\sqrt N(\theta-\tilde \theta_N)|Z^{(N)}), \mathcal N_{\theta_0}) \to^{P_{\theta_0}^\mathbb N} 0$$
where  $W_1=W_{1, H^{-k-2}}$ is the Wasserstein distance (\ref{wassdef}) on $H^{-k-2}(\Omega)$ for any $k>3d$.
\end{theorem}
\begin{proof}
The proof is based on verifying Condition \ref{gemol} of Theorem \ref{mainbvm} and PDE theory developed in Sec.~\ref{pdean}. Here are the pointers: Conditions A), B) are the subject of Corollary \ref{yetagain}, while Condition C) is verified in Theorem \ref{linrob}. The inequality (\ref{zetabus}) in Condition D) follows from (\ref{h2estimate}) and the continuity of $\mathbb I_{\theta_0}$ required in D) follows from (\ref{izeta}). The stability estimate E) follows from (\ref{dochparis}), while  F) is the content of Theorem \ref{raf}.
\end{proof}

While $k>3d$ is not necessarily optimal, Theorem \ref{bvminv} cannot hold in topologies significantly stronger than that of $H^{-k-2}$ (such as, e.g., $L^2$) as convergence in law would require the limiting random variable to be tight in this function space (Thm.~11.5.4 in \cite{D02}), which $\mathcal N_{\theta_0}$ is not in view of the remarks after Proposition \ref{limmeas}. 

\smallskip

Theorem \ref{ganzwien} can be deduced from Theorem \ref{bvminv} by the following argument, which relies on linearisation and Lemma \ref{bootcamp} on the smoothing properties of solutions to the linear PDE (\ref{PDElin}). Let us define (conditional on the observations) random variables $$Y_N = \sqrt N (u_\theta - u_{\tilde \theta_N})|Z^{(N)}, ~~Z_N = \sqrt N (\theta-\tilde \theta_N)|Z^{(N)},~~ \theta \sim \Pi(\cdot|Z^{(N)}).$$ Prior and posterior draws lie in $H^{\bar \gamma}$ almost surely for some $\bar \gamma>d/2$. Similarly the Bochner integrals $\tilde \theta_N$ lie in $H^{\bar \gamma}$  on events of $P_{\theta_0}^\mathbb N$-probability approaching one by the arguments in and after (\ref{qreg}). Thus an application of Proposition \ref{roundrobin} implies that we can consider the random variables $Y_N$ as taking values in the Banach space $\mathscr C$ from (\ref{calc}). Next, Theorem \ref{linrob} provides the linearisation of $\theta \mapsto u_\theta(t)$ around $\theta_0 \in C^\infty \cap L^2_0$ as $$Du_{\theta_0}[\cdot]: L^2(\Omega) \to L^2([t_{\min}, t_{\max}], L^2(\Omega))$$ Since the time-dependent potential  $V(t, \cdot)= f'(u_{\theta_0}(t))$ lies in $C^{1,\infty}$ from (\ref{notsosmooth}) in view of Proposition \ref{roundrobin}, we can invoke Lemma \ref{bootcamp} to deduce that the linear operator $Du_{\theta_0}$ in fact extends to a continuous linear operator from $H^{-k-2}$ to $L^\infty([t_{\min}, t_{\max}], H^\eta)$ for $\eta>d/2$. We can replace $L^\infty$ by $C$ since time-continuity follows by considering solutions to (\ref{weakpdelin}) with initial conditions $h=Du_{\theta_0}(t_{\min}) \in H^\eta$, and using Proposition \ref{roundrobinlin}. By the Sobolev imbedding the map and hence the random variables $Du_{\theta_0}(Z_N)$ take values in $\mathscr C = C([t_{\min},t_{\max}], C(\Omega))$. Similarly, for $Z$ a random variable in $H^{-k-2}$ of law $\mathcal N_{\theta_0}$, the Gaussian Borel random variable $Y=Du_{\theta_0}(Z)$ takes values in the space $\mathscr C$. By the remarks after (\ref{weakpdelin}) its law coincides with the law of the unique solution of the stochastic PDE (\ref{linshow}).

\smallskip

To prove Theorem \ref{ganzwien}, recall from (\ref{wassdef}) that the Wasserstein distance $\mathscr W_1(\mu, \nu)$  involves suprema over $1$-Lipschitz functions $F: \mathscr C \to \R$. Then we have
\begin{align*}
&\mathscr W_1(Law(Y_N), Law (Y)) = \sup_{\|F\|_{Lip}\le 1} |EF(\sqrt N (u_\theta - u_{\tilde \theta_N})) -  EF(Y)| \\
&= \sup_{\|F\|_{Lip}\le 1} \Big|E\big[F\big(\sqrt N (u_\theta - u_{\tilde \theta_N}) - Du_{\theta_0}(Z_N) + Du_{\theta_0}(Z_N)\big) - F(Du_{\theta_0}(Z_N))\big]\\
&~~~~~~~~~~~~~~  +EF(Du_{\theta_0}(Z_N))-  EF(Y)\Big| \\
&\le \sqrt N E\big\|u_\theta - u_{\tilde \theta_N} - Du_{\theta_0}[\theta - \tilde \theta_N]\big\|_\mathscr C + \sup_{\|F\|_{Lip}\le 1}|EF(Du_{\theta_0}(Z_N))-  EF(Du_{\theta_0}[Z])| \\
& = I+II.
\end{align*}
From what precedes the linear map $Du_{\theta_0}$ is continuous from $H^{-k-2}$ to $\mathscr C$ and hence Lipschitz, and therefore the composite map $G \equiv F \circ Du_{\theta_0}$ is Lipschitz from $H^{-k-2}(\Omega) \to \R$. We therefore have from (\ref{wassdef}) that $$II \lesssim \sup_{\|G\|_{Lip}\le 1}|EG(Z_N)-  EG(Z)| = W_{1, H^{-k-2}}(Law(\sqrt N(\theta-\tilde \theta_N)|Z^{(N)}), \mathcal N_{\theta_0})$$ which converges to zero in $P_{\theta_0}^\mathbb N$-probability as $N \to \infty$ by Theorem \ref{bvminv}.

\smallskip

To show that the first term converges to zero we use the following argument: Introduce the event $A=\{\theta: \|\theta\|_{H^{\bar \gamma}} \le M\}$ for $M$ large enough to be chosen. We show first that
$$\sqrt N E\big[ \sup_{t \in [t_{min}, t_{max}] } \|u_\theta(t, \cdot) - u_{\tilde \theta_N}(t, \cdot) - Du_{\theta_0}[\theta - \tilde \theta_N]\|_{\infty}1_{A^c}\big] $$ converges to zero in $P_{\theta_0}^\mathbb N$-probability, where we write $E, \Pr$ for posterior expectation and probability. We know from Theorem \ref{postcont} that $\Pr(A^c) \le e^{-bN\delta_N^2}$ for $b>0$ and $M$ large enough but fixed, with $P_{\theta_0}^\N$-probability approaching one. Also, in the preceding display, only $\theta$ is random under the posterior and (\ref{meanreg}) implies that $\|\tilde \theta_N\|_{H^{\bar \gamma}}$ is uniformly bounded with frequentist probability approaching one. Thus, using also $H^{a}$-Lipschitz continuity of $Du_{\theta_0}$ (Proposition \ref{roundrobinlin}, with $\theta_0$ fixed) and the Sobolev imbedding ($d\le 3$), it suffices to bound
$$ E[\sup_{t}(\|u_\theta(t)\|_\infty + \|Du_{\theta_0}[\theta]\|_{\infty})1_{A^c}] \lesssim \big[(E \sup_t\|u_\theta(t)\|_{H^2}^2)^{1/2} + (E\|\theta\|_{H^2}^2)^{1/2} \big] \Pr (A^c)^{1/2}.$$
The last quantity is $O_{P^\mathbb N_{\theta_0}}(e^{-bN\delta_N^2/2})=o_{P^\mathbb N_{\theta_0}}(1/\sqrt N)$  by Theorem \ref{postcont} if we can show
$$\sup_{0<t<T}\|u_\theta(t)\|_{H^2}^2 \lesssim 1 + \|\theta\|_{H^2}^2$$ for all $\theta \in H^2$, with constants implicit in $\lesssim$ independent of $\theta$. To see this, note that from (\ref{doo}) below with $\gamma = 2$  and (\ref{dooag}), we have
$$\frac{d}{dt} \|u_\theta\|_{h^2}^2 \lesssim 1+\|f(u_\theta)\|^2_{h^1}  \lesssim 1+ \|u_\theta\|^2_{h^1}.$$
 This inequality can be integrated $\int_0^t$, and using also (\ref{another}) we conclude that indeed
 $$\|u_\theta\|_{h^2}^2 \lesssim 1 + \|\theta\|^2_{h^2} + \int_0^t \|\nabla u_\theta\|^2_{L^2}ds \lesssim 1 + \|\theta\|_{H^2}^2.$$ 

It remains to deal with the term involving  $1_A$, where we note that we can in view of Theorem \ref{postcont} further restrict to (frequentist) events $\{\|\tilde \theta_N\|_{H^{\bar \gamma}}\le M\}$. By the Sobolev imbedding, the interpolation inequality (\ref{interpol}) with exponent $m = (\bar \gamma - \zeta)/\bar \gamma$ for any $d/2<\zeta<\bar \gamma<\gamma-d/2$, and Propositions \ref{roundrobin} and \ref{roundrobinlin}, we obtain the uniform in $t$ bound
\begin{align*}
& \|u_\theta(t,\cdot) - u_{\tilde \theta_N}(t,\cdot) - Du_{\theta_0}(t,\cdot)[\theta - \tilde \theta_N]\|_{\infty}1_A \\
&\lesssim \|u_\theta(t,\cdot) - u_{\tilde \theta_N}(t,\cdot) - Du_{\theta_0}(t,\cdot)[\theta - \tilde \theta_N]\|_{H^\zeta}1_A\\
&\lesssim \|u_\theta(t,\cdot) - u_{\tilde \theta_N}(t,\cdot) - Du_{\theta_0}(t,\cdot)[\theta - \tilde \theta_N]\|_{L^2}^m \|u_\theta(t,\cdot) - u_{\tilde \theta_N}(t,\cdot) - Du_{\theta_0}(t,\cdot)[\theta - \tilde \theta_N]\|_{H^{\bar \gamma}}^{1-m}1_A \\
& \lesssim \big(\|\theta-\theta_0\|^2_{L^2} + \|\tilde \theta_N - \theta_0\|_{L^2}^{2}\big)^m
\end{align*}
where the first factor was estimated from above, using also Theorem \ref{linrob}, by
$$\|Du_{\theta_0}[\theta - \theta_0] - Du_{\theta_0}[\tilde \theta_N - \theta_0\big] - Du_{\theta_0}[\theta - \tilde \theta_N]\|_{L^2} + C\|\theta-\theta_0\|^2_{L^2} + C\|\tilde \theta_N - \theta_0\|_{L^2}^{2}.$$  We can use (\ref{meanreg}) to bound the (under the posterior, deterministic) variables $$\|\tilde \theta_N - \theta_0\|_{L^2}^{2} =O_{P^\mathbb N_{\theta_0}}(\tilde \delta^2_N)$$ and then moreover $$ \sqrt N E^\Pi \big[\big( \|\theta-\theta_0\|_{L^2}^{2m} + \tilde \delta^{2m}_N\big)|Z^{(N)}\big] =O_{P_{\theta_0}^\mathbb N}(\sqrt N \tilde \delta_N^{2m})=o_{P_{\theta_0}^\mathbb N} (1),$$ since under the maintained hypothesis $\gamma>2+3d$ we can  choose $\bar \gamma, \zeta$ so that $$ \tilde \delta_N^{4m} = N^{-4 \frac{\gamma}{2\gamma+d}\frac{\bar \gamma-\zeta}{\bar \gamma +1}}= o(N^{-1}).$$ This completes the proof of convergence to zero of $I$ and thus of the first claim in the theorem. The second claim about convergence of $\sqrt N (u_{\tilde \theta_N} - u_{\theta_0})$ is obtained from a similar (in fact simpler) linearisation and continuous mapping argument combined with Theorems \ref{mainbvm} and \ref{bvminv}, where the remainder term $R_N = u_{\tilde \theta_N} - u_{\theta_0} - Du_{\theta_0}[\tilde \theta_N - \theta_0]$ is $o_{P_{\theta_0}^\mathbb N}(1/\sqrt N)$ in $\mathscr C$ in view of (\ref{meanreg}).

\section{PDE results}\label{pdean}

In this section we derive various analytical results needed for the proof of Theorem \ref{ganzwien}. We refer to \cite{R01, E10} for relevant background material on the theory of parabolic PDEs that we will rely upon.

\subsection{Periodic non-linear reaction diffusion equations}

We first develop appropriate theory for the non-linear equation (\ref{evol}) underlying the dynamical system featuring in Theorem \ref{ganzwien}. For $\Omega =[0,1]^d$ and any $T>0$ we consider periodic solutions $u$ to the PDE
\begin{align}\label{PDE}
\frac{\partial}{\partial t} u - \Delta u &= f(u)~~\text{on } (0,T] \times \Omega, \notag \\
u(0, \cdot) &= \theta~~\text{on } \Omega,
\end{align}
where $f: \R \to \R$ is a smooth function of compact support, $f \in C^\infty_c(\R)$, where $\Delta = \sum_{j=1}^d \partial^2/(\partial x_j)^2$ is the Laplacian, and $\theta =u(0,\cdot) \in H^1(\Omega)$ is an initial condition. We note that $f$ acts non-linearly via point-wise composition, $f(u)(x) = f(u(x)), x \in \Omega$. 

\smallskip

For $(H^1(\Omega))^*=H^{-1}(\Omega)$, a weak solution to (\ref{PDE}) is a map $u \in L^2([0,T], H^1(\Omega)) \cap C([0,T], L^2(\Omega))$ such that $(\partial/\partial t)u = u' \in L^2([0,T], H^{-1}(\Omega))$ satisfies
\begin{align} \label{weakpde}
 \langle u', v \rangle_{L^2} + \langle \nabla u, \nabla v \rangle_{L^2} &= \langle f(u), v \rangle_{L^2},~~\forall v \in H^1 \text{ and a.e. } t \in (0,T] \notag \\
u(0) &= \theta,
\end{align}
noting that $f(u)$ is uniformly bounded. We call $u$ solving (\ref{weakpde}) a \textit{strong} solution if in addition $$u \in L^2([0,T], H^2(\Omega)) \cap C([0,T], H^1(\Omega)),~~u' \in L^2([0,T], L^2(\Omega)),$$ in which case $u$ then solves (\ref{PDE}) as an equation in $L^2([0,T], L^2(\Omega))$.

\subsubsection{Existence and uniqueness of strong solutions}

The following proof of existence adapts arguments from p.537f in \cite{E10} to the present setting. 

\begin{theorem}\label{exist}
For $f \in C_c^\infty(\R)$ and any $\theta \in H^1(\Omega), d \in \mathbb N$, there exists a strong solution $u= u_\theta$ to the reaction-diffusion equation (\ref{weakpde}) that is unique in $C([0,T], L^2(\Omega))$.
\end{theorem}
\begin{proof}
Define the Banach space $X \equiv C([0,T], L^2(\Omega))$ with norm $$\|u\|_X = \sup_{0<t<T}\|u(t)\|_{L^2(\Omega)}.$$ For any $u \in X$ define a new function $h(t, \cdot)=f(u(t, \cdot)) \in L^\infty([0,T], L^\infty)$. The \textit{linear} parabolic PDE
\begin{align}\label{linpolw}
\frac{\partial}{\partial t} w - \Delta w &= h~~\text{on } (0,T] \times \Omega \notag \\
w(0,\cdot)&= \theta ~~\text{on } \Omega
\end{align}
has a unique strong solution $w \in L^2([0,T], H^2) \cap C([0,T], H^1(\Omega))$ such that $(\partial/\partial t)w \in L^2([0,T], L^2)$, see, e.g., Theorem 7.7 in \cite{R01}. Such $w$ can be represented via the usual variation of constants formula
\begin{equation}\label{varconst}
w(t,\cdot) = \sum_{j=0}^\infty e^{-\lambda_j t} \langle \theta, e_j \rangle_{L^2} e_j +  \sum_{j=0}^\infty \int_0^t e^{-(t-s)\lambda_j} \langle h(s), e_j \rangle_{L^2} e_j ds,~t>0,
\end{equation}
where $e_j$ are the orthonormal eigenfunctions of the (periodic) Laplacian from (\ref{eftrig}), and with the above sums converging in $L^2$ under the present hypotheses on $h, \theta$. For $u_i \in X, h_i=f(u_i), i=1,2,$ consider solutions $w_i \equiv A(u_i) \in X$ to (\ref{linpolw}) and define the constants $i_{1,2}$ via $$i_{1,2}(t) =  \int_\Omega (w_1 - w_2)(t) = \langle w_1(t)-w_2(t), e_0 \rangle_{L^2}.$$ Since $\langle e_j, 1 \rangle_{L^2} =0$ for $j \neq 0$, we have from Fubini's theorem and the Cauchy-Schwarz inequality
\begin{align*}
|i_{1,2}| &= \left| \int_\Omega \theta - \int_\Omega \theta  +   \int_0^t \int_\Omega (f(u_1(s,x)) - f(u_2(s, x))dx ds \right|  \\
& \le T \sup_{0<t<T} \|f(u_1) - f(u_2)\|_{L^2} \le T \|f\|_{Lip} \|u_1-u_2\|_X.
\end{align*}
Interchanging integration and differentiation (justified by standard arguments for the strong solutions $w_i$ to (\ref{linpolw}))  and using the inequalities of Cauchy-Schwarz, Young and Poincar\'e (cf.~p.151 in \cite{R01}), we have for any $0<\epsilon \le 1$ that
\begin{align*}
\frac{d}{dt} \|w_1-w_2\|_{L^2}^2 &+ 2 \|\nabla (w_1-w_2 - i_{1,2})\|^2_{L^2} \\
& = 2 \langle w_1- w_2, \frac{\partial}{\partial t} (w_1-w_2) -\Delta (w_1-w_2) \rangle_{L^2} \\
&= 2 \langle w_1-w_2, h_1 - h_2 \rangle_{L^2} \\
&\le \frac{\epsilon}{2} \|w_1-w_2 - i_{1,2} + i_{1,2}\|_{L^2}^2 + \frac{2}{\epsilon}\|f(u_1)-f(u_2)\|^2_{L^2} \\
& \le  \epsilon \|\nabla(w_1-w_2-i_{1,2})\|_{L^2}^2 + (T^2 + (2/\epsilon)) \|f\|_{Lip}^2 \|u_1-u_2\|^2_{X}.
\end{align*}
Then choosing $\epsilon$ small enough and subtracting, we deduce
$$\frac{d}{dt} \|w_1 -w_2\|_{L^2}^2 \le c (T^2+1) \|u_1-u_2\|^2_X$$ for some $c=c( \|f\|_{Lip})>0$. We can integrate $\int_0^t (\cdot)ds$ this inequality and use $w_1(0)-w_2(0)=\theta - \theta =0$ to obtain
\begin{equation}\label{conte}
\|A(u_1)-A(u_2)\|^2_X = \sup_{0<t<T}\|w_1(t)-w_2(t)\|^2_{L^2} \le c(T^2+1) T \|u_1-u_2\|^2_X.
\end{equation}
We deduce that for $T=T_0$ small enough depending only on $\|f\|_{Lip}$, the map $A$ is a contraction on $X$ and Banach's fixed point theorem (p.536, \cite{E10}) implies the existence of a unique $u \in X$ satisfying $A(u)=u$ on the interval $[0,T_0]$. Moreover, since $u=A(u)$ is in the range of the linear parabolic solution operator from above, it has the desired regularity for a strong solution. In particular $u(T_0) \in H^1$ and we can iterate the preceding argument finitely many times with initial conditions $u(T_0), u(T_1), \dots,$ etc., to extend the solutions to $[0,T]$. 
\end{proof}

\subsubsection{Regularity estimates}

We will need that regularity of initial conditions $\theta \in H^\gamma$ translates into regularity of solutions $u_\theta \in L^\infty([0,T], H^\gamma)$ for any $\gamma$ and in a quantitative way. We will employ `energy' methods via Sobolev norms (\ref{seqnorm}) with corresponding inner product $\langle \cdot, \cdot \rangle_{h^\gamma}$. Note that we have for smooth periodic $u$ and $d$-dimensional smooth periodic vector field $w$ the `divergence theorem'
\begin{align} \label{ipart0}
\langle w, \nabla u\rangle_{h^\gamma} &=  \sum_{l=1}^d\sum_{j \ge 0} (1+\lambda_j)^\gamma \langle w_l, e_j \rangle_{L^2}  \overline{\langle \partial_l u, e_j \rangle_{L^2}} \notag \\
& = -\sum_{l=1}^d\sum_{j \ge 0} (1+\lambda_j)^\gamma \langle  w_l,  e_j \rangle_{L^2}  \overline{\langle u, \partial_l e_j \rangle_{L^2}} \notag \\
&=\sum_{l=1}^d\sum_{j \ge 0} (1+\lambda_j)^\gamma\langle  w_l,  c_{j,l}e_j \rangle_{L^2}  \overline{\langle u, e_j \rangle_{L^2}} = -\langle \nabla \cdot w, u \rangle_{h^\gamma},
\end{align}
using integration by parts for the $l$-th variable and $\partial_l e_j = c_{j,l} e_j(x)$ for $c_{j,l} = 2\pi i (k_j)_l$ with $\overline{c_{j,l}} = -c_{j,l}$ and $e_j$ from (\ref{eftrig}). These identities hold for all $\gamma \in \R$ and Sobolev smooth $u,w$ as long as the action of $u,w$ on $e_j$ is appropriately defined. In particular if $w = \nabla u$ then we have
\begin{equation}\label{ipart}
\|\nabla u \|_{h^{\gamma}}^2 = -\langle \Delta u, u \rangle_{h^\gamma},~~u \in h^{\gamma+1}.
\end{equation}

The proof of the following proposition works for any $d$ as long as $\gamma \le 2$. For $\gamma>2$ in general dimensions $d>3$, energy methods would need to be replaced by Schauder estimates as in Sec.~5.2 in \cite{M89}, but we abstain from these technicalities to focus on the main ideas.

\begin{proposition}\label{roundrobin}
For $\theta \in H^1(\Omega)$ and $f \in C^\infty_c(\R)$, let $u$ be a solution of (\ref{weakpde}). Suppose $\|\theta\|_{H^\gamma} \le B$ for $\gamma \ge 0$ and if $\gamma>2$ also assume $d\le 3$. Then we have
\begin{equation}\label{guaten}
\int_0^T \|u_\theta(t)\|_{H^{\gamma+1}}^2 dt +\sup_{0 \le t \le T}\|u_\theta(t)\|_{H^\gamma} \le C,
\end{equation}
for some $C=C(T, B,f, \gamma, d)$, and, writing $u' = (\partial u/\partial t)$, we also have
\begin{equation}\label{tbound}
\int_0^T\big\|u'_\theta(t)\big\|_{H^{\gamma-1}}^2 dt + {\rm{ess}}\sup_{0\le t \le T}\|u'_\theta(t)\|_{H^{\gamma-2}}  \le C'.
\end{equation}
In particular $u_\theta \in C([0,T], H^\gamma)$ and if $\theta \in C^\infty(\Omega)$, then $u_\theta$ lies in $\cap_{b>0}C^1([0,T], C^b(\Omega))$.
\end{proposition}
\begin{proof}

We first assume the solutions $u$ are sufficiently regular and establish the a-priori estimates (\ref{guaten}), (\ref{tbound}) in Steps 1 and 2. The general case then follows from a Galerkin argument in Step 3.

\smallskip

\textbf{Step 1: Preliminary identities and estimates.}
Differentiating the squared $h^\gamma$-norm and using (\ref{ipart}) we obtain the identity
\begin{align}\label{1stbd}
\frac{1}{2} \frac{d}{dt} \|u\|_{h^\gamma}^2 + \|\nabla u\|_{h^\gamma}^2 &= \big \langle \frac{\partial}{\partial t} u, u \big\rangle_{h^\gamma} - \langle \Delta u, u\rangle_{h^\gamma} = \langle  f(u), u \rangle_{h^\gamma}.
\end{align}
Since $f(u)$ is uniformly bounded and of compact support, we know that $$\langle  f(u), u \rangle_{L^2} = \int_\Omega f(u) u 1_{|u| \le K} \le c_{f}<\infty, \text{ for some } K>0,$$ and all $u \in L^2$, so that (\ref{1stbd}) gives for some constant $c_f>0$,
\begin{equation}\label{addbd}
\frac{d}{dt}\|u(t)\|^2_{L^2} + 2 \|\nabla u(t)\|^2_{L^2} \le 2c_f
\end{equation}
which we can integrate to obtain $\sup_{0 \le t \le T}\|u(t)\|^2_{L^2} \le 2T c_f + \|u(0, \cdot)\|^2_{L^2}<\infty$ and
\begin{equation}\label{another}
 \int_0^T \|\nabla u(t)\|_{L^2}^2 dt \le Tc_f + \|u(0,\cdot)\|_{L^2}^2 \le c(B,f,T)<\infty.
 \end{equation}
Since $\|u\|^2_{H^1} \simeq \|u\|^2_{L^2} + \|\nabla u\|_{L^2}^2$ this already proves the case $\gamma =0$ in (\ref{guaten}).

Next by definition of the $h^\gamma$-norms and the Cauchy-Schwarz inequality, the r.h.s.~in (\ref{1stbd}) can be upper bounded as $$\langle  f(u), u \rangle_{h^\gamma} \le \|f(u)\|_{h^{\gamma-1}} \|u\|_{h^{\gamma+1}} \lesssim \|f(u)\|_{h^{\gamma-1}} (\|\nabla u\|_{h^{\gamma}} + \|u\|_{L^2})$$
and then using Young's inequality for products this is further bounded by
$$2\|f(u)\|^2_{h^{\gamma-1}} +  \frac{1}{4} \|\nabla u\|^2_{h^{\gamma}}+ \frac{1}{4}\|u\|^2_{L^2}.$$
We conclude from this, (\ref{1stbd}) and the bound for $\|u\|_{L^2}$ just established for $\gamma=0$ that
\begin{equation}\label{doo}
 \frac{d}{dt} \|u\|_{h^\gamma}^2 + \|\nabla u\|_{h^\gamma}^2 \lesssim \|f(u)\|^2_{h^{\gamma-1}} + c
 \end{equation}
holds for some uniform constant $c>0$ and all $\gamma \ge 1$.

\smallskip

\textbf{Step 2: Proof of a priori estimates.} We now turn to the proof of (\ref{guaten}) for $\gamma \ge 1$. We argue by induction on $\gamma$ and assume the result has been proved for $\gamma -1$. The induction hypothesis provides us with $$\|u(t)\|_{L^2} + \|u(t)\|_{h^{\gamma-1}} \le C,~ \gamma \ge 1,$$ with constant $C=C(T,B)$. We show that as a consequence the r.h.s.~in (\ref{doo}) is bounded by a fixed constant. This is clear for $\gamma =1$ and for $\gamma=2$ we have from the chain rule and for multi-index $\alpha$ that 
\begin{equation} \label{dooag}
\|f(u)\|_{h^1}\lesssim \|f(u)\|_{H^1} \lesssim \|f(u)\|_{L^2} + \max_{|\alpha|=1}\|f'(u) D^\alpha u\|_{L^2} \le c+\|f'\|_\infty \|u\|_{h^{\gamma-1}} \le c'<\infty.
\end{equation}
For $\gamma=3$ we have $\|f(u)\|_{h^2} \lesssim \|f(u)\|_{H^2}\lesssim c + \max_{|\alpha|=2}\|D^\alpha(f(u))\|_{L^2}$ and need to estimate the last term: we use the chain rule combined with the Cauchy-Schwarz and Ladyzhenskaya's inequality (\cite{R01}, p.244) to obtain, with $|\alpha_i|=1$, the bound
\begin{align}\label{critchain}
& \|f''\|_\infty \big[\int_\Omega (D^{\alpha_1}u D^{\alpha_2}u)^2  \big]^{1/2} + \|f'\|_\infty \|u\|_{H^2} \lesssim  \|D^{\alpha_1}u\|_{L^4} \|D^{\alpha_2}u\|_{L^4} + \|u\|_{H^2} \\
&\lesssim \|u\|_{H^1}^{2-d/2} \|u\|_{H^2}^{d/2}  + \|u\|_{H^2}  \lesssim \|u\|_{H^2}^{d/2} +\|u\|_{H^2} \le C<\infty. \notag
\end{align}
[This argument is for $d=2,3$ -- for $d=1$ one can argue directly via (\ref{multbasic}).] For $\gamma = 4$, we have from the chain rule and the multiplier inequality (\ref{multbasic}) for $H^{2}$-norms with $d<4$ that
\begin{align*}
\|f(u)\|_{H^3} &\lesssim c+ \max_{0 \le |\alpha| \le 1}\|f'(u) D^\alpha u\|_{H^2} \lesssim c+\|f'(u)\|_{H^2} \max_{0 \le |\alpha| \le 1}\| D^\alpha u\|_{H^2} \\
& \lesssim c+ (\|u\|_{H^2}^{d/2} + \|u\|_{H^2})\|u\|_{H^3} \le c',
\end{align*}
using the same bound as in (\ref{critchain}) but now with $f'''$ replacing $f''$. The last argument can now be iterated for all $\gamma>4$, using the multiplier inequality (\ref{multbasic}) for $H^{\gamma-2}$-norms to see
$$\|f(u)\|_{H^{\gamma-1}} \lesssim c+ \max_{0 \le |\alpha| \le 1}\|f'(u) D^\alpha u\|_{H^{\gamma-2}} \lesssim c+ \|f'(u)\|_{H^{\gamma-2}} \|u\|_{H^{\gamma-1}} \le C,$$ since the bound for $\|f'(u)\|_{H^{\gamma-2}}$ is the same as the one for $\|f(u)\|_{H^{\gamma-2}}$ up to adjusting the constants to depend on $f'$ rather than on $f$.

Summarising, assuming the induction hypothesis with $\gamma-1$, we have shown the inequality $$ \frac{d}{dt} \|u\|_{h^\gamma}^2 + \|\nabla u\|_{h^\gamma}^2 \le c_0,$$ for some uniform constant $c_0>0$. Since $\|u(0)\|_{h^\gamma} = \|\theta\|_{h^\gamma} \lesssim B$ by hypothesis, this can be integrated just as in Step 1 with $\gamma=0$ to give 
$$ \|u(t)\|_{h^\gamma}^2 + \int_0^T \|\nabla u(t)\|_{h^\gamma}^2 dt \le c_1(B,f,\gamma,T)<\infty,$$ which proves (\ref{guaten}), using also that  $\|u\|^2_{H^{\gamma+1}} \simeq \|u\|^2_{L^2} + \|\nabla u\|_{h^\gamma}^2$. Finally the first estimate in (\ref{tbound}) follows from integrating over $[0,T]$ the inequality
\begin{equation*}
\big\|u'_\theta(t)\big\|_{h^{\gamma-1}}^2 = \|\Delta u_\theta(t) + f(u_\theta(t,\cdot))\|_{h^{\gamma-1}}^2 \lesssim \|u_\theta(t)\|_{h^{\gamma+1}}^2 + \|f(u_\theta(t,\cdot))\|^2_{h^{\gamma-1}},
\end{equation*}
and the preceding estimates. Applying the preceding inequality with $\gamma-1$ replaced by $\gamma-2$ gives the second part of (\ref{tbound}).

\smallskip

\textbf{Step 3: Conclusion of the proof.} We now employ the usual Galerkin argument: We can use standard ODE theory to construct smooth solutions $u_J$ to (\ref{weakpde}) in the finite dimensional space $E_J, J \in \N,$ from (\ref{ej}), with projected non-linearity $P_{E_J}(f(\cdot))$ and initial conditions $P_{E_J}\theta \in E_J $, where $P_{E_J}$ is the $L^2$-projector onto $E_J$. Note that $$\|P_{E_J}h-h\|_{h^\gamma} \to_{J \to \infty} 0,~~ \|P_{E_J}h\|_{h^\gamma} \le \|h\|_{h^\gamma} ~ \forall h \in H^\gamma.$$ This and the estimates from above then imply uniform in $J$ bounds for the norms of $u_J$ in (\ref{guaten}) and (\ref{tbound}). We can take weak sub-sequential limits as $J \to \infty$ (by the Banach Alaoglu theorem in the corresponding Hilbert spaces), and the limits provide weak solutions of (\ref{weakpde}). Since $\theta \in H^1$, Theorem \ref{exist} implies that these weak solutions coincide with the unique strong solutions, in particular they inherit the same norms bounds as weak limits in Hilbert space. The details are as on p.222f in \cite{R01} and left to the reader.

Finally, the time continuity of the solutions follows as in Corollary 7.3 in \cite{R01}. If $\theta \in C^\infty$ then $\theta \in H^\gamma$ for every $\gamma$ and (\ref{guaten}), (\ref{tbound}) imply that $\|u_\theta\|_{H^\gamma} + \|u'_\theta\|_{H^{\gamma-2}}$ are bounded on $0<t<T$. An application of the Sobolev imbedding theorem with $\gamma-2>b+d/2$ then bounds the $C^b$norms uniformly on $[0,T]$, for any $b$.
\end{proof}

\subsubsection{Regularity of the forward map $\G$}

We can now derive some first properties of the `forward' map 
\begin{equation} \label{fwdmap}
\G(\theta)=u_\theta,~~ \G: H^1(\Omega) \to C([0,T], L^2(\Omega)),
\end{equation} 
arising from the solution operator $\theta \mapsto u_\theta$ of the PDE (\ref{PDE}) constructed in Theorem \ref{exist}.

\begin{corollary}\label{yetagain}
Let $f \in C^\infty_c(\R)$ be fixed and $d \le 3$. Then hypotheses B) and A) of Condition \ref{gemol} are satisfied, and with any $\bar \gamma>d/2$ in case A).
\end{corollary}
\begin{proof}
Part A) follows from (\ref{guaten}) and the Sobolev imbedding. For Part B), we use (\ref{ipart}) and obtain for $w=u_\theta-u_\vartheta$ and all $0<t<T$ that
$$\frac{1}{2}\frac{d}{dt}\|w\|_{L^2}^2 + \|\nabla w\|_{L^2}^2 = \langle f(u_\theta)-f(u_\vartheta), u_\theta - u_\vartheta \rangle_{L^2} \le \|f\|_{Lip} \|u_\theta - u_\vartheta\|_{L^2}^2 \le c \|w\|_{L^2}^2$$ which combined with Gronwall's inequality gives 
\begin{equation} \label{tbdu}
\|w(t)\|_{L^2}^2 \lesssim \|w(0)\|_{L^2}^2,~~0<t<T.
\end{equation}
Integrating this bound over $[0,T]$ completes the proof.
\end{proof}

Note that we can combine the last two inequalities further to obtain as well
\begin{equation*} 
\int_0^T \|\nabla w(t)\|_{L^2}^2dt \lesssim \int_0^T \|w(t)\|_{L^2}^2dt - \int_0^T \frac{d}{dt}\|w(t)\|_{L^2}^2dt \lesssim \|w(0)\|_{L^2}^2
\end{equation*}
which then also implies
\begin{equation}\label{auxnew}
\int_0^T \|u_\theta(t) - u_\vartheta(t)\|_{H^1}^2dt \lesssim  \int_0^T\|u_\theta(t) - u_\vartheta(t)\|_{L^2}^2dt + \int_0^T\|\nabla (u_\theta(t) - u_\vartheta(t))\|_{L^2}^2 dt \lesssim \|\theta - \vartheta\|^2_{L^2},
\end{equation}
which we shall use below.

\smallskip

To proceed consider now any $\theta, \theta_0$ such that $\|\theta\|_{H^{\bar \gamma}} +\|\theta_0\|_{H^{\bar \gamma}} \le B, \bar \gamma >3+d/2,$ and let $u_\theta, u_{\theta_0} \in C([0,T], H^{\bar \gamma})$ be the corresponding strong solutions of (\ref{PDE}) from Proposition \ref{roundrobin} with initial conditions $\theta, \theta_0$. Then  $w(t)=u_{\theta}(t,\cdot) - u_{\theta_0}(t, \cdot)$ satisfies on $(0,T] \times \Omega$ 
\begin{equation}\label{mvt}
\frac{\partial}{\partial t} w - \Delta w  =  f(u_{\theta}) - f(u_{\theta_0}) = v(f,\theta, \theta_0) (u_\theta-u_{\theta_0})
\end{equation}
with $w(0)=\theta-\theta_0$ and  $$v(f,\theta,\theta_0)(t,x) = \int_0^1 f'(u_{\theta_0}(t,x)+ s(u_\theta(t,x) - u_{\theta_0}(t,x))ds,~~ 0 \le s \le 1.$$ Here we have applied the fundamental theorem of calculus to the map $$H: s \mapsto f(u_{\theta_0}(t,x)+ s(u_\theta(t,x) - u_{\theta_0}(t,x))$$ for $t,x$ fixed, such that $f(u_{\theta}) - f(u_{\theta_0})=H(1)-H(0)=\int_0^1 H'(s)ds$.

 We see that $w$ solves the linear PDE (\ref{PDElin}) below with potential $V=v(f, \theta, \theta_0)$, $m=0$ and initial condition $\theta-\theta_0$. We thus deduce from Lemma \ref{paristexas} below the stability estimate
\begin{equation}\label{dochparis}
\|u_{\theta} - u_{\theta_0}\|^2_{L^2([0,T], L^2(\Omega))} = \|w\|^2_{L^2([0,T], L^2(\Omega))} \ge c \|\theta-\theta_0\|^2_{H^{-1}},
\end{equation}
for some $c>0$ depending only on $d,T,B$, noting that the $C^1([0,T], C^1)$-norm of $V=v(f, \theta, \theta_0)$ is bounded by a fixed constant $M=M(B)>0$, seen as follows: Let us write shorthand $\tilde u = (1-s) u_{\theta_0} + s u_{\theta}$. First we have from (\ref{guaten}) and the Sobolev imbedding with $\bar \gamma >1+d/2$ that $$\|f'(\tilde u(t))\|_{C^1} \lesssim \|f'\|_\infty + \|f''\|_{\infty} \|\tilde u(t)\|_{C^1} \lesssim c +  \sup_{\|\theta\|_{H^{\bar \gamma} \le B}} \|u_\theta(t)\|_{H^{\bar \gamma}} \le M$$ for every $t \in [0,T]$.  Next, for any multi-index $\beta$ and partial differential operator $D^\beta$, the chain rule implies, for almost every $0<t<T$,
\begin{align*}
 \|f''(\tilde u(t)) \tilde u'(t)\|_{C^1} & \lesssim \|f''(\tilde u(t)) \tilde u'(t)\|_{\infty} + \max_{|\beta|=1}\|f'''(\tilde u(t)) D^\beta \tilde u(t) \tilde u'(t) + f''(\tilde u(t)) D^\beta \tilde u'(t)\|_{\infty} \\
& \lesssim \|\tilde u'(t)\|_{\infty} (1 + \|\tilde u(t)\|_{C^1}) + \|\tilde u'(t)\|_{C^1} \\
& \lesssim \sup_{\|\theta\|_{H^{\bar \gamma} \le B}} [\|u_\theta'(t)\|_{\infty} (1 + \|u_\theta(t)\|_{C^1}) + \|u_\theta'(t)\|_{C^1}] \le c<\infty,
\end{align*}
 using (\ref{tbound}) and the Sobolev imbedding $H^{\bar \gamma-2} \subset C^1$ and $\bar \gamma>3+d/2$. These bounds are uniform in $s$ and integrating them $ds$ yields the desired result.

\medskip

Similar to above we can take $w=u_\theta-u_{\vartheta}$ for $\theta, \vartheta \in H^\zeta \cap H^1$ and use the fundamental theorem of calculus as well as the Cauchy-Schwarz and the multiplier inequalities (\ref{multbasic}) for any $\zeta >d/2$ to deduce the inequality
$$\frac{1}{2}\frac{d}{dt}\|w\|_{h^\zeta}^2 + \|\nabla w\|_{h^\zeta}^2 = \langle v(f,\theta, \vartheta)(u_\theta-u_\vartheta), u_\theta - u_\vartheta \rangle_{h^\zeta} \le \|v(f,\theta, \vartheta)\|_{H^\zeta} \|u_\theta - u_\vartheta\|_{H^\zeta}^2 \le c \|w\|_{h^\zeta}^2$$ for all $0<t<T$, where $\|f'(\tilde u)\|_{H^\zeta}$ can be bounded as in the proof of Proposition \ref{roundrobin} by a constant depending only on $f$ and $B \ge \|\theta\|_{H^\zeta} + \|\vartheta\|_{H^\zeta}$. This implies by Gronwall's inequality and the Sobolev imbedding $H^\zeta \subset L^\infty$ that 
\begin{equation} \label{h2estimate}
\sup_{0<t<T}\|u_\theta(t)-u_{\vartheta}(t)\|_\infty \lesssim \sup_{0<t<T}\|u_\theta(t)-u_{\vartheta}(t)\|_{H^\zeta} \lesssim \|\theta-\vartheta\|_{H^\zeta}.
\end{equation}

\subsubsection{The linearisation of $\G$ as a Schr\"odinger equation}

The preceding argument (\ref{mvt}) is only a `pseudo'-linearisation argument since the potential $V$ featuring in the linear PDE still depends on $\theta, \theta_0$ via $\tilde u$. The following theorem derives the full linearisation of the map $\G$ from (\ref{fwdmap}) near a fixed point $\theta_0$.

\begin{theorem}\label{linrob}
Let $f\in C^\infty_c(\R)$ and $\theta_0 \in C^\infty$ be fixed and let $h \in H^{\bar \gamma}, \bar \gamma>\max(1,d/2), d \le 3,$ be such that $\|\theta_0+h\|_{H^{\bar \gamma}} \le B.$ Denote by  $u_{\theta_0}, u_{\theta_0+h} \in L^\infty([0,T], H^{\bar \gamma})$ the strong solutions of (\ref{PDE}) from Proposition \ref{roundrobin} with initial conditions $\theta_0, \theta_0+h$, respectively. Then  there exists a constant $C=C(B, f, \theta_0, \bar \gamma, T)>0$ such that
\begin{equation} \label{quadap}
\sup_{0<t<T}\|u_{\theta_0+h}(t) - u_{\theta_0}(t) - U_{f,\theta_0, h}(t)\|_{L^2} \leq C \|h\|^2_{L^2} 
\end{equation}
where $$U_h=U_{f,\theta_0, h}=:\mathbb I_{\theta_0}[h]$$ is the unique solution to the time-dependent \textit{linear} Schr\"odinger equation (\ref{PDElin}) with $m=0$, initial condition $h$, and potential $V(t, \cdot)=f'(u_{\theta_0}(t, \cdot)) \in C^{1,\infty}$ from (\ref{notsosmooth}).
\begin{proof}
Let us denote the error in the linear approximation by $$r(t) =u_{\theta_0+h}(t) - u_{\theta_0}(t) - U_{f,\theta_0, h}(t)$$ and notice that $r$ solves the linear PDE
\begin{equation}
\frac{\partial}{\partial t} r - \Delta r - f'(u_{\theta_0}) r = m,~~\text{where } m = f(u_{\theta_0+h}) - f(u_{\theta_0}) - f'(u_{\theta_0})(u_{\theta_0+h}-u_{\theta_0}),
\end{equation}
with initial condition $r(0)=\theta_0+h-\theta_0 - h=0$. We can apply the uniform in $0<t<T$ regularity estimate (\ref{inhom}) to $r$ with $a=0$, $h=0$ (noting also that the stipulated regularity of $V$ follows from the chain rule and Proposition \ref{roundrobin} for smooth $\theta_0$). Using also a second order Taylor expansion of $f$, (\ref{multbasicn}), the Cauchy-Schwarz inequality twice, the Sobolev embedding $H^1 \subset L^4$ (since $d \le 3$), and writing $w=u_{\theta_0+h}-u_{\theta_0}$ we thereby obtain
\begin{align*}
 \sup_{0<t<T}\|r(t)\|^2_{L^2}  &\lesssim \int_0^T\|m(t)\|^2_{H^{-1}}dt \lesssim \|f\|_{C^3} \int_0^T \|w(t)^2\|_{H^{-1}}^2dt \\
&\lesssim 
  \int_0^T \Big[\sup_{\|\phi\|_{H^1} \le 1}\int_\Omega \phi(x) w^2(t,x) dx \Big]^2dt \\
&\lesssim \int_0^T\|w(t)\|_{L^4}^2 \|w(t)\|_{L^2}^2 \sup_{\|\phi\|_{H^1} \le 1}\|\phi\|_{L^4}^2dt \lesssim \int_0^T \|w(t)\|_{H^1}^2 \|w(t)\|_{L^2}^2  dt \\
&\lesssim \sup_{0<t<T}\|w(t)\|_{L^2}^2 \int_0^T \|w(t)\|_{H^1}^2dt  \lesssim \|h\|_{L^2}^4,
\end{align*}
where the last inequalities follow from the estimates (\ref{tbdu}), (\ref{auxnew}).
\end{proof}
\end{theorem}

We further record here the consequence of Theorem \ref{linrob} and Proposition \ref{roundrobinlin} that the mapping $\mathbb I_{\theta_0}$ is continuous from $L^2(\Omega)$ to $L^2([0,T], L^2(\Omega))$. In fact we also have
\begin{equation}\label{izeta}
\sup_{0<t<T}\|\mathbb I_{\theta_0}[h](t,\cdot)\|_{\infty} \lesssim \sup_{0<t<T} \|U_h(t,\cdot)\|_{H^\zeta} \lesssim \|h\|_{H^\zeta},~\zeta>d/2,
\end{equation}
in view of Proposition \ref{roundrobinlin} and the Sobolev imbedding $H^\zeta \subset L^\infty$.

\subsection{Regularity theory for time-dependent linear Schr\"odinger equations}

For $\Omega = [0,1]^d, d \in \N$, we now turn to a detailed study of the solution operator $h \mapsto U_h$ of the linear time-dependent periodic Schr\"odinger equation 
\begin{align}\label{PDElin}
\frac{\partial}{\partial t} U(t,\cdot) - \Delta U(t, \cdot) - V(t, \cdot) U(t, \cdot) &=m(t, \cdot)~~\text{on } (0,T] \times \Omega, \notag \\
U(0, \cdot) &= h~~\text{on } \Omega,
\end{align}
for a general uniformly bounded potential $V \in C([0,T], C(\Omega))$, initial condition $h \in H^1(\Omega)$, and source term $m \in L^2([0,T],L^2)$.  The existence of unique weak solutions $$U \in L^2([0,T], H^1) \cap C([0,T], L^2) \text{ with } U' = (\partial U/\partial t) \in L^2({0,T}, H^{-1})$$ satisfying for all $v \in H^1$ and a.e.~$t \in [0,T]$ the equations
\begin{align} \label{weakpdelin}
 \langle U', v \rangle_{L^2} + \langle \nabla U, \nabla v \rangle_{L^2} - \langle V(t,\cdot) U, v \rangle_{L^2} &= \langle m(t, \cdot), v \rangle_{L^2}, \notag \\
U(0) &= h,
\end{align}
can be proved by standard arguments from linear parabolic PDE just as on p.380f.~in \cite{E10}. The regularity estimate Proposition \ref{roundrobinlin} below gives sufficient conditions on $V,m$ for these to be strong solutions $U \in L^2([0,T], H^2) \cap C([0,T], H^1), U' \in L^2([0,T], L^2)$ whenever $h \in H^1$. The estimates in Proposition \ref{roundrobinlin} for $a\le 0$ further imply by the usual Galerkin limiting procedure (again as on p.380f. in \cite{E10}) that unique solutions $U_h \in L^2([0,T], H^{a+1}) \cap C([0,T], H^a)$ with $U' \in L^2([0,T], H^{a-1})$ also exist for initial conditions $h \in H^{a}$ -- in this case (\ref{PDElin}) holds in $H^{a-1}$ in the sense that we are testing against $v \in H^{1-a}$ in (\ref{weakpdelin}). 

\smallskip

We will apply the estimates that follow with $V=f'(u_{\theta})$ for $u_\theta$ the solution of (\ref{PDE}), as this provides the linearisation of the forward map $\G$ derived in Theorem \ref{linrob}. For the verification of Condition \ref{gemol}E) in (\ref{dochparis}) we need to track the explicit dependence of the constants on $\theta \in U(\bar \gamma, B)$ via the $C^1([0,T], C^1)$-norm of the potential $V=f'(u_\theta)$. For the verification of Condition \ref{gemol}F) and to prove the key smoothing Lemma \ref{bootcamp}, we  consider potentials $V \in C^{1,\infty}$ from (\ref{notsosmooth}). This is compatible with our choice $V=f'(u_{\theta_0}), f \in C^\infty_c, \theta_0 \in C^\infty$, in view of Proposition \ref{roundrobin}.   

\subsubsection{A parabolic regularity estimate}

\begin{proposition}\label{roundrobinlin}
Suppose $U_h \in L^2([0,T], H^{a+1})$ with $U_h' \in L^2([0,T], H^{a-1})$ is a weak solution in $H^{a-1}$ of (\ref{PDElin}) for $h \in H^a$,  $m \in L^2([0,T], H^a)$, $a \in \R$. Then $U_h \in C([0,T], H^a(\Omega))$.  Assume further either i) that $a=-1$ and $B \ge \|V\|_{C^1([0,T], C^1)}$ for some $B>0$, or that ii) $V \in C^{1,\infty}$ from (\ref{notsosmooth}). Then we have for all $0<T_0 \le T$,
\begin{align} \label{inhom}
& \int_0^{T_0} \|U_h(t)\|_{H^{a+1}}^2 dt + \sup_{0<t<T_0} \|U_h(t)\|^2_{H^a} + \int_0^{T_0} \|U_h'(t)\|^2_{H^{a-1}}dt \\
&~~~~~~~~~~~~~~~~~~ \le C \|h\|^2_{H^a} + C\int_0^{T_0} \|m(t)\|^2_{H^{a-1}}dt \notag
\end{align}
where $C$ is a constant that depends on $B, f, T, d$ in case i) and on $a, f, T, d, V$ in case ii).  Moreover,
\begin{equation}\label{timedereg}
{\rm{ess}}\sup_{0<t<T_0}\|U'_h(t)\|^2_{H^{a-2}} \le C \|h\|^2_{H^a} + C\int_0^{T_0} \|m(t)\|^2_{H^{a-1}}dt + C {\rm{ess}}\sup_{0<t<T_0}\|m(t)\|_{H^{a-2}}^2.
\end{equation}

\end{proposition}
\begin{proof}
We first establish the bounds assuming the solution $U_h$ is sufficiently regular so that all expressions are well defined. One can then employ a standard Galerkin argument as in Sections 7.2-7.4 in \cite{R01} to extend the result to hold in general. Time continuity of solutions into $H^a(\Omega)$ follows as in Corollary 7.3 in \cite{R01}. In case ii), we first establish this and the subsequent linear estimates at integer Sobolev orders $a \in \mathbb Z$, the real case then follows by interpolation.

Let us write $U=U_h$ in this proof, and set $T_0=T$ without loss of generality. It suffices to establish the bound for the equivalent $h^a$-norms. From (\ref{ipart}) we have $\|\nabla U \|^2_{h^a} = -\langle \Delta U, U \rangle_{h^a}$ for any $a \in \R$. Differentiating the squared $h^a$-norm (\ref{seqnorm}) we have from the Cauchy-Schwarz, Sobolev multiplier (\ref{multbasicn}), and Young inequalities
\begin{align}\label{1stbdlin}
\frac{1}{2} \frac{d}{dt} \|U\|_{h^a}^2 + \|\nabla U \|^2_{h^a} &= \langle \frac{\partial}{\partial t} U, U \rangle_{h^a} - \langle \Delta U, U\rangle_{h^a} \\
&= \langle VU, U \rangle_{h^a} + \langle m, U \rangle_{h^a} \notag \\
& \le \|V U\|_{h^a} \|U\|_{h^a} + \|m\|_{h^{a-1}} \|U\|_{h^{a+1}} \notag \\
&\lesssim (1+\|V\|_{C^{|a|}}) \|U\|_{h^a}^2 +  \|m\|_{h^{a-1}}^2 + \frac{1}{2} \|\nabla U\|_{h^a}^2. \notag
\end{align}
using also that $\|\cdot\|_{h^a} + \|\nabla (\cdot)\|_{h^a}$ is equivalent to the $\|\cdot\|_{h^{a+1}}$-norm. Subtracting $\|\nabla U\|_{h^a}^2/2$ implies first that $\frac{d}{dt}\|U(t)\|^2_{h^a} \lesssim \|U(t)\|^2_{h^a} + \|m(t)\|_{h^{a-1}}^2$ for all $0<t<T$ to which we can apply Gronwall's inequality to deduce $$\|U(t)\|^2_{h^a} \lesssim \|U(0)\|_{h^a}^2 + \int_0^t  \|m(s)\|_{h^{a-1}}^2ds \le \|h\|_{h^a}^2 + \int_0^T  \|m(s)\|_{h^{a-1}}^2ds ,~~~0 < t \le T.$$ Then integrating (\ref{1stbdlin}) we further obtain
$$\int_0^T \|\nabla U(t)\|_{h^a}^2 dt \lesssim \|h\|_{h^a}^2 + \int_0^T  \|m(t)\|_{h^{a-1}}^2dt - \frac{1}{2} \int_0^T \frac{d}{dt} \|U\|_{h^a}^2 \lesssim \|h\|_{h^a}^2 + \int_0^T  \|m(t)\|_{h^{a-1}}^2dt,$$ which in particular bounds $\int_0^T \|U(t)\|^2_{h^{a+1}}dt$ by the r.h.s.~in (\ref{inhom}). The proof of the first inequality is completed upon noting, using again (\ref{multbasic}) and the assumption on $V$, that for a.e.~$t$,
\begin{equation}
\|U'(t)\|_{H^{b}} \lesssim \|\Delta U(t)\|_{H^{b}} + \|V(t)U(t)\|_{H^{b}} + \|m(t)\|_{H^{b}} \lesssim \|U(t)\|_{H^{b+2}} + \|m(t)\|_{H^{b}},
\end{equation}
which applied with $b=a-1$ and after integrating combines with the preceding bound for $\int_0^T\|U(t)\|^2_{H^{a+1}}dt$ to complete the proof of (\ref{inhom}). The second inequality then follows similarly from $b=a-2$ in the last display and the preceding bound for $\|U(t)\|_{H^a}$. 
\end{proof}

\subsubsection{Forward smoothing of the semigroup}

We now strengthen the preceding estimates for strictly positive times, adapting an argument from p.294 in \cite{R01} to the present situation.

\begin{proposition}\label{fwdsmooth}
In the setting of Proposition \ref{roundrobinlin}, let $m=0$ and assume $V \in C^{1,\infty}$ from (\ref{notsosmooth}). For all fixed $0<t_0<T_0\le T$ there exists a constant $c=c(t_0,T_0, V, a ,d)$ such that
\begin{equation}\label{bootie}
{\rm{ess}}\sup_{t_0 \le t \le T_0}\|U'_h(t)\|_{H^{a-1}} \le c \|h\|_{H^a},
\end{equation}
\begin{equation}\label{bootienull}
\sup_{t_0 \le t \le T_0}\|U_h(t)\|_{H^{a+1}} \le c \|h\|_{H^a}.
\end{equation}
\end{proposition}
\begin{proof}
Again we prove the required estimates first assuming $h, U_h, U_h'$ are smooth and a Galerkin approximation argument just as in the proofs of Propositions \ref{roundrobin} and \ref{roundrobinlin} then implies the general result. 

It suffices to consider the equivalent $h^a$-norms. Let us write $U=U_h$ in the proof and define $g(t) = tU'(t)$ for $t\in (0,T]$ with $g(0)=0$. Differentiate the equation (\ref{PDElin}) with respect to $t$ and multiply by $tg(t)$ to deduce
\begin{align*}
\langle t^2 U', (\partial/\partial t) U' \rangle_{h^{a-1}} = t^2 \langle U', \Delta U'\rangle_{h^{a-1}} + t^2 \langle V'U + V U', U' \rangle_{h^{a-1}}~a.e. \text{ on } [0,T]
\end{align*}
which implies by (\ref{ipart}), the Cauchy-Schwarz and Young inequalities, as well as  (\ref{multbasicn}) that almost everywhere on $[0,T]$,
$$\frac{1}{2}\frac{d}{dt}\|t U'(t)\|_{h^{a-1}}^2 - t \|U'(t)\|_{h^{a-1}}^2 + t^2 \|\nabla U'(t)\|_{h^{a-1}}^2 \le c_{V, T_0} (\|U'(t)\|^2_{h^{a-1}} + \|U(t)\|_{h^{a-1}}^2)$$ for some finite constant $c_{V,T_0}$ depending on $\|V\|_{C^1([0,T], C^{|a-1|})}<\infty$. Therefore, integrating the penultimate identity over $(0,t_0)$, any $t_0>0,$ and using $g(0)=0$ as well as Proposition \ref{roundrobinlin},
\begin{align*}
&\|t_0U'(t_0)\|_{h^{a-1}}^2 + \int_0^{t_0}t^2 \|\nabla U'\|_{h^{a-1}}^2dt \\
&\le ~~\int_0^{t_0} \big[(c_{V, T_0}+t)\|U'(t)\|^2_{h^{a-1}} + c_{V,T_0}\|U(t)\|^2_{h^{a-1}}\big]dt  \lesssim \|h\|_{h^a}^2
\end{align*}
 which provides the required bound on $\|U'(t_0)\|_{H^{a-1}}$ after dividing by $t_0>0$, and proves (\ref{bootie}) since $t_0$ was arbitrary. For the final estimate we notice that for a.e.~$t \in [t_0,T_0]$,
$$\Delta U(t) = U'(t) - V(t) U(t) $$ and hence, using the standard elliptic regularity estimate $\|u\|_{h^{a+1}} \lesssim \|(Id - \Delta) u\|_{h^{a-1}}$ (e.g., via (\ref{seqnorm}) and (\ref{ipart})), (\ref{bootie}) and again Proposition \ref{roundrobinlin} we obtain for almost all $t \ge t_0$,
\begin{equation}
\|U_h(t)\|_{h^{a+1}} \lesssim \|\Delta U_h(t)\|_{h^{a-1}} + \|U_h(t)\|_{h^{a-1}} \lesssim \|U_h'(t)\|_{h^{a-1}} +  \|U_h(t)\|_{h^{a-1}} \lesssim \|h\|_{h^a}.
\end{equation}
Since $U_h \in C([0,T], H^a)$ this inequality can in fact be shown to hold everywhere on $[t_0,T_0]$ (e.g., as in Lemma 11.2 in \cite{R01}).
\end{proof}

Using the above proposition iteratively allows by a bootstrap argument to show that the mapping $h \to U_h(t_0)$ maps any $H^a$ into any $H^b$ space, reflecting the smoothing nature of the semigroup action at strictly positive times $t>0$.
\begin{lemma}\label{bootcamp}
In the setting of Proposition \ref{roundrobinlin}, assume $m=0$ and that $V$ lies in $C^{1,\infty}$ from (\ref{notsosmooth}). Let $t_{\min}>0$, and consider any real numbers $b>a$. Then
\begin{equation}
\sup_{t_{\min} \le t \le T}\|U_h(t)\|_{H^b} \le c\|h\|_{H^a},
\end{equation}
for some constant $c=c(t_{\min}, T, V, a,b, d)>0$.
\end{lemma}
\begin{proof}
Fix a small $0<\epsilon<t_{\min}$ and dissect $[\epsilon, t_{\min}]$ into at least $b-a$-many points $$\epsilon_l = \epsilon + \frac{l}{At_{\min}}(t_{\min}-\epsilon),~~ l=0,\dots, At_{\min},$$ for appropriate $A=A(a,b,t_{\min})>0$. We use (\ref{bootienull}) with $t_0 = \epsilon_l$ and for equation (\ref{PDElin}) started at time $\epsilon_{l-1}$ to obtain the chain of inequalities
\begin{equation}\label{onegain}
\|U_h(\epsilon_l)\|_{H^b} \lesssim \|U_h(\epsilon_{l-1})\|_{H^{b-1}},~~ l \ge 1,
\end{equation}
which can be iterated to obtain for any $t\ge t_{min}$ that $$\|U_h(t)\|_{H^b} \lesssim \|U_h(t_{\min})\|_{H^b} \lesssim \|U_h(\epsilon)\|_{H^a} \lesssim \|h\|_{H^a},$$ where we have also used Proposition \ref{roundrobinlin}.
\end{proof}

\subsubsection{Constant approximation of the potential}

For a potential $V \in C^1([0,T], C^b), b \ge 1$, and $0<\varepsilon \le T$, define $$V^{(\varepsilon)}(t, \cdot) = V(0,\cdot),~\text{for}~ t \in[0,\varepsilon/2],~~ V^{(\varepsilon)}(t, \cdot) = V(t, \cdot),~\text{for}~ \varepsilon \le t \le T,$$ and linear in between $$ V^{(\varepsilon)}(t, \cdot) =\Big(1-\frac{t-(\varepsilon/2)}{\varepsilon/2}\Big)V(0, \cdot) + \frac{t-(\varepsilon/2)}{\varepsilon/2}  V(\varepsilon, \cdot),~~~t \in [\varepsilon/2, \varepsilon].$$ This function lies again in $C^1([0,T], C^b)$, with norm bounded by at most a constant multiple of $\|V\|_{C^1([0,T], C^b)}$. In particular if $V \in C^{1,\infty}$ then $V^{(\varepsilon)}$ also lies in $C^{1,\infty}$. 

\begin{proposition}\label{cstpert}
Assume either 

\smallskip

 i) that $a=-1, |b|=1,$ and $B \ge \|V\|_{C^1([0,T], C^{|b|})}$ for some $B>0$,

\smallskip

ii) or that $a \in \R, b \ge a-1,$ and $V \in C^{1,\infty}$ from (\ref{notsosmooth}). 

\smallskip

Suppose $U_h, U_h^{(\varepsilon)}$ are weak solutions to (\ref{PDElin}) as in Proposition \ref{roundrobinlin} with potentials $V$ and $V^{(\varepsilon)}$ respectively, both with initial condition $h \in H^a$ and source $m=0$. Then there exists a constant $c>0$ depending on $T,B,d$ in case i) and on $a, T, V, d$ in case ii) such that for all $0<\varepsilon \le T, 0<T_0 \le T,$ we have 
$$\int_0^{T_0} \|U_h(t) - U^{(\varepsilon)}_h(t)\|_{H^{a+1}}^2 dt \le c T_0 \varepsilon^2  \|h\|^2_{H^{b}}.$$ 
In case ii) we further have, for some constant $c=c(a,T, V, d)>0$,
$$\int_0^{T_0} \|U_h(t) - U^{(\varepsilon)}_h(t)\|_{H^{a+1}}^2 dt \le c \varepsilon^2  \|h\|^2_{H^{a-2}}$$
\end{proposition}
\begin{proof}
The function $\bar v_h = U_h - U_h^{(\varepsilon)}$ solves on $(0,T] \times \Omega$ the equation
$$\Big(\frac{\partial}{\partial t} - \Delta - V^{(\varepsilon)} \Big)\bar v = (V-V^{(\varepsilon)})  U_h  \equiv m$$ with $\bar v(0,\cdot) =h-h=0$. From the regularity estimate (\ref{inhom}) we obtain
$$ \int_0^{T_0}\|\bar v(s)\|_{h^{a+1}}^2 ds \lesssim \int_0^{T_0} \|m(t)\|_{h^{a-1}}^2dt.$$
Next, since $V=V^{(\varepsilon)}$ outside of $[0,\varepsilon]$, using Proposition 7.1 in \cite{R01}, Jensen's inequality, (\ref{multbasicn}) and again (\ref{inhom}) we further bound the r.h.s.~for any $b \ge a-1$ by
\begin{align*}
& \int_0^{\min(\varepsilon, T_0)} \|(V(t)-V(0) + V^{(\varepsilon)}(0)-V^{(\varepsilon)}(t))  U_h(t)\|_{h^{b}}^2dt \\
&=\int_0^{\min(\varepsilon, T_0)} t^2\big\|\frac{1}{t}\int_0^t(V'(s) - (V^{(\varepsilon)})'(s))ds  U_h(t)\big\|_{h^{b}}^2dt \\
&\lesssim \int_0^{\min(\varepsilon, T_0)} t\int_0^t\big\|V'(s) - (V^{(\varepsilon)})'(s)\|^2_{C^{|b|}}ds  \|U_h(t)\big\|_{h^{b}}^2dt \\
& \lesssim \varepsilon^2 ~ {\rm{ess}}\sup_{0 < s < \varepsilon}(\|V'(s)\|_{C^{|b|}} + \|(V^{(\varepsilon)})'(s)\|_{C^{|b|}})\int_0^{T_0} \|U_h(t)\|^2_{h^{b}}dt \\
&\lesssim \varepsilon^2 T_0 \sup_{0 < t<T_0} \|U_h(t)\|_{h^{b}}^2 \lesssim T_0 \varepsilon^2 \|h\|_{h^{b}}^2.
\end{align*}
This proves the first inequality of the proposition. For the second we follow the same arguments with $b=a-1$ but in the last step bound $$\int_0^{T_0} \|U_h(t)\|^2_{H^{a-1}}dt \lesssim \|h\|_{H^{a-2}}^2$$ directly from (\ref{inhom}).
\end{proof}

\subsection{Spectral results for Schr\"odinger operators}\label{graphscale}

The solutions $U_h^{(\varepsilon)}$ of (\ref{PDElin}) with potential $V^{(\varepsilon)}$ constant in time on $[0,\varepsilon/2]$ from Proposition \ref{cstpert} can be studied via the theory of elliptic Schr\"odinger operators which we briefly review here for convenience: Define 
\begin{equation} \label{sop}
\mathscr S_W= \Delta - W, ~\textit{with bounded potential}~~W: \Omega \to \R,~~\|W\|_\infty \le \bar W.
\end{equation}

\begin{lemma}\label{hooray}
The kernel of $\sS_W$ given by
\begin{equation}\label{kernelW}
\mathcal K =\mathcal K_W = \{\phi \in H^1(\Omega):\Delta \phi -W\phi =0\}
\end{equation} 
is finite-dimensional.
\end{lemma}
\begin{proof}
Suppose $\phi$ lies in the kernel $\mathcal K$ so that $\Delta \phi = W \phi$ and define $A = (Id -\Delta) = \sum_{j=0}^\infty (1+\lambda_j) e_j \langle \cdot, e_j\rangle_{L^2}$ with $e_j, \lambda_j$ as in (\ref{weyl}), so that the inverse $A^{-1}$  is a compact linear operator on $H^{-1}(\Omega) \supset L^2(\Omega)$. Then $A\phi = (Id-\Delta) \phi = \phi - W\phi$ or equivalently 
\begin{equation}\label{fredkern}
[Id - A^{-1}(Id-W \cdot Id)]\phi =0.
\end{equation}
 But for $W \in L^\infty$ the operator $K\equiv A^{-1}(Id-W \cdot Id)$ is a composition of a compact and a continuous linear operator on $L^2$, hence itself compact, and we deduce that $Id -K$ is Fredholm (Proposition A.7.1 on p.593 in \cite{TI}) so that the kernel of $\phi$'s for which (\ref{fredkern}) holds is necessarily finite-dimensional, proving the lemma.
\end{proof}

\smallskip

Now consider a new Schr\"odinger operator $\sS_{W_+}$ with shifted potential 
\begin{equation} \label{w+}
W_+(x):= W(x) + \bar W +1,~ x \in \Omega,
\end{equation}
so that $\inf_{x \in \Omega} W_+(x) \ge 1>0$. Then we have from the divergence theorem
\begin{equation}\label{selfsp}
\langle -\sS_{W_+} u, u \rangle_{L^2} = \|\nabla u\|_{L^2}^2 + \langle W_+ u, u \rangle_{L^2} \simeq \|u\|_{H^1}^2,~~ u \in H^1(\Omega),
\end{equation}
hence by (\ref{dual}), $\|\sS_{W_+}u\|_{H^{-1}}=0$ implies $\|u\|_{H^1}=0$ and $\sS_{W_+}$ is injective from $H^1$ to $H^{-1}$. It is also surjective, for if it is not, then there exists a (nonzero) linear functional $u_0 \in (H^{-1})^* = H^1$ that is orthogonal to the (as is easily seen closed) range of $\sS_{W_+}$, i.e., such that $\langle \sS_{W_+} u, u_0 \rangle_{L^2}=0$ for all $u \in H^1$. But testing $u=u_0$ implies $u_0=0$, a contradiction to (\ref{selfsp}). Thus the operator $\sS_{W_+}$ has an inverse $\sS_{W_+}^{-1}$ defining a compact linear operator on $L^2 \subset H^{-1}$ that is also self-adjoint since the divergence theorem implies that $\langle \sS_{W_+} u, v \rangle_{L^2} = \langle u, \sS_{W_+} v\rangle_{L^2}$ and then
\begin{equation}\label{selfad}
\langle \sS^{-1}_{W_+} \phi, \psi \rangle_{L^2} = \langle \sS^{-1}_{W_+} \phi, \sS_{W_+} \sS_{W_+}^{-1} \psi\rangle_{L^2} = \langle \phi, \sS_{W_+}^{-1} \psi \rangle_{L^2}.
\end{equation}
The spectral theorem now implies the existence of real eigen-pairs $$(e_{j,W}, \lambda_{j,W_+}) \in H^1(\Omega) \times \R,~~j=0,1,2,\dots,$$ of $-\sS_{W_+}$; the $e_{j,W}$ form an orthonormal basis of the Hilbert space $L^2(\Omega)$. Taking $\|u\|_{L^2}=1$ in (\ref{selfsp}), the variational characterisation of eigenvalues (e.g., Sec.~4.5 in \cite{D95}) and (\ref{weyl}) imply that
\begin{equation} \label{specpert}
\lambda_{j,W_+} \in [\lambda_j +1 , \lambda_{j} + 2\bar W +1],~~ j \ge 0,~~ 0 \le \lambda_j \simeq j^{2/d}.
\end{equation}
Thus we arrive at the spectral formulae, for $\psi \in H^1, \phi \in L^2$, $$-\mathscr S_{W_+} \psi = \sum_j \lambda_{j,W_+} e_{j,W} \langle \psi, e_{j,W}\rangle_{L^2},~~-\mathscr S^{-1}_{W_+} \phi = \sum_j \lambda^{-1}_{j,W_+} e_{j,W} \langle \phi, e_{j,W}\rangle_{L^2}.$$ 

Now let us return to the original operator $\sS_W$ \textit{without shift}. If $\phi$ lies in the kernel $\mathcal K_W$  from (\ref{kernelW}), then $\sS_W\phi =0$ and so $\sS_{W_+} \phi = -(\bar W + 1) \phi$, hence if $e_{0,k}, k =1, \dots, dim(\mathcal K_W),$ is any $L^2(\Omega)$-orthonormal basis of $\mathcal K_W$, then these are eigenfunctions of $-\sS_{W_+}$ for the eigenvalue $\lambda_{W_+} = \bar W +1$. Generally, any eigenfunction $e_{j,W}$ satisfies $$(\sS_W - \bar W -1) e_{j,W} = -\lambda_{j,W_+} e_{j,W},~~ j \ge 0,$$ and hence
the operator $-\sS_W$ has the same eigenfunctions $e_{j,W}$ but for eigenvalues
\begin{equation} \label{tildeev}
 \lambda_{j,W} = \lambda_{j,W_+} - \bar W -1 \in [\lambda_j -\bar W, \lambda_j+ \bar W].
\end{equation}
If $\psi \in H^1, \phi \in L^2$ also lie in the orthogonal complement $L^2_{\mathcal K^\perp} \equiv L^2(\Omega) \ominus \mathcal K_W,$ we obtain $$\mathscr S_{W} \psi =- \sum_{j} \lambda_{j,W} e_{j,W} \langle \psi, e_{j,W}\rangle_{L^2}, ~~\mathscr S^{-1}_{W} \phi = -\sum_{j} \lambda^{-1}_{j,W} e_{j,W} \langle \phi, e_{j,W}\rangle_{L^2},$$ with series extending over $j$ such that $\lambda_{j,W} \neq 0$. In particular we can represent periodic weak solutions (as in (\ref{weakpdelin})) $w=w_h =w_{h,W}$ to the linear parabolic PDE with Schr\"odinger operator $\sS_W = \Delta - W$, 
\begin{align}\label{PDElinW}
\frac{\partial}{\partial t} w(t,\cdot) - \sS_W w(t, \cdot) &=0~~\text{on } (0,T_0] \times \Omega, \notag \\
w(0, \cdot) &= h~~\text{on } \Omega
\end{align}
for any $0 < T_0 \le T$ and $h \in H^1$  by the formula
\begin{equation}\label{ubahr}
w_{W,h}(t) = \sum_j e^{-t \lambda_{j,W}} e_{j,W} \langle e_{j,W}, h\rangle_{L^2} = \int_\Omega p_{t, W}(x,y)h(y)dy,~0 <t \le T_0,
\end{equation}
with symmetric Green kernel 
\begin{equation} \label{green}
p_{t,W}(x,y)=p_{t,W}(y,x)= \sum_j e^{-t \lambda_{j,W}} e_{j,W}(x) e_{j,W}(y).
\end{equation}
In the above, the sum extends also over the basis functions $e_{j, W} \equiv e_{0,k} $ of the finite-dimensional linear subspace $\mathcal K_W$ of $H^1$, with eigenvalues $\lambda_{j,W}=0$, where the semigroup acts just as the identity operation. 

\subsubsection{An auxiliary Sobolev scale}

In the proof of the key stability estimate Lemma \ref{paristexas} to follow, we exploit the structure of function spaces defined spectrally from $\sS_{W_+}$ as follows: 
\begin{equation}\label{tildeh}
\tilde H^a_W(\Omega)  \equiv \left\{h: \sum_{j \ge 0} \lambda^a_{j, W_+} \langle h, e_{j, W} \rangle_{L^2}^2 \equiv \|h\|_{\tilde H^a_W}^2 <\infty \right\},~a \in \R.
\end{equation}
By Parseval's identity we have the isometry $\tilde H^0_W=L^2$ for all $W$. Moreover for $u \in H^1$, and with constants implicit in $\simeq$ depending only on $\bar W$,
\begin{align}\label{w1est}
\|u\|^2_{H^1} &\simeq \|u\|^2_{L^2} + \|\nabla u\|_{L^2}^2 \simeq \|(W_+)^{1/2} u\|_{L^2}^2 - \langle \Delta u, u \rangle_{L^2} \notag \\
&= -\langle \Delta u - W_+ u, u \rangle_{L^2}   =\sum_j \lambda_{j, W_+} \langle u, e_{j,W} \rangle_{L^2}^2 \simeq \|u\|^2_{\tilde H^1_W},
\end{align}
so that $H^1=\tilde H^1_W$ with equivalent norms. By duality one shows further that $\tilde H_W^{-1} =H^{-1}$ with equivalent norms -- below we shall only need
\begin{align} \label{normeq}
\|u\|_{H^{-1}}&=\sup_{\|\phi\|_{H^1} \le 1}\Big|\int_\Omega \phi u \Big|  =\sup_{\|\phi\|_{H^1} \le 1}\Big|\sum_j \lambda^{1/2}_{j,W_+} \langle \phi, e_{j,W} \rangle_{L^2} \langle e_{j,W}, u \rangle_{L^2} \lambda_{j,W_+}^{-1/2} \Big|  \lesssim  \|u\|_{\tilde H^{-1}_W}
\end{align}
for all $u \in H^1$, which follows from (\ref{dual}), Parseval's identity, the Cauchy-Schwarz inequality and (\ref{w1est}). 

The final facts we need below are the following: We have $$u \in \tilde H^2_W \iff \Delta u - W_+ u \in L^2 \iff u \in H^2$$ 
and the norms are equivalent: On the one hand
$$\|u\|_{\tilde H_W^2} = \|(\Delta - W_+) u\|_{L^2} \lesssim \|\Delta u\|_{L^2} + \|u\|_{L^2} \simeq \|u\|_{H^2}$$ and conversely
$$\|u\|_{H^2} \simeq \|\Delta u\|_{L^2} + \|u\|_{L^2} \lesssim \|(\Delta - W_+)u\|_{L^2} + \|u\|_{L^2} \lesssim \|u\|_{\tilde H^2_W}.$$ In particular for $h \in H^2$ so that $\sS_{W_+} h \in L^2$, the series $\sum_{j} e_{j,W} \langle h, e_{j,W}\rangle_{L^2}$ converges in $\tilde H_W^2$ and for $d \le 3$ also uniformly on $\Omega$, since its partial sums $h_J$ satisfy, by the Sobolev imbedding, as $J \to \infty$
\begin{align} \label{uniflim}
\|h-h_J\|^2_\infty &\lesssim \|h-h_J\|^2_{H^2} \simeq \|h-h_J\|^2_{\tilde H^2_W} \notag \\
&= \sum_{j>J} \lambda^2_{j, W_+} \langle h, e_{j,W}\rangle_{L^2}^2 = \sum_{j>J} \langle \sS_{W_+}h, e_{j,W}\rangle_{L^2}^2 \to 0.
\end{align}

\subsubsection{Lipschitz stability of the integrated parabolic flow}

The following lemma exploits the preceding results and is at the heart of the key injectivity results for $\G, \mathbb I_\theta$ relevant in Condition \ref{gemol}.

\begin{lemma}\label{paristexas}
Let $U_h(t) \in L^2([0,T], L^2(\Omega))$ be a solution of (\ref{PDElin}) for source $m=0$, initial condition $h \in H^1$, and potential $V$ such that $M  \ge \|V\|_{C^1([0,T], C^1)}$. For every $T>0$ there exists a constant $c=c(M, T,d)$ such that
\begin{equation}
\int_0^T\|U_h(s)\|^2_{L^2(\Omega)}ds \ge c\|h\|_{H^{-1}}^2,~~ \forall h \in  H^1.
\end{equation}
\end{lemma}
\begin{proof} 
The solution $U_h^{(\varepsilon)}$ for $h \in H^1$ and bounded potential $V^{(\varepsilon)}$ can be represented via (\ref{ubahr}) on $(0,\varepsilon/2]$ since $V^{(\varepsilon)}= V(0, \cdot) \equiv -W$ is time-independent there. From Proposition \ref{cstpert} with $a= b=-1$ and with $T_0=\varepsilon/2$ we deduce, using also Parseval's identity 
\begin{align*}
\int_0^T\|U_h(s)\|^2_{L^2}ds & \ge \int_0^{\varepsilon/2}\|U_h(s)\|^2_{L^2}ds \\
&\ge \frac{1}{2}\int_0^{\varepsilon/2} \|U^{(\varepsilon)}_h(s)\|^2_{L^2}ds - \int_0^{\varepsilon/2} \|U^{(\varepsilon)}_h(s) -U_h(s)\|^2_{L^2}ds  \\
&\ge  \frac{1}{2}\int_0^{\varepsilon/2} \sum_j e^{-2s \lambda_{j,W}} \langle e_{j, W}, h\rangle^2_{L^2}ds - \frac{c_M}{4} \varepsilon^3 \|h\|_{H^{-1}}^2,
\end{align*}
For eigenvalues $\lambda_{j,W} \le 1$ and recalling $\bar W + 1 + \lambda_{j,W} \ge 1$  in view of (\ref{specpert}), (\ref{tildeev}) we see $$\int_0^{\varepsilon/2}  e^{-2s \lambda_{j, W}}ds \ge \int_0^{\varepsilon/2} e^{-2 s}ds = \frac{1}{2} (1 - e^{-\varepsilon})  \ge \frac{1}{2}\frac{ 1 - e^{-\varepsilon}}{\bar W + \lambda_{j,W}+1}.$$ 
For large eigenvalues  $\lambda_{j,W}>1$ we have the estimate
$$\int_0^{\varepsilon/2} e^{-2s \lambda_{j,W}}ds = \frac{1}{2\lambda_{j,W}}(1- e^{-\varepsilon \lambda_{j,W}}) \ge \frac{1}{2}\frac{1- e^{-\varepsilon }}{\bar W + \lambda_{j,W}+1}.$$ Combining these and integrating term-wise, we obtain the bound
\begin{equation}
\sum_j \int_0^{\varepsilon/2}e^{-2s \lambda_{j,W}}  \langle e_{j,W}, h\rangle^2_{L^2}ds \ge \frac{1}{2}(1 - e^{-\varepsilon})\|h\|_{\tilde H_W^{-1}}^2 \ge \frac{\bar b}{2}(1 - e^{-\varepsilon})\|h\|_{H^{-1}}^2,
\end{equation}
for constant $\bar b$ from (\ref{normeq}) depending only on $M$. Then choosing $\varepsilon$ small enough s.t. $$\frac{\bar b}{4}(1 - e^{-\varepsilon}) - \frac{c_M}{4} \varepsilon^3 >c''>0,$$ possible since $(1-e^{-\varepsilon})/\varepsilon \to 1$ as $\varepsilon \to 0$, we obtain $$\int_0^{T}\|U_h(s)\|^2_{L^2}ds \ge  c'' \|h\|_{H^{-1}}^2,$$ completing the proof. \end{proof}

\subsection{The information operator and its inverse}

In this subsection, the parameter $\theta_0\in C^\infty$ is a ground truth initial condition for the reaction diffusion equation considered in Theorem \ref{ganzwien}. As $\theta_0$ will be fixed throughout and no other values of $\theta$ will be considered, we will write $\theta =\theta_0$ to ease notation. We also set the time horizon $T=1$ for notational simplicity so that the inner products of $L^2(\mathcal X)$ and $L^2([0,T], L^2(\Omega))$ coincide. Then by Theorem \ref{linrob} and Proposition \ref{roundrobinlin} $$\mathbb I_\theta h = D\G_\theta[h] = U_{h}(t): L^2(\Omega) \to L^2(\mathcal X),~~L^2(\mathcal X)=L^2([0,1], L^2(\Omega)),$$ is the continuous linear operator solving the PDE (\ref{weakpdelin}) with $m=0$, initial condition $h$, and potential $V=f'(u_{\theta}(t, \cdot))$ which, in view of Proposition \ref{roundrobin}, lies in $C^{1,\infty}$ from (\ref{notsosmooth}). In particular $\mathbb I_{\theta}$ has a continuous and linear adjoint operator $\mathbb I_\theta^*: L^2(\mathcal X) \to L^2(\Omega)$ such that for all $G \in L^2(\mathcal X), h \in L^2(\Omega)$, we have $\langle \mathbb I_\theta h, G \rangle_{L^2(\mathcal X)} = \langle h, \mathbb I_\theta^*G \rangle_{L^2(\Omega)}.$ The (Fisher-) information operator $\mathbb I_\theta^* \mathbb I_\theta$ is then a bounded linear operator acting on $L^2(\Omega)$. To study its mapping properties let us first prove the following result:
\begin{lemma}\label{injep0}
Let $\eta \ge 0$. The linear operator
\begin{equation}\label{deltafish}
\mathcal I = \Delta \mathbb I_\theta^* \mathbb I_\theta
\end{equation}
maps $H^\eta$ continuously into $H^\eta_0$, and hence $H^\eta_0$ into itself. Moreover $\mathcal I$ is injective on $H^1_0$ and hence also on its subspaces $H^\eta_0$ for any $\eta \ge 1$.
\end{lemma}
\begin{proof}
Let us first prove the continuity statement: We have from the definition of the $h^\eta$ norms, (\ref{dual}), the Cauchy-Schwarz inequality and Proposition \ref{roundrobinlin} with $m=0$, $a=\eta$ and $a=-\eta-2$, for any $\phi \in H^\eta$,
\begin{align*}
\|\mathcal I \phi\|^2_{h^\eta} &\lesssim \sup_{\psi \in C^\infty: \|\psi\|_{h^{-\eta}} \le 1}|\langle \Delta \mathbb I_\theta^* \mathbb I_\theta \phi, \psi \rangle_{L^2}|^2 = \sup_{\psi \in C^\infty: \|\psi\|_{h^{-\eta}} \le 1}|\langle \mathbb I_\theta^* \mathbb I_\theta \phi, \Delta \psi \rangle_{L^2}|^2 \\
& = \sup_{\psi \in C^\infty: \|\psi\|_{h^{-\eta}} \le 1}\Big|\int_0^1 \langle \mathbb I_\theta \phi, \mathbb I_\theta \Delta \psi \rangle_{L^2(\Omega)}\Big|^2  \\
&  \lesssim  \sup_{\psi \in C^\infty: \|\psi\|_{h^{-\eta}} \le 1} \int_0^1 \|\mathbb I_\theta \phi\|^2_{H^{\eta+1}}dt \int_0^1 \|\mathbb I_\theta \Delta \psi\|_{H^{-\eta-1}}^2 dt \\
& \lesssim \sup_{\psi \in C^\infty: \|\psi\|_{h^{-\eta}} \le 1} \|\phi\|^2_{H^{\eta}} \|\Delta \psi\|_{H^{-\eta-2}}^2  \lesssim \|\phi\|^2_{H^\eta}.
\end{align*}
It is also clear from the divergence theorem (e.g., (\ref{ipart0})) that $\mathcal I$ maps into $H^\eta_0$.

Now to prove that the linear operator $\mathcal I$ is injective, let $\psi \in H^1_0$. If $\mathcal I \psi =  \Delta \mathbb I_\theta ^* \mathbb I_\theta \psi =0$ then $\mathbb I_\theta ^* \mathbb I_\theta \psi$ is constant a.e.~on $\Omega$ by the maximum principle for $\Delta$ (or by (\ref{seqnorm}), (\ref{weyl})). But then $$\int_0^1\|U_\psi(s)\|^2_{L^2(\Omega)}ds =  \|\mathbb I_\theta\psi \|^2_{L^2(\mathcal X)} = \langle \psi, \mathbb I_\theta^*\mathbb I_\theta  \psi \rangle_{L^2(\Omega)} = const \times \int_\Omega \psi =0$$ and this implies $\|\psi\|_{H^{-1}}=0$ in view of Lemma \ref{paristexas}, and then also $\psi =0$ almost everywhere and hence $\psi = 0$ in $H^1$.
\end{proof}

The following key result of this article verifies Condition \ref{gemol}F) for the reaction-diffusion system (\ref{evol}).

\begin{theorem}\label{raf}
Let $\eta> 2+d/2$. Then the operator $\mathcal I$ from (\ref{deltafish}) defines a continuous linear homeomorphism of $H^\eta_0(\Omega)$ onto itself. 
\end{theorem}
\begin{proof}
We will prove that $-\mathcal I$ equals half the identity $Id/2$ plus a compact operator $K$ mapping $H^\eta_0$ into $H^{\eta+1}_0$ for $\eta>2+d/2$. Since $\mathcal I$ is injective by Lemma \ref{injep0}, the Fredholm alternative, p.583 in \cite{TI}, implies that $\mathcal I$ is then also surjective onto $H^\eta_0$ and hence has an inverse $H^\eta_0 \to H^\eta_0$ which is continuous by the open mapping theorem.

\smallskip

We start with a comparison to the new operator 
\begin{equation}\label{deltafishy}
\mathcal I_\varepsilon = \Delta \mathbb I_{\theta, \varepsilon} ^* \mathbb I_{\theta, \varepsilon} 
\end{equation}
where $\mathbb I_{\theta, \varepsilon}$ is obtained as above but replacing $V=f'(u_{\theta}) \in C^{1, \infty}$ by the corresponding locally constant potential $V^{(\varepsilon)} \in C^{1,\infty}$ from Proposition \ref{cstpert}. Notice that $V^{(\varepsilon)}=f'(\theta)\equiv -W$ on $(0,\varepsilon/2)$. The operator $\mathcal I_\varepsilon$ also maps $H^\eta$ into $H^\eta_0$ (proved just as Lemma \ref{injep0}). 

\begin{lemma}\label{epsapp}
For the operator norm from $H^\eta \to H^{\eta+2}$ and all $0<\varepsilon \le 1$ we have $$\|\mathcal I_\varepsilon - \mathcal I\|_{H^\eta \to H^{\eta+2}}\lesssim \varepsilon,$$ in particular $\mathcal I_\varepsilon - \mathcal I$ is a compact operator on $H^\eta$.
\end{lemma}
\begin{proof}
By linearity and taking suprema over smooth test functions $\psi$,
\begin{align*}
\|\Delta \mathbb I_{\theta, \varepsilon} ^* \mathbb I_{\theta, \varepsilon} - \Delta\mathbb I_\theta^* \mathbb I_\theta\|_{h^\eta \to h^{\eta+2}}& = \sup_{\|\psi\|_{h^{-\eta-2}} \le 1} \sup_{\|\phi\|_{h^\eta \le 1}} |\langle (\mathbb I_{\theta, \varepsilon} ^* \mathbb I_{\theta, \varepsilon} - \mathbb I_\theta^* \mathbb I_\theta) \phi, \Delta \psi \rangle_{L^2}| \\
&\le \sup_{\|\psi\|_{h^{-\eta-2}} \le 1} \sup_{\|\phi\|_{h^\eta \le 1}} |\langle (\mathbb I_{\theta, \varepsilon} ^*\mathbb I_{\theta, \varepsilon} -  \mathbb I_{\theta, \varepsilon} ^*\mathbb I_\theta) \phi, \Delta \psi \rangle_{L^2}| \\ &~~+  \sup_{\|\psi\|_{h^{-\eta-2}} \le 1}\sup_{\|\phi\|_{h^\eta \le 1}} |\langle (\mathbb I_{\theta, \varepsilon}^*\mathbb I_\theta - \mathbb I_\theta^*\mathbb I_\theta)  \phi, \Delta \psi \rangle_{L^2}| \\
&= \sup_{\|\psi\|_{h^{-\eta-2}} \le 1} \sup_{\|\phi\|_{h^\eta \le 1}} |\langle (\mathbb I_{\theta, \varepsilon} -  \mathbb I_\theta) \phi, \mathbb I_{\theta, \varepsilon} \Delta \psi \rangle_{L^2(\mathcal X)}| \\ &~~+  \sup_{\|\psi\|_{h^{-\eta-2}} \le 1}\sup_{\|\phi\|_{h^\eta \le 1}} |\langle  \mathbb I_\theta \phi, (\mathbb I_\theta - \mathbb I_{\theta, \varepsilon}) \Delta \psi \rangle_{L^2(\mathcal X)}| 
\end{align*} 
By density we can restrict to smooth $\psi, \phi$ when bounding the preceding suprema. The term inside the supremum of the first summand in the last equation is upper bounded, using the definition of the $h^\eta$ norms, the Cauchy-Schwarz inequality, the second part of Proposition \ref{cstpert} as well as (\ref{inhom}) with $m=0$, by 
\begin{align*}
&\Big(\int_0^1 \|U_\phi(t) - U_\phi^{(\varepsilon)}(t)\|^2_{h^{\eta+3}}dt\Big)^{1/2} \Big(\int_0^1 \|U^{(\varepsilon)}_{\Delta \psi}(t)\|^2_{h^{-\eta-3}}dt\Big)^{1/2} \lesssim  \varepsilon \|\phi\|_{h^\eta} \|\Delta \psi\|_{h^{-\eta-4}}
\end{align*}
By similar arguments the second term is bounded by
$$\Big(\int_0^1 \|U_{\Delta\psi}(t) - U_{\Delta \psi}^{(\varepsilon)}(t)\|^2_{h^{-\eta-1}}dt \Big)^{1/2} \Big(\int_0^1\|U_{\phi}(t)\|^2_{h^{\eta+1}}dt\Big)^{1/2} \lesssim \varepsilon  \|\phi\|_{h^\eta} \|\Delta \psi\|_{h^{-\eta-4}},$$ so the result follows since $\|\Delta \psi\|_{h^{-\eta-4}} \lesssim \|\psi\|_{h^{-\eta-2}} \le 1$.
\end{proof}

Writing $\mathcal I = \mathcal I_{\varepsilon} + \mathcal I - \mathcal I_{\varepsilon}$ we conclude that it suffices to prove that $-\mathcal I_\varepsilon$ equals half the identity operator plus a compact perturbation. To achieve this we first derive a more explicit representation of the information operator $\mathbb I_{\theta, \varepsilon}^*\mathbb I_{\theta, \varepsilon}$. For $h \in H^\eta$ we can use (\ref{ubahr}) and write the solutions of (\ref{weakpdelin}) for $m=0$ and potential $W=-f'(\theta)$ as 
\begin{equation}\label{greenkern}
\mathbb I_{\theta, \varepsilon}(h) (t, \cdot) = U^{(\varepsilon)}_h(t) = \int_\Omega p_{t,W}(\cdot,y) h(y)dy,~~0<t<\varepsilon/2.
\end{equation}
[As the Fredholm alternative only requires the perturbation to be compact and not small, we do in fact not require $\varepsilon \to 0$ in what follows,   see also Remark \ref{clever}.]

For any $t>0$, Proposition \ref{roundrobinlin} implies that the linear operators $h \mapsto U_h^{(\varepsilon)}(t, \cdot)$ for $t$ fixed are continuous from $L^2(\Omega) \to L^2(\Omega)$ with operator norms uniformly bounded in $t$. Therefore they have continuous linear adjoint operators $$U_{t,\varepsilon}^*: L^2(\Omega) \to L^2(\Omega)$$ with operator norms bounded also uniformly in $0\le t \le 1$. Thus for any $h_1 \in L^2, h_2 \in H^\eta$ we have, using also Parseval's identity for the  basis $\{e_{j,W}\}$, and Fubini's theorem,
\begin{align}\label{infoopgreen}
&\langle h_1, \mathbb I_{\theta, \varepsilon}^*\mathbb I_{\theta, \varepsilon} h_2\rangle_{L^2(\Omega)}\\
&=\langle \mathbb I_{\theta, \varepsilon} h_1, \mathbb I_{\theta, \varepsilon} h_2 \rangle_{L^2([0,1],L^2(\Omega))} =\int_0^1 \langle U^{(\varepsilon)}_{h_1}(t), U^{(\varepsilon)}_{h_2}(t) \rangle_{L^2(\Omega)}dt \notag \\
&=\int_0^{\varepsilon/4} \sum_{j} e^{-2t\lambda_{j,W}} \langle e_{j,W}, h_1 \rangle_{L^2} \langle e_{j,W}, h_2 \rangle_{L^2} dt +  \Big\langle h_1,   \int_{\varepsilon/4}^1 U_{t, \varepsilon}^* [U_{h_2}^{(\varepsilon)}(t)] dt\Big\rangle_{L^2}. \notag
\end{align}
For the first term we use again (\ref{ubahr}) and $2t \le \varepsilon/2$ to recognise
$$U^{(\varepsilon)}_{h_2}(2t,\cdot) = \sum_{j}e^{-2t\lambda_{j,W}}\langle e_{j,W}, h_2 \rangle_{L^2} e_{j,W}$$ and this series converges uniformly on $[0,\varepsilon/4] \times \Omega$ in view of (\ref{uniflim}), $h_2 \in H^\eta \subset H^2$ and (\ref{tildeev}). Therefore by the dominated convergence and Fubini's theorem
\begin{align*}
& \int_0^{\varepsilon/4} \sum_j e^{-2t\lambda_{j,W}} \langle e_{j,W}, h_2 \rangle_{L^2}  \int_\Omega e_{j,W}(y) h_1(y) dy = \int_\Omega h_1(y) \int_0^{\varepsilon/4} U^{(\varepsilon)}_{h_2}(2t,y) dtdy.
\end{align*}
and hence we can write (\ref{infoopgreen}), for all $h_1 \in L^2, h_2 \in H^\eta$, as
 $$\langle h_1, \mathbb I_{\theta, \varepsilon}^*\mathbb I_{\theta, \varepsilon} h_2\rangle_{L^2(\Omega)} = \Big\langle h_1,  \int_0^{\varepsilon/4} U^{(\varepsilon)}_{h_2}(2t,\cdot) dt + \int_{\varepsilon/4}^1  U_{t, \varepsilon}^* [U_{h_2}^{(\varepsilon)}(t)]dt  \Big\rangle_{L^2}.$$
In summary, the action of the operator $\mathcal I_\varepsilon= \Delta\mathbb I_{\theta, \varepsilon}^* \mathbb I_{\theta, \varepsilon}$ on $H^\eta$ can be represented as
\begin{align}\label{keyrep}
\mathcal I_{\varepsilon} h & =\Delta \int_0^{\varepsilon/4} U^{(\varepsilon)}_h(2t,\cdot) dt + \Delta \int_{\varepsilon/4}^1  U_{t, \varepsilon}^* [U_{h}^{(\varepsilon)}(t)] dt = A(h)+B(h).
\end{align}
For the second term we use (\ref{dual}), Fubini's theorem, Lemma \ref{bootcamp} and Proposition \ref{roundrobinlin} in
\begin{align*}
\|B(h)\|_{H^{\eta+1}} &= \sup_{\psi \in C^\infty: \|\psi\|_{H^{-\eta-1}}\le 1} \Big|\big\langle \psi, \Delta\int_{\varepsilon/4}^1  U_{t, \varepsilon}^* [U_h^{(\varepsilon)}(t)] dt \big\rangle_{L^2}\Big| \\
& =  \sup_{\psi \in C^\infty: \|\psi\|_{H^{-\eta-1}}\le 1}\Big |\int_{\varepsilon/4}^1 \langle \Delta \psi, U_{t, \varepsilon}^* [U_h^{(\varepsilon)}(t)]  \rangle_{L^2} dt\Big | \\
& =  \sup_{\psi \in C^\infty: \|\psi\|_{H^{-\eta-1}}\le 1}\Big|\int_{\varepsilon/4}^1 \langle U_{\Delta \psi}^{(\varepsilon)}(t), U_h^{(\varepsilon)}(t)  \rangle_{L^2} dt \Big|\\
&\lesssim \sup_{t \in [\varepsilon/4,1], \psi \in C^\infty: \|\psi\|_{H^{-\eta-3}}\le c} \|U_{\psi}(t)\|_{H^{-\eta}} \|U_h(t)\|_{H^{\eta}} \lesssim \|h\|_{H^\eta},
\end{align*}
for some $c>0$, so that we conclude that $B$ is a compact operator on $H^\eta$.

About term $A$: For $h \in H^\eta, \eta>2+d/2,$ the functions $\Delta U^{(\varepsilon)}_h$ are uniformly bounded on $[0,\varepsilon/2] \times \Omega$ by Proposition \ref{roundrobinlin} and the Sobolev imbedding, and so we can take $\Delta$ inside of the integral, and likewise $U^{(\varepsilon)}_h(\cdot,x)$ is absolutely continuous on $[0,\varepsilon/2]$ for every $x \in \Omega$. Using (\ref{PDElinW}) with our $W=-f'(\theta) \in C^\infty$ and the fundamental theorem of calculus for Lebesgue integrals (p.106 in \cite{Fo99}), we thereupon obtain
\begin{align*}
A(h) &=  \frac{1}{2}\int_0^{\varepsilon/2} \Delta U^{(\varepsilon)}_h(t,\cdot) dt = \frac{1}{2}\Big[\int_0^{\varepsilon/2} \frac{\partial}{\partial t}  U^{(\varepsilon)}_h(t,\cdot) dt +\int_0^{{\varepsilon/2}}  W U^{(\varepsilon)}_h(t,\cdot) dt \Big]  \\
&=-\frac{1}{2}Id(h) +  \frac{1}{2}U^{(\varepsilon)}_h(\varepsilon/2) + \int_0^{\varepsilon/4}  W U^{(\varepsilon)}_h(2t,\cdot) dt
\end{align*}
where $Id$ is the identity operator. By Proposition \ref{fwdsmooth} the second term maps $h \in H^\eta$ linearly and continuously into $H^{\eta+1}$ for every fixed $\varepsilon>0$. The same is true for the third term since its $H^{\eta+1}$-norms can be bounded, using the Cauchy-Schwarz and Minkowski's integral inequality, (\ref{multbasic}) as well as Proposition \ref{roundrobinlin} by
$$c \|W\|_{H^{\eta+1}} \Big(\int_0^{1}\|U_h(t)\|_{H^{\eta+1}}^2dt\Big)^{1/2} \lesssim \|h\|_{H^\eta}$$ for some $c=c(\eta,T)<\infty$. Hence these terms are also compact linear operators on $H^\eta$. 

In summary, we have shown that $-\mathcal I = \frac{Id}{2} +K$ on $H^\eta$ for a compact operator $K: H^\eta \to H^\eta$. To show that $K$ in fact maps $H^\eta_0$ compactly into $H^\eta_0$, let $h_n$ be a bounded sequence in $H^\eta_0$. Then $\int_\Omega \mathcal I h_n =0$ for all $n$ by Lemma \ref{injep0}, hence $$\int_\Omega K(h_n) = -\int_\Omega \mathcal I h_n - \frac{1}{2}\int_\Omega Id(h_n) =0 -0 =0,$$ so $K(h_n) \in H^\eta_0$, and by compactness $K(h_n)$ has a $\|\cdot\|_{H^\eta}$-convergent subsequence. As $H^\eta_0$ is a closed subspace of $H^\eta$ we see that $K$ maps compactly into $H^\eta_0$.
\end{proof}

\begin{remark}\label{clever}\normalfont
Unlike in the proof of the injectivity result Lemma \ref{injep0} (via Lemma \ref{paristexas}), the proof of surjectivity of the operator $\mathcal I$ does not require Proposition \ref{cstpert} for $\varepsilon$ sufficiently \textit{small}, since \textit{any} smooth bounded perturbation of the potential $V$ results in a compact perturbation of $\mathcal I$ via Lemma \ref{epsapp}. But for $\varepsilon \to 0$ our proof provides some further qualitative understanding of the structure of the information operator, which decomposes as
\begin{equation}\label{infoapprox}
\mathbb I_\theta^* \mathbb I_\theta = \mathbb I_{\theta, \varepsilon}^* \mathbb I_{\theta, \varepsilon}  + \varepsilon K_{2,\epsilon} = -\frac{1}{2}\Delta^{-1} +K_{1,\varepsilon} +\varepsilon  K_{2,\epsilon}  ~~\text{any } 0<\varepsilon \le 1,
\end{equation}
 where $K_{2,\epsilon}$ is smoothing of order $3$ and $K_{1,\varepsilon}$ is $4$- smoothing for any fixed $\varepsilon$. The operator $\mathbb I_{\theta, \varepsilon}^* \mathbb I_{\theta, \varepsilon}$ corresponds to the case where the infinitesimal behaviour of the reaction-diffusion system at the beginning of time $t\in(0,\varepsilon/2)$ is described by equation (\ref{ubahr}) with Schr\"odinger operator $\mathscr S=\Delta + f'(\theta)$. On the eigen-spaces of $\mathscr S$ with negative eigenvalues $-\lambda_{j,W}<0$, the dynamics instantly smooth the initial condition, but on the null space of $\mathscr S$ the system `stalls' and the linearised dynamics $\mathbb I_{\theta, \varepsilon}$ equal the identity operator on the `frozen' interval $[0,\varepsilon/2]$, reproducing a `direct' regression model where non-existence results \cite{F99} for Bernstein-von Mises theorems in infinite dimensions would apply. A deeper reason for the possibility of Theorem \ref{ganzwien} holding in the strong topology of $\mathscr C$ is the fact that the kernel of $\mathscr S$ is at most finite-dimensional (Lemma \ref{hooray}), and that $\mathbb I_\theta^* \mathbb I_\theta$ is, in view of (\ref{infoapprox}), approximated with arbitrary precision as $\varepsilon \to 0$ by the information operator of a system with this property. This remark applies equally to the (in view of (\ref{tildeev}) and (\ref{weyl})) finite-dimensional eigenspaces corresponding to positive eigenvalues $-\lambda_{j,W}>0$.
\end{remark}

\textbf{Acknowledgement.} The author thanks Dimitri Konen for pointing out various minor but important corrections to the proofs in the first version, as well as Aur\'elien Castre and an anonymous referee for helpful remarks and suggestions. RN further gratefully acknowledges support through an ERC Advanced Grant (UKRI G116786) as well as by EPSRC programme grant EP/V026259.

\bibliography{react}{}
\bibliographystyle{plain}

\bigskip

\textsc{Department of Pure Mathematics \& Mathematical Statistics}

\textsc{University of Cambridge}, Cambridge, UK

Email: nickl@maths.cam.ac.uk

\end{document}